\documentclass[a4paper,11pt]{amsart}
\usepackage{amsmath}
\usepackage{amssymb}
\usepackage{amsthm}
\usepackage[english]{babel}
\usepackage{hyperref}
\usepackage{mathrsfs}
\usepackage{centernot}

\usepackage{enumerate}

\newtheorem{theor}{Theorem}[section]
\newtheorem{lemma}[theor]{Lemma}

\newcommand{\hd}{\mathrm{hd}}
\newcommand{\soc}{\mathrm{soc}}

\newcommand{\Hom}{\mathrm{Hom}}
\newcommand{\End}{\mathrm{End}}

\newcommand{\cont}{\mathrm{cont}}

\newcommand{\res}{\mathrm{res}}
\newcommand{\diag}{\mathrm{diag}}

\newcommand{\s}{{\sf S}}
\newcommand{\A}{{\sf A}}

\newcommand{\md}{\mbox{-}\mathrm{mod}}
\def\Md#1{\text{ }(\text{\rm mod } #1)\,}

\newcommand{\Z}{\mathbb{Z}}
\newcommand{\N}{\mathbb{N}}
\newcommand{\Q}{\mathbb{Q}}

\newcommand{\1}{\mathbf{1}}
\newcommand{\sgn}{\mathbf{\mathrm{sgn}}}

\newcommand{\dbl}{\mathrm{dbl}}
\newcommand{\dblb}{\overline{\mathrm{dbl}}}
\renewcommand{\epsilon}{\varepsilon}
\newcommand{\eps}{\epsilon}
\renewcommand{\phi}{\varphi}
\newcommand{\xymat}{\xymatrix@R=6pt@C=10pt}

\newcommand{\la}{\lambda}
\newcommand{\be}{\beta}
\newcommand{\al}{\alpha}
\newcommand{\ga}{\gamma}
\newcommand{\de}{\delta}
\newcommand{\ze}{\zeta}
\newcommand{\si}{\sigma}

\newcommand{\da}{{\downarrow}}
\newcommand{\ua}{{\uparrow}}

\def\Par{{\mathscr P}}
\def\Parinv{{\mathscr P}^A}
\def\Paro{{\mathscr P}_{\text{odd}}}
\def\Parod{{\mathscr P}_{2,\text{odd}}}

\renewcommand{\nmid}{\centernot\mid}

\begin{document}

\title[Tensor products for alternating groups]{Irreducible tensor products for alternating groups in characteristic 2}

\author{\sc Lucia Morotti}
\address
{Institut f\"{u}r Algebra, Zahlentheorie und Diskrete Mathematik\\ Leibniz Universit\"{a}t Hannover\\ 30167 Hannover\\ Germany} 
\email{morotti@math.uni-hannover.de}

\thanks{The author was supported by the DFG grant MO 3377/1-2.}


\begin{abstract}
In this paper we completely characterise irreducible tensor products of a basic spin module with an irreducible module for alternating groups in characteristic 2.  This completes the classification of irreducible tensor products of representations of alternating groups.
\end{abstract}

\maketitle

\section{Introduction}

Let $F$ be an algebraically closed field of characteristic $p$ and $G$ be a group. Given two irreducible $FG$-representations $V$ and $W$ one may ask if the tensor product $V\otimes W$ is also irreducible. This is always the case if $V$ or $W$ is 1-dimensional. On the other hand if neither $V$ nor $W$ is 1-dimensional then in general $V\otimes W$ is not irreducible, but there are cases where this happens. Irreducible tensor products $V\otimes W$ where neither $V$ nor $W$ is 1-dimensional are called non-trivial. 

For $G$ an alternating group the question of completely classifying irreducible tensor products of representations has been answered in \cite{bk3} in characteristic 0 (and implicitly in \cite{z1}), showing that there is only one family of non-trivial tensor products. This family of non-trivial irreducible tensor products can be generalised in a natural way to give a family of non-trivial irreducible tensor products in arbitrary characteristic, see \cite{bk2,m2,m4}. In characteristic at least 3 it has also been proved that non-trivial irreducible tensor products belong to this family (apart for one exceptional case in characteristic 3). In characteristic 2 however it had remained open to study which tensor products belonging to a further family are irreducible. In this paper we study this family of tensor products and show that, apart for a single case, none of them is irreducible.

Before being able to state the main result of this paper we need to introduce some notation. Let $p$ be a prime and $\Par_p(n)$ be the set of $p$-regular partitions of $n$ (that is, partitions were no part is repeated $p$ or more times). It is well known (see for example \cite{JamesBook}) that, up to isomorphism, irreducible representations of the symmetric group $\s_n$ in characteristic $p$ are labelled by $\Par_p(n)$. So, given $\la\in\Par_p(n)$, we let $D^\la$ be the corresponding irreducible representation. By \cite{Benson}, when $D^\la$ is restricted to the alternating group $\A_n$, two things can happen: either $D^\la\da_{\A_n}\cong E^\la$ is irreducible or $D^\la\da_{\A_n}\cong E^\la_+\oplus E^\la_-$ is the direct sum of two non-isomorphic irreducible representations. We say that a partition $\la\in\Par_p(n)$ is a JS-partition if the restriction $D^\la\da_{\s_{n-1}}$ is irreducible.

Reductions modulo 2 of spin representations of the covering groups $\widetilde{\s}_n$ or $\widetilde{\A}_n$ can also be viewed as representations of $\s_n$ or $\A_n$. With this identification we say that, in characteristic 2, an irreducible representation of $\s_n$ or $\A_n$ is a basic spin representation if it is a composition factor of the reduction modulo 2 of a basic spin representation of $\widetilde{\s}_n$ or $\widetilde{\A}_n$, see \cite{Benson,Wales}.

In this paper we will prove the following theorem:

\begin{theor}\label{MT}
Let $p=2$ and let $V$ and $W$ be irreducible $FA_n$-modules of dimension larger than 1. If $W$ is basic spin then $V\otimes W$ is irreducible if and only if $n=5$ and, up to exchange of $V$ and $W$, $V\cong E^{(3,2)}_+$ and $W\cong E^{(3,2)}_-$, in which case $V\otimes W\cong E^{(4,1)}$.
\end{theor}


Together with \cite{bk3,bk2,m2,m4,z1} this gives the following characterisation of non-trivial irreducible tensor products of representations of alternating groups in arbitrary characteristic (see Section \ref{s11} for details):

\begin{theor}\label{T150620}
Let $V$ and $W$ be irreducible $FA_n$-modules of dimension larger than 1. Then $V\otimes W$ is irreducible if and only if one of the following holds, up to exchanging $V$ and $W$:
\begin{enumerate}
\item $p\nmid n$, $V\cong E^\la_\pm$ where $\la$ is a JS-partition and $W\cong E^{(n-1,1)}$, in which case $V\otimes W\cong E^{(\la\setminus A)\cup B}$, where $A$ is the top removable node of $\la$ and $B$ is the second-lowest addable node of $\la$,

\item $p=3$, $n=6$, $V\cong E^{(4,1^2)}_+$ and $W\cong E^{(4,1^2)}_-$, in which case $V\otimes W\cong E^{(4,2)}$,

\item $p=2$, $n=5$, $V\cong E^{(3,2)}_+$ and $W\cong E^{(3,2)}_-$, in which case $V\otimes W\cong E^{(4,1)}$.
\end{enumerate}
\end{theor}

In view of \cite{bk,gj,gk,m1,z1} for symmetric groups and \cite{b2,bk4,kt,m3} for double covers of symmetric and alternating groups, this concludes the problem of describing non-trivial tensor products for symmetric and alternating groups as well as their covering groups in arbitrary characteristic (for the exceptional covering groups for $n=6$ or 7 non-trivial irreducible tensor products can be determined using character tables).

In Section \ref{s2} we will introduce some notation and give references for some basic results. In Section \ref{s3} we will consider branching of (irreducible) modules. In Section \ref{s4} we will present some results connected to the reduction modulo 2 of spin representations, and in Section \ref{s5} we will study some properties of their Brauer characters. Using these results we will then in Section \ref{s6} present reduction results which will allow us to greatly reduce the cases to be considered to prove the main result of this paper. In Section \ref{s7} we will study certain permutation modules and use their structure to study the submodule structure of certain endomorphism modules. In Sections \ref{s8} and \ref{s9} we will then consider certain tensor products in the double split or split-non-split cases respectively. This will allow us to prove Theorem \ref{MT} in Section \ref{s10}. Using Theorem \ref{MT} and the aforementioned papers we then prove Theorem \ref{T150620} in Section \ref{s11}.

\section{Notation and basic results}\label{s2}

Let $p$ be a prime (in fact we will take $p=2$ throughout the paper except in Section \ref{sbr}). For a non-negative integer $n$, let $\Par(n)$ be the set of partitions of $n$, let $\Par_p(n)$ be the set of $p$-regular partitions of $n$, let $\Paro(n)$ be the set of partitions of $n$ consisting only of odd parts, and let $\Parod(n)$ be the set of partitions of $n$ in odd distinct parts.  Note that $\Parod(n)=\Par_2(n)\cap\Paro(n)$. For any partition $\la$ let $h(\la)$ be the number of non-zero parts of $\la$ and $h_2(\la)$ the number of non-zero even parts of $\la$. In addition, let $|\la|$ be the sum of the parts of $\la$. If $\la$ is a partition and $A$ is a removable node of $\la$ then we set $\la_A:=\la\setminus\{A\}$ (identifying a partition and its Young diagram). Similarly if $B$ is an addable node of $\la$ then $\la^B:=\la\cup\{B\}$.

Sometimes we will define partitions by defining their (multi)set of parts (this will be useful to avoid splitting up certain sums depending on the different relations between the lengths of their parts). Thus, given a multiset $\mu$, we define a partition $p(\mu)$ by just reordering the parts of $\mu$. Furthermore, if $\la$ is any partition and $j\geq 1$ let $\hat\la_j:=(\la_1\ldots,\la_{j-1},\la_{j+1},\ldots)$ be the partition obtained by removing the $j$-th part of $\la$.

We will now give an overview of basic results of representation theory of symmetric and alternating groups. 
We refer the reader to \cite{JamesBook,JK} for more information and proofs of the statements. It is well known that $\Par(n)$ indexes the irreducible $\s_n$-representations in characteristic $0$. For $\la\in\Par(n)$ let $S^\la$ denote the Specht module indexed by $\la$ and $M^\la$ be the permutation module $\1\ua_{\s_\la}^{\s_n}$ (with $\1$ the trivial representation and $\s_\la\cong\s_{\la_1}\times\s_{\la_2}\times\ldots$ the corresponding Young subgroup), and identify these modules with their reduction modulo $p$. As in the introduction, given $\la\in\Par_p(n)$, let $D^\la$ be the corresponding irreducible module of $\s_n$ in characteristic $p$. By \cite[\S11]{JamesBook}, $D^\la$ is the head of $S^\la$ and $S^\la$ has no other composition factor isomorphic to $D^\la$. Furthermore $\Paro(n)$ indexes the $2$-regular conjugacy classes of $\s_n$ and, for $n\geq 2$, $\Parod(n)$ indexes the conjugacy classes of $\s_n$ which consist of elements of $\A_n$ and which split in two conjugacy classes of $\A_n$. In order to keep notation shorter when considering certain small permutation modules, if $\mu\in\Par(m)$ with $m<n$ and $(n-m,\mu)\in\Par(n)$, we define $S_\mu:=S^{(n-m,\mu)}$ and similarly for $M_\mu$ and $D_\mu$ (provided $(n-m,\mu)\in\Par_p(n)$, in the case of $D_\mu$).

For $p=2$, we define $\Parinv_2(n)\subseteq\Par_2(n)$ to be the set of partitions for which the module $D^\la\da_{\A_n}$ splits. It was proved in \cite[Theorem 1.1]{Benson} that $\Parinv_2(n)$ is given by the set
\[\{\la\in\Par_2(n)|\la_{2k-1}-\la_{2k}\leq 2\text{ and }\la_{2k-1}+\la_{2k}\not\equiv 2\Md{4}\!\text{ for }k\geq 0\}.\]
As in the introduction we then define $E^\la$ (for $\la\in\Par_2(n)\setminus\Parinv_2(n)$) and $E^\la_\pm$ (for $\la\in\Parinv_2(n)$) to be the irreducible representations of $\A_n$ in characteristic $2$ (we will not need irreducible representations of alternating groups in odd characteristic). In the following, when writing $E^\la_{(\pm)}$, we will mean either $E^\la$ if $\la\in\Par_p(n)\setminus\Parinv_p(n)$ or $E^\la_\pm$ if $\la\in\Parinv_p(n)$. 

We will also use spin representations (and their reductions modulo $2$). Spin representations of symmetric groups have been studied in \cite{s5}, see also \cite{s}. Spin representations of symmetric groups (or pairs of them) are indexed by partitions into distinct parts ($2$-regular partitions). Let $\widetilde{\s}_n$ be a proper double cover of $\s_n$. Then there is $z\in\widetilde{\s}_n$ which is central and of order 2 such that $\widetilde{\s}_n/\langle z\rangle\cong\s_n$. Irreducible spin representations of $\s_n$ are the irreducible representations of $\widetilde{\s}_n$ on which $z$ does not act trivially. These representations of $\widetilde{\s}_n$ cannot be viewed as representations of $\s_n$. In addition, let $\widetilde{\A}_n\subseteq\widetilde{\s}_n$ be the double cover of $\A_n$, and let $\la\in\Par_2(n)$. If $n-h(\la)$ is even then, in characteristic 0, there exists a unique spin representation of $\s_n$ indexed by $\la$. We denote this representation by $S(\la,0)$. It is known (see \cite[Corollary 7.5]{s}) that $S(\la,0)\da_{\widetilde{\A}_n}$ splits as direct sum of two non-isomorphic irreducible representations. If $n-h(\la)$ is odd then, in characteristic 0, there exist two non-isomorphic spin representation of $\s_n$ indexed by $\la$. We denote these representations by $S(\la,\pm)$. It is known (see again \cite[Corollary 7.5]{s}) that $S(\la,+)\da_{\widetilde{\A}_n}\cong S(\la,-)\da_{\widetilde{\A}_n}$ is irreducible.

Reductions modulo 2 of the representations $S(\la,\eps)$ may be viewed as $2$-modular representations of $\s_n$ (with $\eps=0$ or $\pm$ depending on $\la$), since $z$ acts as $\pm 1$ on irreducible representations (being central of order 2). We identify $S(\la,\eps)$ with its reduction modulo $2$ and view it in this way as a $\s_n$-representation. It is known that $S(\la,\eps)\cong S(\la,-\eps)$ when reduced modulo $2$. See \cite{Benson,bo2} for more information. We can therefore define $S(\la)$ to be the reduction modulo 2 of $S(\la,\eps)$. This module is well defined thanks to the previous remark. 

For $m>0$ an integer let $\dbl(m):=(\lceil(m+1)/2\rceil,\lfloor(m-1)/2\rfloor)$ and $\dblb(m):=(\lceil m/2\rceil,\lfloor m/2\rfloor)$. For any partition $\la$ define its double $\dbl(\la):=(\lceil(\la_1+1)/2\rceil,\lfloor(\la_1-1)/2\rfloor,\ldots,\lceil(\la_{h(\la)}+1)/2\rceil,\lfloor(\la_{h(\la)}-1)/2\rfloor)$ to be the concatenation of $\dbl(\la_1),\ldots,\dbl(\la_{h(\la)})$, and similarly define $\dblb(\la)$ to be the concatenation of $\dblb(\la_1),\ldots,\dblb(\la_{h(\la)})$. It can be easily checked that if $\la\in\Par_2(n)$ then $\dblb(\la)\in\Par(n)$.

Further if $\la\in\Par(n)$ define its regularisation $\la^R$ as in \cite[6.3.48]{JK} (for $p=2$). For example $\dbl(m)=(\dblb(m))^R$ for any $m\geq 1$ and if $\dbl(\la)\in\Par_2(n)$ then $\dbl(\la)=(\dblb(\la))^R$. 

In characteristic 2, by \cite[Theorem 6.3.50]{JK} $S^\la$ has exactly one composition factor isomorphic to $D^{\la^R}$ and all of its other composition factors are of the form $D^\mu$ with $\mu\rhd\la^R$ (where $\unrhd$ dentoes the dominance order). Similarly it has been proved in \cite[Theorem 5.1]{bo2} that $S(\la)$ has exactly $2^{\lfloor h_2(\la)/2\rfloor}$ composition factors isomorphic to $D^{\dblb(\la)^R}$ and all other composition factors are of the form $D^\mu$ with $\mu\rhd (\dblb(\la))^R$ (see also \cite[Theorem 1.2]{Benson} for the case where $\dbl(\la)\in\Par_2(n)$).

In characteristic 0, basic spin representations are the irreducible spin representations labelled by the partition $(n)$. Second basic spin representations are spin representations labelled by the partition $(n-1,1)$. By \cite[Tables III]{Wales} (and decomposition matrices/characters for small $n$) and the previous paragraph, it is known that all composition factors of the reduction modulo 2 of $S((n))$ are labelled by the partition $\dbl(n)$. Since these modules will play a crucial role in this paper we let $\be_n:=\dbl(n)$. Thus, in characteristic 2, basic spin representations of $\s_n$ or $\A_n$ are just the irreducible representations labelled by $\be_n$. Further, for $n\geq 6$, we have from \cite[Tables IV]{Wales} and \cite[Theorem 1.2]{Benson} that any composition factor of the reduction modulo 2 of $S((n-1,1))$ is either a basic spin module or is indexed by the partition $\dbl(n-1,1)$. We call any characteristic 2 irreducible representation of $\s_n$ or $\A_n$ indexed by $\dbl(n-1,1)$ a second basic spin module.

For $\la\in\Par_2(n)$ let $\psi^\la$ be the Brauer character of the irreducible characteristic 2 module $D^\la$. If $\la\not\in\Parinv_2(n)$ then $\psi^\la$ is also the Brauer character of $E^\la$. If $\la\in\Parinv_2(n)$ let $\psi^{\la,\pm}$ be the Brauer characters of $E^\la_\pm$. In addition, again for $\la\in\Par_2(n)$, let $\ze^\la$ be the Brauer character of $S(\la)$.

For $\al\in\Paro(n)$ let $C_\al$ be the $\s_n$-conjugacy class with cycle type $\al$. For $\psi$ a Brauer character of $\s_n$, let $\psi_\al$ be the value that $\psi$ takes on $C_\al$. If $\al\not\in\Parod(n)$ then $C_\al$ is also an $\A_n$-conjugacy class. In this case if $\pi$ is a Brauer character of $\A_n$ we let $\pi_\al$ be the value that $\pi$ takes on the conjugacy class $C_\al$. Note that if $\pi=\psi\da^{\s_n}_{\A_n}$ then $\pi_{\al}=\psi_\al$, so this notation does not cause a problem when identifying Brauer characters of $\s_n$ with their restrictions to $\A_n$ (in particular $\psi^\la_\al$ is well defined when $\la\in\Par_2(n)\setminus\Parinv_2(n)$). If instead $\al\in\Parod(n)$ then $C_\al$ splits into two $\A_n$-conjugacy classes $C_{\al,+}$ and $C_{\al,-}$. In this case we write $\pi_{\al,\pm}$ for the value $\pi$ takes on $C_{\al,\pm}$. Note that if again $\pi=\psi\da^{\s_n}_{\A_n}$ then $\pi_{\al,\pm}=\psi_\al$.

For a given module $W$ we write $W\sim V_1|\ldots|V_k$ if $M$ has a filtration $0=W_0\subseteq W_1\subseteq\ldots\subseteq W_k=W$ with $W_j/W_{j-1}\cong V_j$ for $1\leq j\leq k$. If $W$ is uniserial and has composition factors $C_1,\ldots,C_k$ listed from socle to head, then we also write $W\cong C_1|\ldots|C_k$.

Further for any groups $G$ and $H$ and any representations $A$ of $G$ and $B$ of $H$ define $A\boxtimes B$ to be the outer tensor product of $A$ and $B$.

\section{Branching results}\label{sbr}\label{s3}

In this section we will work in arbitrary characteristic $p$, in order to be able to state Lemmas \ref{lspechtfiltration}, \ref{L110820_2} and \ref{L090620} in general form. So let $p$ be any prime and $I:=\Z/p\Z$. We call elements of $I$ residues. In the following we will identify a residue with any of its representatives. Given any node $(a,b)$, we define its residue to be $\res(a,b):=b-a\Md{p}$. If $\al\in\Par(n)$ is a partition then its content is given by $\cont(\al):=(a_j)_{j\in I}$ where, for each residue $i$, $a_i$ is the number of nodes of $\al$ of residue $i$. It is well known (see for example \cite[2.7.41 and 6.1.21]{JK}) that if $\al,\be\in\Par(n)$ then $S^\al$ and $S^\be$ are in the same block if and only if the $p$-cores of $\al$ and $\be$ are equal, which happens if and only if $\cont(\al)=\cont(\be)$. 

Given any tuple $C=(c_j)_{j\in I}$ define $B_C$ as follows: if there exists a partition $\ga$ with $\cont(\ga)=C$ then $B_C$ is the block of $S^\ga$, while if no such partition $\ga$ exists then $B_C:=0$. In view of the above remark $B_C$ is well defined.

If $V$ is any $\s_n$-module contained in the block $B_{(c_j)}$ and $i\in I$ is a fixed residue, define $e_iV$ to be the projection of $V\da_{\s_{n-1}}$ to the block $B_{(c_j-\de_{i,j})}$, and define $f_iV$ to be the projection of $V\ua^{\s_{n+1}}$ to the block $B_{(c_j+\de_{i,j})}$, see \cite[\S11.2]{KBook}. Here $\de_{i,j}$ is the Kronecker symbol. The definitions of $e_iV$ and $f_iV$ can then be extended additively to arbitrary $S_n$-modules. This gives functors
\[e_i:F\s_n\md\to F\s_{n-1}\md\hspace{22pt}\text{and}\hspace{22pt}f_i:F\s_n\md\to F\s_{n+1}\md.\]
More generally for any $r\geq 1$ let
\[e_i^{(r)}:F\s_n\md\to F\s_{n-r}\md\hspace{22pt}\text{and}\hspace{22pt}f_i^{(r)}:F\s_n\md\to F\s_{n+r}\md\]
be the divided power functors defined in \cite[\S11.2]{KBook}. Note that $e_i^{(1)}=e_i$ and $f_i^{(1)}=f_i$. For $r=0$ define $e_i^{(0)}V$ and $f_i^{(0)}V$ to be equal to $V$ for any $\s_n$-module $V$. Then the following holds by \cite[Lemma 8.2.2(ii), Theorems 8.3.2(i), 11.2.7, 11.2.8]{KBook} for $r\geq 1$ and trivially for $r=0$.

\begin{lemma}\label{lef}
For any $i\in I$ and $r\geq 0$ the functors $e_i^{(r)}$ and $f_i^{(r)}$ are biadjoint and commute with duality. Furthermore, for any $\s_n$-module $V$
\[V\da_{\s_{n-1}}\cong e_0V\oplus \ldots\oplus e_{p-1}V\hspace{22pt}\text{and}\hspace{22pt}V\ua^{\s_{n+1}}\cong f_0V\oplus\ldots f_{p-1}V.\]
\end{lemma}

In characteristic 2 if two modules $E^\la_{(\pm)}$ and $E^\mu_{(\pm)}$ are in the same block then so are $S^\la$ and $S^\mu$, see \cite[6.1.46]{JK}. So if $E^\la_{(\pm)}$ and $E^\mu_{(\pm)}$ are in the same block then $\la$ and $\mu$ have the same content. We can then define $e_i$ and $f_i$ also for $\A_n$-modules (again defined blockwise): if $V$ is an $\A_n$-module contained in the block(s) with content $(c_j)$ then, similarly to the $\s_n$-case, define $e_iV$ to be the projection of $V\da_{\A_{n-1}}$ to the block(s) with content $(c_j-\de_{i,j})$, and define $f_iV$ to be the projection of $V\ua^{\A_{n+1}}$ to the block(s) with content $(c_j+\de_{i,j})$. By definition we have that
\[e_i:F\A_n\md\to F\A_{n-1}\md\hspace{22pt}\text{and}\hspace{22pt}f_i:F\A_n\md\to F\A_{n+1}\md\]
are biadjoint. Further $e_i(M\da^{\s_n}_{\A_n})\cong(e_iM)\da^{\s_{n-1}}_{\A_{n-1}}$ and $f_i(M\da^{\s_n}_{\A_n})\cong(f_iM)\da^{\s_{n+1}}_{\A_{n+1}}$ for any $\s_n$-module $M$.

When considering $V=D^\la$ for some $\la\in\Par_p(n)$, more is known about the modules $e_i^{(r)}D^\la$ and $f_i^{(r)}D^\la$. Given any partition $\al\in\Par(n)$ and a residue $i$, let the $i$-signature of $\al$ to be the sequence of signs consisting of a $-$ resp. $+$ for each removable resp. addable $i$-node of $\al$, read from left to right. The reduced $i$-signature of $\al$ is then obtained by recursively removing $-+$ pairs from the $i$-signature. Nodes which corresponds to $-$ resp. $+$ in the reduced $i$-signature are called $i$-normal resp. $i$-conormal. If $\al$ has at least one $i$-normal node then the leftmost $i$-normal node is called $i$-good. Similarly if $\al$ has at least one $i$-conormal node then the rightmost $i$-conormal node is called $i$-cogood. 
For any residue $i$, let $\eps_i(\al)$ be the number of $i$-normal nodes of $\al$ and $\phi_i(\al)$ the number of $i$-conormal nodes of $\al$. If $\eps_i(\al)>0$ let $\tilde e_i\al:=\al_A$ where $A$ is the $i$-good node of $\al$, while if $\phi_i(\al)>0$ let $\tilde f_i\al:=\al^B$ where $B$ is the $i$-cogood node of $\al$. 

The following two results hold by \cite[Theorems E(iv), E'(iv)]{bk6}, \cite[Theorem 1.4]{k4} and \cite[Theorems 11.2.10,  11.2.11]{KBook} (again the case $r=0$ is trivial).

\begin{lemma}\label{Lemma39}
Let $\lambda\in\Par_p(n)$. Then for any residue $i$ and any $r\geq 0$:
\begin{enumerate}[{\rm (i)}]
\item $e_i^rD^\lambda\cong(e_i^{(r)}D^\lambda)^{\oplus r!}$,
\item  $e_i^{(r)}D^\lambda\not=0$ if and only if $r\leq \eps_i(\lambda)$, in which case $e_i^{(r)}D^\lambda$ is a self-dual indecomposable module with socle and head both isomorphic to $D^{\widetilde e_i^r\la}$, 
\item  $[e_i^{(r)}D^\lambda:D^{\widetilde e_i^r\la}]=\binom{\eps_i(\lambda)}{r}=\dim\End_{\s_{n-r}}(e_i^{(r)}D^\lambda)$,
\item if $D^\mu$ is a composition factor of $e_i^{(r)}D^\lambda$ then $\eps_i(\mu)\leq \eps_i(\lambda)-r$, with equality holding if and only if $\mu=\widetilde e_i^r\la$,
\item  $\dim\End_{\s_{n-1}}(D^\lambda\da_{\s_{n-1}})=\sum_{j\in I}\eps_j(\lambda)$,
\item if $A$ is a removable node of $\la$ such that $\la_A$ is $p$-regular, then $D^{\la_A}$ is a composition factor of $e_i D^\la$ if and only if $A$ is $i$-normal, in which case $[e_i D^\la:D^{\la_A}]$ is one more than the number of $i$-normal nodes of $\la$ above $A$. 
\end{enumerate}
\end{lemma}

\begin{lemma}\label{Lemma40}
Let $\lambda\in\Par_p(n)$. Then for any residue $i$ and any $r\geq 0$:
\begin{enumerate}[{\rm (i)}]
\item $f_i^rD^\lambda\cong(f_i^{(r)}D^\lambda)^{\oplus r!}$,
\item $f_i^{(r)}D^\lambda\not=0$ if and only if $r\leq \phi_i(\lambda)$, in which case $f_i^{(r)}D^\lambda$ is a self-dual indecomposable module with socle and head both isomorphic to $D^{\widetilde f_i^r\la}$,
\item  $[f_i^{(r)}D^\lambda:D^{\widetilde f_i^r\la}]=\binom{\phi_i(\lambda)}{r}=\dim\End_{\s_{n+r}}(f_i^{(r)}D^\lambda)$,
\item if $D^\mu$ is a composition factor of $f_i^{(r)}D^\lambda$ then $\phi_i(\mu)\leq \phi_i(\lambda)-r$, with equality holding if and only if $\mu=\widetilde f_i^r\la$,
\item $\dim\End_{\s_{n+1}}(D^\lambda\ua^{\s_{n+1}})=\sum_{j\in I}\phi_j(\lambda)$,
\item if $B$ is an addable node for $\la$ such that $\la^B$ is $p$-regular, then $D^{\la^B}$ is a composition factor of $f_i D^\la$ if and only if $B$ is $i$-conormal, in which case $[f_i D^\la:D^{\la^B}]$ is one more than the number of $i$-conormal nodes of $\la$ below~$B$. 
\end{enumerate}
\end{lemma}

By parts (ii) and (iii) of Lemma \ref{Lemma39}, we have that $D^\la\da_{\s_{n-1}}$ is irreducible if and only if $\la$ has only one normal node. Thus JS-partitions are exactly the $p$-regular partitions with only one normal node.

When considering the modules $e_iD^\la$ and $f_iD^\la$ more can be said on their submodule structure.

\begin{lemma}\label{lspechtfiltration}
Let $\la\in\Par_p(n)$. Let $B_1,\ldots,B_{\eps_i(\la)}$ be the $i$-normal nodes of $\la$ labeled starting from the bottom and $C_1,\ldots,C_{\phi_i(\la)}$ be the $i$-conormal nodes of $\la$ labeled starting from the top. Then
\[e_iD^\la\sim W_{\eps_i(\la)}|\ldots|W_1\hspace{22pt}\text{and}\hspace{22pt}f_iD^\la\sim X_{\phi_i(\la)}|\ldots|X_1\]
for certain modules $W_k$, $X_k$ such that:
\begin{enumerate}[{\rm (i)}]
\item  for $1\leq k\leq\eps_i(\la)$, $W_k$ is a non-zero quotient of $S^{\la_{B_k}}$,

\item for $1\leq\ell\leq k\leq\eps_i(\la)$ if $\la_{B_\ell}\in\Par_p(n-1)$ then $[W_k:D^{\la_{B_\ell}}]=1$,

\item  for $1\leq k\leq\phi_i(\la)$, $X_k$ is a non-zero quotient of $S^{\la^{C_k}}$,

\item for $1\leq\ell\leq k\leq\phi_i(\la)$ if $\la^{C_\ell}\in\Par_p(n+1)$ then $[X_k:D^{\la^{C_\ell}}]=1$.
\end{enumerate}
\end{lemma}

\begin{proof}
The statement for $f_iD^\la$ is just the first two points of \cite[Remark on p. 83]{bk6}, which can be obtained from \cite[Theorem C]{bk6} using Schur functors. The statement for $e_iD^\la$ can be similarly obtained from \cite[Theorem C']{bk6}. See \cite[p. 84]{bk6} for some details on how this works (where some other results for representation of symmetric groups have been obtained from the corresponding results for representations of general linear groups).
\end{proof}

\begin{lemma}\label{l8}
Let $\la\in\Par_p(n)$, $i\in I$ and $\eps_i(\la)>0$. Then, for $1\leq a\leq \eps_i(\la)$, there exist submodules $V_a\subseteq e_iD^\la$ such that the following hold:
\begin{enumerate}[{\rm (i)}]
\item $[V_a:D^{\tilde{e}_i \la}]=a$,

\item $V_a$ has socle and head both isomorphic to $D^{\tilde{e}_i\la}$,

\item $V_a\subseteq V_{a+1}$ for $1\leq a<\eps_i(\la)$,

\item $V_a$ is self-dual,

\item if $V\subseteq e_iD^\la$ and $a=[V:D^{\tilde e_i\la}]$ then $V_a\subseteq V$,

\item if $V\subseteq e_iD^\la$ has head isomorphic to $(D^{\tilde e_i\la})^{\oplus b}$ with $b\geq 1$ then $V=V_a$ where $a=[V:D^{\tilde e_i\la}]$.
\end{enumerate}
\end{lemma}

\begin{proof}
The first four points are just \cite[Lemma 3.9]{kmt}. The last two part then hold by \cite[Lemma 2.2]{gkm} in view of Lemma \ref{Lemma39}(iii).
\end{proof}

\begin{lemma}\label{l9}
Let $\la\in\Par_p(n)$, $i\in I$ and $\phi_i(\la)>0$. Then, for $1\leq a\leq \phi_i(\la)$, there exist submodules $V_a\subseteq f_iD^\la$ such that the following hold:
\begin{enumerate}[{\rm (i)}]
\item $[V_a:D^{\tilde{f}_i \la}]=a$,

\item $V_a$ has socle and head both isomorphic to $D^{\tilde{f}_i\la}$,

\item $V_a\subseteq V_{a+1}$ for $1\leq a<\phi_i(\la)$,

\item $V_a$ is self-dual,

\item if $V\subseteq e_iD^\la$ and $a=[V:D^{\tilde e_i\la}]$ then $V_a\subseteq V$,

\item if $V\subseteq f_iD^\la$ has head isomorphic to $(D^{\tilde f_i\la})^{\oplus b}$ with $b\geq 1$ then $V=V_a$ where $a=[V:D^{\tilde f_i\la}]$.
\end{enumerate}
\end{lemma}

\begin{proof}
As the previous lemma, using \cite[Lemma 3.11]{kmt} and Lemma \ref{Lemma40}(iii).
\end{proof}

The next lemma studies socles of certain submodules of $e_je_iD^\la$ for $i\not=j$. The case $i=j$ is not considered, since it is already covered by Lemma \ref{Lemma39}.

\begin{lemma}\label{L110820_2}
Let $i\not=j$ be residues, $\la\in\Par_p(n)$ with $\eps_i(\la)\geq k$ and $M\subseteq e_iD^\la$ with $[M:D^{\tilde e_i\la}]=k$ and $\hd(M),\soc(M)\cong D^{\tilde e_i\la}$. If $D^\mu\subseteq e_jM$ then the following hold:
\begin{itemize}
\item $\tilde f_j\mu\unrhd \la_A$ for $A$ is the $k$-th lowest $i$-normal node of $\la$,

\item $\tilde e_i\la\unrhd \mu^B$ for $B$ the lowest $j$-conormal node of $\mu$,

\item $\phi_j(\mu)\geq\phi_j(\tilde e_i\la)+1$,

\item $\eps_i(\la)\geq\eps_i(\tilde f_j\mu)+1$.
\end{itemize}
If either of the last two inequalities hold then $\eps_j(\tilde e_i\la)>0$ and $\mu=\tilde e_j\tilde e_i\la$.
\end{lemma}

\begin{proof}
By Lemma \ref{l8}, there exists a self-dual submodule $V_k\subseteq e_iD^\la$ such that if $N\subseteq e_iD^\la$ and $[N:D^{\tilde e_i\la}]=k$ and $\hd(N)\cong (D^{\tilde e_i\la})^{\oplus a}$ for some $a\geq 1$, then $N=V_k$. In particular $M=V_k$, from which it also follows that $M$ is self-dual, and is the unique submodule of $e_iD^\la$ with $[M:D^{\tilde e_i\la}]=k$ and $\hd(M),\soc(M)\cong D^{\tilde e_i\la}$.

Similarly, if $N'$ is a quotient of $e_iD^\la$ with $[N':D^{\tilde e_i\la}]=k$ and $\soc(N')\cong (D^{\tilde e_i\la})^{\oplus b}$ for some $b\geq 1$, then $(N')^*$ is isomorphic to a submodule of $(e_iD^\la)^*\cong e_iD^\la$ (the last isomorphism holding by Lemma \ref{Lemma39}) and then $(N')^*\cong V_k=M$. By self-duality of $M$ it follows that $N'\cong M$.

By Lemma \ref{lspechtfiltration} if $A_\ell$ is the $\ell$-th lowest $i$-normal node of $\la$ for $1\leq\ell\leq k$, there exist quotients $W_\ell$ of $S^{\la_{A_\ell}}$ with $[W_\ell:D^{\tilde e_i\la}]=1$ for every $1\leq \ell\leq k$ and a quotient $W$ of $e_iD^\la$ such that $W\sim W_k|\ldots|W_1$. Note that $[W:D^{\tilde e_i\la}]=k$. If $Y\subseteq W$ has no composition factor isomorphic to $D^{\tilde e_i\la}$ and $Y$ is maximal with this property, it follows that $M\cong W/Y$. In particular if $D^\pi$ is a composition factor of $M$, then it is a composition factor of some $S^{\la_{A_\ell}}$ with $1\leq \ell\leq k$ and so $\pi\unrhd\la_{A_\ell}\unrhd\la_{A_k}=\la_A$. Furthermore $\eps_i(\pi)\leq \eps_i(\la)-1$ from Lemma \ref{Lemma39} and if equality holds then $\pi=\tilde e_i\la$.

Again by Lemma \ref{lspechtfiltration} if $D^\psi$ is a composition factor of $f_jD^\mu$ then there exists a $j$-conormal node $B'$ of $\mu$ such that $D^\psi$ is a composition factor of $S^{\mu^{B'}}$ and then $\psi\unrhd\mu^{B'}\unrhd\mu^B$. From Lemma \ref{Lemma40}, $\phi_j(\psi)\leq\phi_j(\mu)-1$ and if equality holds then $\psi=\tilde f_j\mu$.

By Lemma \ref{lef}
\[\dim\Hom_{\s_{n-1}}(f_jD^\mu,M)=\dim\Hom_{\s_{n-2}}(D^\mu,e_jM)\geq 1\]
and so $D^{\tilde e_i\la}\cong\soc(M)$ is a composition factor of $f_jD^\mu$ and $D^{\tilde f_i\mu}\cong\hd(f_jD^\mu)$ is a composition of factor of $M$. The lemma follows.
\end{proof}

We will now consider socles of the restrictions of $e_i^{(r)}D^\la$ to $\A_{n-r}$ in characteristic 2.

\begin{lemma}\label{L060820}
Let $p=2$ and $\la\in\Par_2(n)$. If $\eps_i(\la)\geq r$ and $\tilde e_i^r\la\in\Parinv_2(n-r)$ then $\soc((e_i^{(r)}D^\la)\da_{\A_{n-r}})\cong E^{\tilde e_i^r\la}_+\oplus E^{\tilde e_i^r\la}_-$.
\end{lemma}

\begin{proof}
For $\mu\in\Parinv_2(n-r)$ we have $E^\mu_\pm\ua^{\s_{n-r}}\cong D^\mu$. For $\mu\in\Par_2(n-r)\setminus\Parinv_2(n-r)$ we instead have that
\[E^\mu\ua^{\s_{n-r}}\cong D^\mu\da_{\A_{n-r}}\ua^{\s_{n-r}}\cong (D^\mu\otimes\1)\da_{\A_{n-r}}\ua^{\s_{n-r}}\cong D^\mu\otimes(\1\ua_{\A_{n-r}}^{\s_{n-r}})\sim D^\mu|D^\mu.\]
From Lemma \ref{Lemma39} we have that $\soc(e_i^{(r)}D^\la)\cong D^{\tilde e_i^r\la}$. So
\begin{align*}
\dim\Hom_{\A_{n-r}}(E^\mu_{(\pm)},(e_i^{(r)}D^\la)\da_{\A_{n-r}})&=\dim\Hom_{\s_{n-r}}(E^\mu_{(\pm)}\ua^{\s_{n-r}},e_i^{(r)}D^\la)\\
&=\de_{\mu,\tilde e_i^r\la}.
\end{align*}
\end{proof}

\begin{lemma}\label{L060820_2}
Let $p=2$ and $\la\in\Par_2(n)$. Assume that $\eps_i(\la)\geq r$ and $\tilde e_i^r\la\not\in\Parinv_2(n-r)$. Then:
\begin{itemize}
\item $\soc((e_i^{(r)}D^\la)\da_{\A_{n-r}})\cong (E^{\tilde e_i^r\la})^{\oplus a}$ with $1\leq a\leq 2$,

\item $\soc((e_i^{(r)}D^\la)\da_{\A_{n-r}})\cong (E^{\tilde e_i^r\la})^{\oplus 2}$ if and only if $E^{\tilde e_i^r\la}\ua^{\s_{n-r}}\subseteq e_i^{(r)}D^\la$,

\item if $\la\in\Parinv_2(n)$ and $r=1$ then $\soc((e_iD^\la)\da_{\A_{n-1}})\cong (E^{\tilde e_i\la})^{\oplus 2}$,

\item if $\tilde e_i^{\eps_i(\la)}\la\not\in\Parinv_2(n-\eps_i(\la))$ or $\tilde f_i^{\phi_i(\la)}\la\not\in\Parinv_2(n+\phi_i(\la))$, then $\soc((e_i^{(r)}D^\la)\da_{\A_{n-r}})\cong E^{\tilde e_i^r\la}$.
\end{itemize}
\end{lemma}

\begin{proof}
We will use Lemmas \ref{Lemma39} and \ref{Lemma40} without further comment.

As in the previous lemma $E^\mu_\pm\ua^{\s_{n-r}}\cong D^\mu$ if $\mu\in\Parinv_2(n-r)$ while $E^\mu\ua^{\s_{n-r}}\sim D^\mu|D^\mu$ if $\mu\in\Par_2(n-r)\setminus\Parinv_2(n-r)$. In this last case from
\[\dim\Hom_{\s_{n-r}}(E^\mu\ua^{\s_{n-r}},D^\mu)=\dim\Hom_{\A_{n-r}}(E^\mu,D^\mu\da_{\A_{n-r}})=1\]
it follows that $E^\mu\ua^{\s_{n-r}}\cong D^\mu|D^\mu$.


For any $\mu\in\Par_2(n-r)$ we have that
\[\Hom_{\A_{n-r}}(E^\mu_{(\pm)},(e_i^{(r)}D^\la)\da_{\A_{n-r}})\cong\Hom_{\s_{n-r}}(E^\mu_{(\pm)}\ua^{\s_{n-r}},e_i^{(r)}D^\la).\]
So if $\mu\in\Parinv_2(n-r)$ then since $E^\mu_\pm\ua^{\s_{n-r}}\cong D^\mu$ we have that
\[\dim\Hom_{\s_{n-r}}(E^\mu_\pm\ua^{\s_{n-r}},e_i^{(r)}D^\la)=\dim\Hom_{\s_{n-r}}(D^\mu,e_i^{(r)}D^\la),\]
while if $\mu\in\Par_2(n-r)\setminus\Parinv_2(b-r)$ then, from $E^\mu\ua^{\s_{n-r}}\cong D^\mu|D^\mu$, we have that
\begin{align*}
\dim\Hom_{\s_{n-r}}(D^\mu,e_i^{(r)}D^\la)&=\dim\Hom_{\s_{n-r}}(\hd(E^\mu\ua^{\s_{n-r}}),e_i^{(r)}D^\la)\\
&\leq\dim\Hom_{\s_{n-r}}(E^\mu\ua^{\s_{n-r}},e_i^{(r)}D^\la)\\
&\leq 2\cdot\dim\Hom_{\s_{n-r}}(D^\mu,e_i^{(r)}D^\la).
\end{align*}

If $\mu\not=\tilde e_i^r\la$ then
\[\dim\Hom_{\A_{n-r}}(E^\mu_{(\pm)},(e_i^{(r)}D^\la)\da_{\A_{n-r}})=0,\]
since $D^\mu\not\subseteq e_i^{(r)}D^\la$. If on the other hand $\mu=\tilde e_i^r\la$ then
\[1\leq \dim\Hom_{\A_{n-r}}(E^{\tilde e_i^r\la},(e_i^{(r)}D^\la)\da_{\A_{n-r}})\leq 2.\]
The first statement then holds. Furthermore we have that
\[\dim\Hom_{\A_{n-r}}(E^{\tilde e_i^r\la},(e_i^{(r)}D^\la)\da_{\A_{n-r}})=2\]
if and only if
\[\dim\Hom_{\s_{n-r}}(E^{\tilde e_i^r\la}\ua^{\s_{n-r}},e_i^{(r)}D^\la)>\dim\Hom_{\s_{n-r}}(\hd(E^{\tilde e_i^r\la}\ua^{\s_{n-r}}),e_i^{(r)}D^\la).\]
As $E^{\tilde e_i^r\la}\ua^{\s_{n-r}}$ has only two composition factors, the second statement also holds.

If $\la\in\Parinv_2(n)$ and $r=1$ then $(e_iD^\la)\da_{\A_{n-1}}\cong (e_iE^\la_+)\oplus (e_iE^\la_-)$. Since $(e_iE^\la_+)^\si\cong e_iE^\la_-$ for $\si\in\s_{n-1}\setminus\A_{n-1}$, the socle of  $(e_iD^\la)\da_{\A_{n-1}}$ cannot be simple, so $a=2$ in this case.

We may now assume that $\tilde e_i^{\eps_i(\la)}\la\not\in\Parinv_2(n-\eps_i(\la))$ or that $\tilde f_i^{\phi_i(\la)}\la\not\in\Parinv_2(n+\phi_i(\la))$. From \cite[Lemmas 3.3, 3.4]{m1}
\[e_i^{(r)}D^\la\subseteq e_i^{(\phi_i(\la)+r)}D^{\tilde f_i^{\phi_i(\la)}\la}\cong f_i^{(\eps_i(\la)-r)}D^{\tilde e_i^{\eps_i(\la)}\la}.\]
Thus it is enough to check that $\dim\Hom_{\A_{n-r}}(E^\la,M\da_{\A_{n-r}})=1$ for $M=e_i^{(\phi_i(\la)+r)}D^{\tilde f_i^{\phi_i(\la)}\la}\cong f_i^{(\eps_i(\la)-r)}D^{\tilde e_i^{\eps_i(\la)}\la}$. If $\tilde e_i^{\eps_i(\la)}\la\not\in\Parinv_2(n-\eps_i(\la))$ then
\begin{align*}
&\dim\Hom_{\A_{n-r}}(E^\la,(f_i^{(\eps_i(\la)-r)}D^{\tilde e_i^{\eps_i(\la)}\la})\da_{\A_{n-r}})\\
&=1/(\eps_i(\la)-r)!\dim\Hom_{\A_{n-r}}(E^\la,(f_i^{\eps_i(\la)-r}D^{\tilde e_i^{\eps_i(\la)}\la})\da_{\A_{n-r}})\\
&=1/(\eps_i(\la)-r)!\dim\Hom_{\A_{n-r}}(E^\la,f_i^{\eps_i(\la)-r}(E^{\tilde e_i^{\eps_i(\la)}\la}))\\
&=1/(\eps_i(\la)-r)!\dim\Hom_{\A_{n-\eps_i(\la)}}(e_i^{\eps_i(\la)-r}E^\la,E^{\tilde e_i^{\eps_i(\la)}\la})\\
&=1/(\eps_i(\la)-r)!\dim\Hom_{\A_{n-\eps_i(\la)}}((e_i^{\eps_i(\la)-r}D^\la)\da_{\A_{n-\eps_i(\la)}},E^{\tilde e_i^{\eps_i(\la)}\la})\\
&=1/(\eps_i(\la)-r)!\dim\Hom_{\A_{n-\eps_i(\la)}}((E^{\tilde e_i^{\eps_i(\la)}\la})^{\oplus (\eps_i(\la)-r)!},E^{\tilde e_i^{\eps_i(\la)}\la})\\
&=1.
\end{align*}
The case $\tilde f_i^{\phi_i(\la)}\la\not\in\Parinv_2(n+\phi_i(\la))$ is similar.
\end{proof}

\begin{lemma}\label{L060820a}
Let $p=2$ and $\la\in\Par_2(n)$. If $\phi_i(\la)\geq r$ and $\tilde f_i^r\la\in\Parinv_2(n+r)$ then $\soc((f_i^{(r)}D^\la)\da_{\A_{n+r}})\cong E^{\tilde f_i^r\la}_+\oplus E^{\tilde f_i^r\la}_-$.
\end{lemma}

\begin{proof}
Similar to Lemma \ref{L060820}.
\end{proof}

\begin{lemma}\label{L060820a_2}
Let $p=2$ and $\la\in\Par_2(n)$. Assume that $\phi_i(\la)\geq r$ and $\tilde f_i^r\la\not\in\Parinv_2(n+r)$. Then:
\begin{itemize}
\item $\soc((f_i^{(r)}D^\la)\da_{\A_{n+r}})\cong (E^{\tilde f_i^r\la})^{\oplus a}$ with $1\leq a\leq 2$,

\item $\soc((f_i^{(r)}D^\la)\da_{\A_{n+r}})\cong (E^{\tilde f_i^r\la})^{\oplus 2}$ if and only if $E^{\tilde f_i^r\la}\ua^{\s_{n+r}}\subseteq f_i^{(r)}D^\la$,

\item if $\la\in\Parinv_2(n)$ and $r=1$ then $\soc((f_iD^\la)\da_{\A_{n+1}})\cong (E^{\tilde f_i\la})^{\oplus 2}$,

\item if $\tilde e_i^{\eps_i(\la)}\la\not\in\Parinv_2(n-\eps_i(\la))$ or $\tilde f_i^{\phi_i(\la)}\la\not\in\Parinv_2(n+\phi_i(\la))$, then $\soc((f_i^{(r)}D^\la)\da_{\A_{n+r}})\cong E^{\tilde f_i^r\la}$.
\end{itemize}
\end{lemma}

\begin{proof}
Similar to Lemma \ref{L060820_2}.
\end{proof}

We will now consider restrictions to $\s_{n-2,2}$. First we need the following result. For a proof see for example \cite[Lemma 1.2]{bk2}.

\begin{lemma}\label{simplesoc}
Let $A,B,C$ be modules of a group $G$. If $A$ has simple socle and $A\subseteq B\oplus C$ then $A$ is isomorphic to a submodule of $A$ or $B$.
\end{lemma}

\begin{lemma}\label{L090620}
Let $\la\in\Par_p(n)$. If $A\subseteq D^\la\da_{\s_{n-2,2}}$ with $A\da_{\s_{n-2}}= e_i^2D^\la$, then $A\cong e_i^{(2)}D^\la\boxtimes M^{(1^2)}$.
\end{lemma}

\begin{proof}
From Lemma \ref{Lemma39} we may assume that $\eps_i(\la)\geq 2$, since otherwise $e_i^2D^\la,e_i^{(2)}D^\la=0$ (and then also $A=0$). In addition $e_i^2D^\la\cong (e_i^{(2)}D^\la)^{\oplus 2}$ and $\soc(e_i^2D^\la)=(D^{\tilde e_i^2\la})^{\oplus 2}$. It follows that the only simple modules which may appear in the socle of $A$ are $D^{\tilde e_i^2\la}\boxtimes D^{(2)}$ and $D^{\tilde e_i^2\la}\boxtimes D^{(1^2)}$ (the second one only if $p>2$). From \cite[\S11.2]{KBook} and block decomposition
\begin{align*}
\dim\Hom_{\s_{n-2,2}}(D^{\tilde e_i^2\la}\boxtimes D^{(2)},D^\la\da_{\s_n})&=\dim\Hom_{\s_n}((D^{\tilde e_i^2\la}\boxtimes D^{(2)})\ua^{\s_n},D^\la)\\
&=\dim\Hom_{\s_n}(f_i^{(2)}D^{\tilde e_i^2\la},D^\la)\\
&=1.
\end{align*}
Furthermore,
\[A\subseteq A\da_{\s_{n-2}}\ua^{\s_{n-2,2}}\cong (e_i^{(2)}D^\la\boxtimes M^{(1^2)})^{\oplus 2}.\]

If $p=2$ then $A$ has simple socle. If instead $p>2$ then $M^{(1^2)}\cong D^{(2)}\oplus D^{(1^2)}$ (since in this case $FS_2$ is semisimple) and
\begin{align*}
&\dim\Hom_{\s_{n-2,2}}(D^{\tilde e_i^2\la}\boxtimes D^{(1^2)},D^\la\da_{\s_n})\\
&=\dim\Hom_{\s_{n-2,2}}((D^{\tilde e_i^2\la}\otimes\sgn)\boxtimes D^{(2)},(D^\la\otimes\sgn)\da_{\s_n})\\
&=1
\end{align*}
and so, by block decomposition, $A=A'\oplus A''$ with $A'$ and $A''$ having simple socle and $A'\subseteq (e_i^{(2)}D^\la\boxtimes D^{(2)})^{\oplus 2}$ and $A''\subseteq (e_i^{(2)}D^\la\boxtimes D^{(1^2)})^{\oplus 2}$. 

In either case by Lemma \ref{simplesoc} applied either to $A$ or to both $A'$ and $A''$ we have that $A$ is isomorphic to a submodule of $e_i^{(2)}D^\la\boxtimes M^{(1^2)}$. The lemma then follows by comparing dimensions.
\end{proof}

\section{Spin modules}\label{s4}

In this section we give some results connected to decomposition matrices of spin modules. We start by stating the formula for branching reduction modulo 2 of spin representations (we consider here only the reductions modulo 2 to obtain simpler formulas when partitions label two irreducible spin representations).

\begin{lemma}\label{brs2}
Let $\la\in\Par_2(n)$. Then in characteristic 2, in the Grothendieck group,
\[[S(\la)\da_{\s_{n-1}}]=\sum_A2^{x_A}[S(\la_A)],\]
where the sum is taken over all removable nodes $A$ such that $\la_A\in\Par_2(n-1)$, and $x_A=1$ if $S(\la)=S(\la,0)$ and $S(\la_A)=S(\la_A,\pm)$, while $x_A=0$ in all other cases.
\end{lemma}

\begin{proof}
This holds by \cite[Theorem 2]{mo} or \cite[Theorem 8.1]{s}.
\end{proof}

We will next prove a result limiting possible composition factors of $S(\la)$.

\begin{lemma}\label{L280520_2}\label{L080620}
Let $p=2$ and $\la\in\Par_2(n)$. Then $S(\la)$ and $S^{\overline{\dbl}(\la)}$ are in the same block. Furthermore the following hold:
\begin{enumerate}[{\rm (i)}]
\item if $[S(\la):D^\mu]>0$ then $\mu\unrhd (\overline{\dbl}(\la))^R$,

\item if $[S(\la):D^\mu]>0$ and $h(\mu)\geq 2h(\la)-1$ then $\mu=\dbl(\nu)$ with $\nu\unrhd\la$,

\item if $\dbl(\la)\in\Par_2(n)$ then $[S(\la):D^{\dbl(\la)}]=2^{\lfloor h_2(\la)/2\rfloor}>0$.
\end{enumerate}
\end{lemma}

\begin{proof}
The assertion on blocks holds by \cite[Theorem 5.1]{bo2}, as does (i). (iii) holds by \cite[Theorem 1.2]{Benson} (or \cite[Theorem 5.1]{bo2}). So we only have to prove (ii). In view of (i) we may assume that $h(\mu)\leq 2h(\la)$.

If $\mu=\dbl(\nu)$ for some $\nu\in\Par_2(n)$ and $[S(\la):D^\mu]>0$ then from $\dbl(\nu)=\mu\unrhd (\overline{\dbl}(\la))^R\unrhd\dbl(\la)$. To see that $(\overline{\dbl}(\la))^R\unrhd\dbl(\la)$ note first that $\overline{\dbl}(\la)$ and $\dbl(\la)$ have the same number of nodes on each ladder (even though $\dbl(\la)$ is in general not a partition but only a composition, we may still identify it with its Young diagram and thus count nodes on a given ladder), from which the statement follows as $(\overline{\dbl}(\la))^R$ is obtained from $\overline{\dbl}(\la)$ (and then also from $\dbl(\la)$) by moving nodes as high as possible on ladders. So $\dbl(\nu)\unrhd\dbl(\la)$ and then $\nu\unrhd\la$ (as $\sum_{j\leq i}\nu_j=\sum_{j\leq 2i}(\dbl(\nu))_j$ for each $i$ and similarly for $\la$).

So assume now that $\mu$ is not the double of any partition. We may assume that $2h(\la)-1\leq h(\mu)\leq2h(\la)$. As $\mu$ is not the double of a partition, there exists $1\leq k\leq h(\la)$ such that $\mu_{2k-1}-\mu_{2k}>2$. Fix such a $k$ and let
\[\pi:=(2h(\la)+1,2h(\la),\ldots,2h(\la)-2k+3,2h(\la)-2k,2h(\la)-2k-1,\ldots,1)\]
(that is, $\pi$ is obtained from the staircase partition $(2h(\la)+1,2h(\la),\ldots,1)$ by removing the parts equal to $2h(\la)-2k+2$ and $2h(\la)-2k+1$). In particular $h(\pi)=2h(\la)-1$.

Then $D^\pi$ is a composition factor of $D^\mu\da_{\s_{|\pi|}}$. This follows by Lemma \ref{Lemma39}(vi), by always removing the highest removable node which gives a 2-regular partition $\psi$ with $\psi_{2k-1}-\psi_{2k}>2$. We will now check that $D^\pi$ is not the composition factor of the reduction modulo 2 of a spin module indexed by a partition with at most $h(\la)$ part. As all composition factors of $S(\la)\da_{\widetilde{\s}_{|\pi|}}$ have at most $h(\la)$ rows by Lemma \ref{brs2}, we then have that $D^\mu$ cannot be a composition factor of the reduction modulo 2 of $S(\la)$, so that (ii) holds.

From Lemma \ref{Lemma39} it can be checked that
\[e_0e_1^{(2)}e_0^{(3)}\ldots e_1^{(2h(\la)-2)}e_0^{(2h(\la)-1)}e_1^{(2k-1)}e_0^{(2k-1)}D^\pi\not=0\]
by always removing the highest removable node which gives a 2-regular partition (such a node is always normal). So if $\zeta\in\Par_2(|\pi|)$ and $D^\pi$ is a composition factor of $S(\zeta)$ then
\[e_0e_1^{(2)}e_0^{(3)}\ldots e_1^{(2h(\la)-2)}e_0^{(2h(\la)-1)}e_1^{(2k-1)}e_0^{(2k-1)}S(\zeta)\not=0.\]

As in \cite{bo2}, we define spin residues of nodes as follows: $\overline{\res}(a,b)=0$ if $b\equiv 0,3\Md{4}$, while $\overline{\res}(a,b)=1$ if $b\equiv 1,2\Md{4}$. It follows from the definition that for any residue $i$ and any 2-regular partition $\pi$, the number of nodes of $\pi$ of spin residue $i$ is equal to the number of nodes of residue $i$ of $\overline{\dbl}(\pi)$ (see also \cite{bo2}).

By Lemma \ref{brs2}, there then exist partitions $\zeta^\ell\in\Par_2(|\zeta^\ell|)$ for $0\leq \ell\leq 2h(\la)$ such that the following hold:
\begin{itemize}
\item $\zeta^0=()$,

\item for $1\leq\ell\leq 2h(\la)-1$, $\zeta^\ell$ is obtained from $\zeta^{\ell-1}$ by adding $\ell$ nodes of spin residue 0 if $\ell$ is odd or 1 if $\ell$ is even,

\item $\zeta^{2h(\la)}$ is obtained from $\zeta^{2h(\la)-1}$ by adding $2k-1$ nodes of spin residue 1,

\item $\zeta$ is obtained from $\zeta^{2h(\la)}$ by adding $2k-1$ nodes of spin residue 0.
\end{itemize}
It can be checked that if $1\leq\ell\leq 2h(\la)-1$ is odd then $\zeta^\ell=(2\ell-1,2\ell-5,\ldots,1)$, while if $\ell$ is even then $\zeta^\ell=(2\ell-1,2\ell-5,\ldots,3)$. In particular $\zeta^{2h(\la)-1}=(4h(\la)-3,4h(\la)-7,\ldots,1)$. Note that $h(\zeta^{2h(\la)-1})=h(\la)$. Since $\zeta^{2h(\la)-1}_j=1$ for $1\leq j\leq h(\la)$, it follows by the definition of spin residues that if $\zeta^{2h(\la)}$ has some addable node of spin residue $0$ on row $j$ with $1\leq j\leq h(\la)$, then $\zeta^{2h(\la)}_j-\zeta^{2h(\la)-1}_j=2$. In this case $\zeta_j-\zeta^{2h(\la)}_j\leq 2$ (that is when going from $\zeta^{2h(\la)}$ to $\zeta$ we add at most 2 nodes on row $j$). Since $|\zeta|-|\zeta^{2h(\la)}|=|\zeta^{2h(\la)}|-|\zeta^{2h(\la)-1}|=2k-1<2k$ when going from $\zeta^{2h(\la)}$ to $\zeta$ we cannot add nodes only to the top $h(\la)$ rows. Thus $h(\zeta)\geq h(\la)+1$.
\end{proof}

\begin{lemma}\label{branchingbs}\label{L180121}
Let $p=2$ and $n\geq 2$. Then $D^{\be_n}$ is the reduction modulo 2 of $S((n))$. In particular, in the Grothendieck group, $[D^{\be_n}\da_{\s_{n-1}}]=(1+\de_{2\nmid n})[D^{\be_{n-1}}]$, where $\de_{2\nmid n}=1$ if $2\nmid n$ and $\de_{2\nmid n}=0$ otherwise.
\end{lemma}

\begin{proof}
This follows from \cite[Table III]{Wales}, Lemma \ref{L280520_2} and branching of basic spin modules in characteristic 0 using Lemma \ref{brs2} (for small cases we can use Lemma \ref{L080620}(iii) and compare dimensions instead of \cite{Wales}).
\end{proof}

\begin{lemma}\label{bsspecht}
Let $p=2$. Let $S(\nu):=0$ for $\nu\in\Par(n)\setminus\Par_2(n)$. If $\la\in\Par(n)$ then
\[[S^\la\otimes S((n))]=d[S(\la)]+\sum_{\mu\rhd\la}d_\mu[S(\mu)]\]
with $d>0$ and $d_\mu\geq 0$ for $\mu\in\Par(n)$ with $\mu\rhd\la$.
\end{lemma}

\begin{proof}
This holds by \cite[Theorem 9.3]{s}.
\end{proof}

\begin{lemma}\label{bsspin}
Let $p=2$. If $\la\in\Par_2(n)$ then
\[[S(\la)\otimes S((n))]=\sum_{\mu\unrhd(\overline{\dbl}(\la))^R}d_\mu[D^\mu]\]
for some $d_\mu\geq 0$. If $\mu\in\Par_2(n)$ and $(\mu_1+\mu_2,\mu_3+\mu_4,\ldots)=\la$ then $d_\mu>0$. In particular if $\dbl(\la)\in\Par_2(n)$ then $d_{\dbl(\la)}>0$.
\end{lemma}

\begin{proof}
For any partition $\pi$ and any node $(i,j)$ of $\pi$ let $\diag(\pi)$ be the partition consisting of the diagonal hook-lengths of $\pi$. By \cite[Theorem 9.3]{s} we have that
\[[S(\la)\otimes S((n))]=c\sum_{\pi:\diag(\pi)=\la}[S^\pi]+\sum_{\pi:\diag(\pi)\rhd\la}c_\pi[S^\pi]\]
with $c>0$ and $c_\pi\geq 0$ for $\pi\in\Par(n)$. For any partition $\pi$ let $\pi'$ be its transposed partition. For any $j\geq 1$ define $\pi^1_j:=\max\{\pi_j,\pi'_j\}$ and $\pi^2_j:=\min\{\pi_j,\pi'_j\}$. Let $k:=h(\diag(\pi))$, that is $k$ is maximal with $(k,k)\in\pi$. We have that
\[\pi^R\unrhd (\pi_1^1,\pi^2_1-1,\pi^1_2-1,\pi^2_2-2,\ldots,\pi^1_k-k+1,\pi^2_k-k)\unrhd\overline{\dbl}(\diag(\pi))\]
(though in the middle we in general only have a composition and not necessarily a partition). To see this note that the first inequality holds by comparing the number of nodes on each ladder The second from
\[\overline{\dbl}(\diag(\pi))_{2j-1}+\overline{\dbl}(\diag(\pi))_{2j}=\diag(\pi)_j=\pi_j+\pi_j'-2j+1=\pi^1_j-j+1+\pi^2_j-j\]
together with
\[\overline{\dbl}(\diag(\pi))_{2j-1}-\overline{\dbl}(\diag(\pi))_{2j}\leq 1\leq(\pi^1_j-j+1)-(\pi^2_j-j).\]
The first assertion follows due to the leading composition factor of $S^\pi$, see \cite[Theorem 6.3.50]{JK}.

Assume now that $\mu\in\Par_2(n)$ and $(\mu_1+\mu_2,\mu_3+\mu_4,\ldots)=\la$. Let $h:=h(\la)$ and let $\pi$ be the unique partition of $n$ with $(\pi_1,\ldots,\pi_h)=(\mu_1,\mu_3+1,\mu_{2h-1}+h-1)$ and $(\pi_1',\ldots,\pi_h')=(\mu_2+1,\mu_4+2,\ldots,\mu_{2h}+h)$ (the parts of $(\pi_1,\ldots,\pi_h)$ and $(\pi'_1,\ldots,\pi'_h)$ are decreasing as $\mu$ is 2-regular). Then $\diag(\pi)=\la$ and $\pi^R=\mu$. So $d_\mu>0$ by the previous part.
\end{proof}

\begin{lemma}\label{L280520}
Let $p=2$. If $\sum_{\la\in\Par_2(n)}a_\la\ze^\la=\sum_{\la\in\Par_2(n)}b_\la\psi^\la$ and $a_\mu\not=0$, then there exists $\nu$ with $D^\nu$ and $S(\mu)$ in the same block such that $b_\nu\not=0$.
\end{lemma}

\begin{proof}
For any $\la\in\Par_2(n)$ let $\xi^\la$ be the character (not the 2-Brauer character) of the $\widetilde{\s}_n$-module $S(\la,0)$ or $S(\la,+)\oplus S(\la,-)$. Then the characters $\xi^\la$ are linearly independent. For any partition $\al\in\Par(n)$ let $\si_\al\in\widetilde{\s}_n$ be the lift of an element with cycle partition $\al$ with the assumption that $\si_\al$ has odd order if $\al\in\Paro(n)$. Let $z$ be the central element of $\widetilde{\s_n}$ with $\widetilde{\s}_n/\langle z\rangle\cong\s_n$. Then by \cite[Theorem 7.1]{s} or \cite[p. 235]{s5} we have that $\xi^\la(\si_\al)=0$ if $\al\not\in\Paro(n)$. If on the other hand $\al\in\Paro(n)$ then $\xi^\la(\si_\al)=2^{a_\la}\ze^\la_\al$ and $\xi^\la(z\si_\al)=-2^{a_\la}\ze^\la_\al$, where $a_\la=0$ if $S(\la)=S(\la,0)$ while $a_\la=1$ if $S(\la)=S(\la,\pm)$. It then follows that $\{\ze^\la:\la\in\Par_2(n)\}$ is linearly independent.

As $\ze^\la$ can be written as a linear combination of the characters $\psi^\mu$ with $D^\mu$ in the same block as $S(\la)$, the result follows.
\end{proof}

\begin{lemma}\label{L080620_2}
Let $p=2$. If $\la\in\Par_2(n)$ then $D^\la\otimes D^{\be_n}$ has a composition factor in the block of $S(\la)$.
\end{lemma}

\begin{proof}
From Lemma \ref{branchingbs} we have that $D^{\be_n}=S((n))$. Furthermore there exist $c_\mu\in\Z$ such that, in the Grothendieck group, $[D^\la]=[S^\la]+\sum_{\mu\rhd\la}c_\mu[S^\mu]$. So, by Lemma \ref{bsspecht},
\begin{align*}
[D^\la\otimes D^{\be_n}]&=[D^\la\otimes S((n))]\\
&=[S^\la\otimes S((n))]+\sum_{\mu\rhd\la}c_\mu[S^\mu\otimes S((n))]\\
&=d[S(\la)]+\sum_{\mu\rhd\la}d_\mu[S(\mu)]
\end{align*}
for some $d,d_\mu \in\Q$ with $d\not=0$. The lemma then follows from Lemma \ref{L280520}.
\end{proof}

\begin{lemma}\label{L280520_3}
Let $p=2$ and $\la,\nu\in\Par_2(n)$. If $\dbl(\la)\in\Par_2(n)$ and $D^\nu$ is a composition factor of $D^{\dbl(\la)}\otimes D^{\be_n}$ then $\nu\unrhd\dbl(\la)$.
\end{lemma}

\begin{proof}
By Lemma \ref{L280520_2} we have that $D^\nu$ is a composition factor of $S(\la)\otimes S((n))$. The result then follows from Lemma \ref{bsspin} as $\overline{\dbl(\la)}^R=\dbl(\la)$ because $\dbl(\la)\in\Par_2(n)$.
\end{proof}

\section{Characters}\label{s5}

In this section we will study values of certain 2-Brauer characters and their divisibility by 2 or 4.

\begin{lemma}\label{cbs}
Let $p=2$, $n\geq 2$ and $\al\in\Paro(n)$. Then:
\begin{enumerate}
\item $\psi^{\be_n}_\al=\pm 2^{\lfloor (h(\al)-1)/2\rfloor}$,

\item if $n$ is odd and $\al\in\Parod(n)$ then $\psi^{\be_n,+}_{\al,+}\not=\psi^{\be_n,-}_{\al,+}$ if and only if $\al=(n)$,

\item if $n\equiv 0\Md{4}$ and $\al\in\Parod(n)$ then $\psi^{\be_n,+}_{\al,+}\not=\psi^{\be_n,-}_{\al,+}$ if and only if $h(\al)=2$,

\item if $n\not\equiv 2\Md{4}$ then $\psi^{\be_n,\pm}$ is non-zero on every 2-regular conjugacy class.
\end{enumerate}
\end{lemma}

\begin{proof}
It is well known that $\ze^{(n)}_\al=\pm 2^{\lfloor (h(\al)-1)/2\rfloor}$, see \cite[VII, p.203]{s5}. So (i) follows by $\psi^{\be_n}=\ze^{(n)}$, see \cite[Table III]{Wales}.

Consider now the case where $n$ is odd. Then by \cite[Table III]{Wales} we also have that the modules $E^{\be_n}_\pm$ are the reductions modulo 2 of the basic spin modules of $\widetilde\A_n$. So it is easy to recover from \cite[VII*, p.205]{s5} that $\psi^{\be_n,+}_{\al,+}\not=\psi^{\be_n,-}_{\al,+}$ if and only if $\al=(n)$ and that 
\[|\psi^{\be_n,+}_{(n),+}-\psi^{\be_n,-}_{(n),+}|^2=n>1=|\psi^{\be_n}_{(n)}|^2.\]
It then follows that $\psi^{\be_n,\pm}$ is non-zero on any 2-regular conjugacy class.

Now suppose $n\equiv 0\Md{4}$ and fix $\al\in\Parod(n)$. All composition factors of $D^{\be_n}\da_{\s_{\al_1,n-\al_1}}$ are of the form $D^{\be_{\al_1}}\boxtimes D^{\be_{n-\al_1}}$ since any composition factor of $D^{\be_n}\da_{\s_m}$ is of the form $D^{\be_m}$ by Lemma \ref{branchingbs} (for any $m\leq n$). By Lemma \ref{branchingbs} we also have that $\dim D^{\be_m}=2^{\lfloor (m-1)/2\rfloor}$ for any $m\geq 1$. As $n$ is even and $\al_1$ is odd it follows that $D^{\be_n}$ and $D^{\be_{\al_1}}\boxtimes D^{\be_{n-\al_1}}$ have the same dimension. So $D^{\be_n}\da_{\s_{\al_1,n-\al_1}}\cong D^{\be_{\al_1}}\boxtimes D^{\be_{n-\al_1}}$.

Note that $\al$ consists of an even number of distinct odd parts, as $\al\in\Parod(n)$ and $n$ is even. In particular $3\leq \al_1<n$ is odd. Since $\al_1>1$ is odd $D^{\be_{\al_1}}\da_{\A_{\al_1}}$ splits, as does $D^{\be_{n-\al_1}}\da_{\A_{n-\al_1}}$ if $n-\al_1>1$.

Assume first that $h(\al)>2$. Then $n-\al_1>1$. So, conjugating representations with the permutations $(1,2)$ and $(n-1,n)$,
\begin{align*}
(E^{\be_n}_\pm)^{(1,2)}&\cong E^{\be_n}_\mp\cong(E^{\be_n}_\pm)^{(n-1,n)},\\
(E^{\be_{\al_1}}_\eps\boxtimes E^{\be_{n-\al_1}}_\de)^{(1,2)}&\cong E^{\be_{\al_1}}_{-\eps}\boxtimes E^{\be_{n-\al_1}}_\de,\\
(E^{\be_{\al_1}}_\eps\boxtimes E^{\be_{n-\al_1}}_\de)^{(n-1,n)}&\cong E^{\be_{\al_1}}_\eps\boxtimes E^{\be_{n-\al_1}}_{-\de}
\end{align*}
for any $\eps,\de\in\{\pm\}$. It follows that
\[E^{\be_n}_\pm\da_{\A_{\al_1}\times \A_{n-\al_1}}\cong(E^{\be_{\al_1}}_+\boxtimes E^{\be_{n-\al_1}}_\pm)\oplus (E^{\be_{\al_1}}_-\boxtimes E^{\be_{n-\al_1}}_\mp).\]
As $h((\al_2,\al_3,\ldots))>1$ by the case where $n$ is odd we have that
\begin{align*}
\psi^{\be_n,+}_{\al,+}&=\psi^{\be_{\al_1},+}_{(\al_1),+}\psi^{\be_{n-\al_1},\pm}_{(\al_2,\al_3,\ldots),+}+\psi^{\be_{\al_1},-}_{(\al_1),+}\psi^{\be_{n-\al_1},\mp}_{(\al_2,\al_3,\ldots),+}\\
&=\psi^{\be_{\al_1},+}_{(\al_1),+}\psi^{\be_{n-\al_1},\mp}_{(\al_2,\al_3,\ldots),+}+\psi^{\be_{\al_1},-}_{(\al_1),+}\psi^{\be_{n-\al_1},\pm}_{(\al_2,\al_3,\ldots),+}\\
&=\psi^{\be_n,-}_{\al,+}
\end{align*}
(identifying conjugacy classes of $\A_{\al_1}\times \A_{n-\al_1}$ so that the conjugacy class labeled by $(((\al_1),+),((\al_2,\al_3,\ldots),+))$ is contained in the $\A_n$ conjugacy class labeled by $\al,+$).

On the other hand if $h(\al)=2$ then $\psi^{\be_n,+}_{\al,+}\not=\psi^{\be_n,-}_{\al,+}$, since otherwise $\psi^{\be_n,\pm}_{\al,+}=\psi^{\be_n}_\al/2=\pm 1/2$. 

Since $\psi^{\be_n}_\al=\pm 2^{\lfloor (h(\al)-1)/2\rfloor}$, in order to prove the lemma it is enough to check that
\[|\psi^{\be_n,+}_{(m,n-m),+}-\psi^{\be_n,-}_{(m,n-m),+}|^2>1=|\psi^{\be_n}_{(m,n-m)}|^2\]
for $n/2<m<n$ odd.

If $m=n-1$ then $D^{\be_n}\da_{\A_{n-1}}\cong D^{\be_{n-1}}$ and so $E^{\be_n}_\pm\da_{\A_{n-1}}\cong E^{\be_{n-1}}_\pm$. From the $n$ odd case we obtain that
\[|\psi^{\be_n,+}_{(n-1,1),+}-\psi^{\be_n,-}_{(n-1,1),+}|^2=|\psi^{\be_{n-1},+}_{(n-1),+}-\psi^{\be_{n-1},-}_{(n-1),+}|^2=n-1>1.\]
So assume now that $m<n-1$. Using the case where $n$ is odd and the fact that
\[E^{\be_n}_\pm\da_{\A_{m}\times \A_{n-m}}\cong(E^{\be_{m}}_+\boxtimes E^{\be_{n-m}}_\pm)\oplus (E^{\be_{m}}_-\boxtimes E^{\be_{n-m}}_\mp)\]
we obtain
\begin{align*}
|\psi^{\be_n,+}_{(m,n-m),+}-\psi^{\be_n,-}_{(n-m,m),+}|^2&=|\psi^{\be_{m},+}_{(m),+}-\psi^{\be_{m},-}_{(m),+}|^2\cdot|\psi^{\be_{n-m},+}_{(n-m),+}-\psi^{\be_{n-m},-}_{(n-m),+}|^2\\
&=(n-m)m\\
&>1.
\end{align*}
\end{proof}

We will now consider characters of spin representations and study their divisibility by 2 or 4. Before doing this we though have to introduce some combinatorial notation.

Let $P=(\la^0,\ldots,\la^k)$ with $\la^0=()$, $\la^j\in\Par_2(|\la^j|)$ for $1\leq j\leq k$ and suppose $\al\in\Paro(n)$. We say that $P$ is an $\overline{\al}$-path if $k=h(\al)$ and, for each $1\leq j\leq k$, $\la^j$ is obtained from $\la^{j-1}$ in one of the following ways:
\begin{itemize}
\item adding a part equal to $\al_j$,

\item adding two parts $x$ and $y$ with $x+y=\al_j$,

\item adding $\al_j$ to one of the parts
\end{itemize}
and then reordering the parts to form a partition. If $P$ is an $\overline{\al}$-path let $a(P)$ be the number of steps $j$ where $\la^j$ is not obtained from $\la^{j-1}$ by adding a part $\al_j$. Also let $P(\la,\al)$ be the set of $\overline{\al}$-paths $P=(\la^0,\ldots,\la^k)$ with $\la^k=\la$.

\begin{lemma}{\cite[Theorem 2]{mo}}\label{brs}.
Let $\la\in\Par_2(n)$ and $\al\in\Paro(n)$. Then
\[\ze^{\la}_\al=\sum_{P\in P(\la,\al)} (-1)^{x(P)}2^{\lfloor a(P)/2\rfloor}\]
for some $x(P)\in\{0,1\}$.
\end{lemma}

\begin{lemma}\label{L290520}
Let $\la\in\Parod(n)$. Then the following statements hold.
\begin{enumerate}[(i)]
\item For $\al\in\Paro(n)$, $2\nmid \ze^{\la}_\al$ if and only if $\al=\la$.

\item If $c\geq 1$ is odd, and $\al\in\Paro(n+2c)$ with $2\nmid\ze^{p(\la,2c)}_\al$ then $\al=p(\hat\la_j,\la_j+2c)$ for some $1\leq j\leq h(\la)$ or $\al=p(\la,2c-e,e)$ for some $1\leq e\leq c$ odd, and in particular 
$\al\unrhd p(\la,c^2)$. In addition, $2\nmid\ze^{p(\la,2c)}_{p(\la,c^2)}$.

\item If $c>d\geq 1$ are both odd and $\al\in\Paro(2c+2d)$ then $2\mid \ze^{(2c,2d)}_\al$. If $4\nmid\ze^{(2c,2d)}_\al$ then $\al=(2c+2d-e,e)$ or $p(2c-e,e,2d-f,f)$ with $e,f$ odd, and in particular $\al\unrhd(c^2,d^2)$.  In addition, $4\nmid\ze^{(2c,2d)}_{(c^2,d^2)}$.
\end{enumerate}
\end{lemma}

\begin{proof}
We will use Lemma \ref{brs} without further comment.

Let $P=(\la^0,\ldots,\la^k)$ be any $\overline{\alpha}$-path (for any $\alpha\in\Paro(n)$). Let $1\leq j\leq k$. If $\la^j$ is obtained from $\la^{j-1}$ by adding a part $\alpha_j$ then the numbers of even parts of $\la^j$ and of $\la^{j-1}$ are equal. If on the other hand $\la^j$ is obtained from $\la^{j-1}$ by adding two parts $x$ and $y$ with $x+y=\al_j$ or by adding $\al_j$ to a non-zero part, then the numbers of even parts of $\la^j$ and of $\la^{j-1}$ differ by 1. So $a(P)\equiv h_2(\la^k)\Md{2}$ and $a(P)\geq h_2(\la^k)$.

Furthermore $2^{\lfloor a(P)/2\rfloor}=1$ if $a(P)=0$ or $1$, while $2^{\lfloor a(P)/2\rfloor}=2$ if $a(P)=2$ or $3$.

To see (i) one then has to observe that if $\la^k=\la$ and $a(P)\leq 1$ then $a(P)=0$. This is only possible if the parts of $\la$ are added one at the time. So there exists a unique path with $a(P)\leq 1$ (as we are assuming that $|\la^j|-|\la^{j-1}|$ is non-increasing) and this path is a $\overline{\la}$-path. We then have that $2\nmid \ze^{\la}_\al$ if and only if $\al=\la$.

Next consider (ii). In this case if $a(P)\leq 1$ then $a(P)=1$. In addition, $a(P)=1$ if and only if either the part $2c$ is added together with one of the parts of $\la$ and all other parts of $\la$ are added separately (each one in only one step) or if the part $2c$ is obtained by first adding a part $e$ and at a later step adding $2c-e$ to $e$ and adding the parts of $\la$ separately as parts at all other steps. In particular if $2\nmid\ze^{p(\la,2c)}_\al$ then $\al=p(\hat\la_j,\la_j+2c)$ for some $1\leq j\leq h(\la)$ or $\al=p(\la,2c-e,e)$ for some $1\leq e\leq c$ odd. Since $(2c-f,f)\unrhd(c^2)$ for any $0\leq f\leq c$, it can then be checked that if $2\nmid\ze^{p(\la,2c)}_\al$ then $\al\unrhd p(\la,c^2)$. To see that $2\nmid\ze^{p(\la,2c)}_{p(\la,c^2)}$ one can check that there is a unique $\overline{p(\la,c^2)}$-path with $\la^k=p(\la,2c)$ and $a(P)=1$ as follows. If $c$ is not a part of $\la$ then $p(\la,c^2)=(\la_1,\ldots,\la_j,c^2,\la_{j+1},\ldots)$ for some $0\leq j\leq h(\la)$ and the unique path with $\la^k=p(\la,2c)$ and $a(P)=1$ is the path obtained by adding the parts $\la_1,\ldots,\la_j$ one at a time as new parts, then adding a part equal to $c$, adding $c$ to the part $c$ and then adding the parts $\la_{j+1},\ldots$ again one a time as new parts. If $c=\la_j$ then $p(\la,c^2)=(\la_1,\ldots,\la_{j-1},c^3,\la_{j+1},\ldots)$. In this case the unique path with $\la^k=p(\la,2c)$ and $a(P)=1$ is the path obtained by adding the parts $\la_1,\ldots,\la_{j-1}$ one at a time as new parts, then adding a part equal to $c$, adding $c$ to the part $c$, the adding again $c$ as a part and then adding the parts $\la_{j+1},\ldots$ again one a time as new parts.

Finally consider (iii). In this case $a(P)\geq 2$ is even as $(2c,2d)$ has two even parts. It then follows that $2\mid \ze^{(2c,2d)}_\al$ for every $\al$. Furthermore if $a(P)=2$ then $P$ is of one of the following forms: $((),p(2x,2y-e),(2c,2d))$ for $\{x,y\}=\{c,d\}$ and some $e$ odd, or of one of the forms $((),(e),(2x),p(2x,f),(2c,2d))$ or $((),(e),p(e,f),p(2x,e),(2c,2d))$ or $((),(e),p(e,f),p(2x,f),(2c,2d))$ for $\{x,y\}=\{c,d\}$ and some $e$ and $f$ odd. In particular if $4\nmid\ze^{(2c,2d)}_\al$ then $\al=(2c+2d-e,e)$ or $p(2c-e,e,2d-f,f)$ with $e,f$ odd and then $\al\unrhd(c^2,d^2)$. To see that $4\nmid\ze^{(2c,2d)}_{(c^2,d^2)}$, note that the unique $\overline{(c^2,d^2)}$-path with $\la^k=(2c,2d)$ is the path $((),(c),(2c),(2c,d),(2c,2d))$.
\end{proof}

\section{Reduction results}\label{s6}

In this section we will prove some results which give restrictions on partitions $\la$ and $\nu$ for which $E^\la_{(\pm)}\otimes E^{\be_n}_{(\pm)}\cong E^\nu_{(\pm)}$.

\begin{theor}\label{T280520}
Let $p=2$, $n\geq 5$ and $\la,\nu\in\Par_2(n)$ be such that $E^\la_{(\pm)}$ is not 1-dimensional and $E^\la_{(\pm)}\otimes E^{\be_n}_{(\pm)}\cong E^\nu_{(\pm)}$. Then $\nu\not\in\Parinv_2(n)$. Furthermore, if $\la$ is the double of a partition then $\nu$ is not the double of any partition.
\end{theor}

\begin{proof}
First assume that $\la$ is not the double of a partition. Then $\la\not\in\Parinv_2(n)$, so $E^\la=D^\la\da^{\s_n}_{\A_n}$ extends to $\s_n$. If $n\equiv 2\Md{4}$ then $E^{\be_n}$ also extends to $\s_n$. So the same holds for $E^\la\otimes E^{\be_n}$ and so $\nu\not\in\Parinv_2(n)$ (actually in this case $E^\la\otimes E^{\be_n}$ is never irreducible, thanks to \cite[Theorem 1.1]{m1}). If on the other hand $n\not\equiv 2\Md{4}$ then $\nu\not\in\Parinv_2(n)$ by \cite[Theorem 13.2]{m4}.

Assume now that $\la=\dbl(\overline{\la})$ for a partition $\overline{\la}$. In this case it is enough to check that $\nu$ is not the double of a partition (since in this case $\nu\not\in\Parinv_2(n)$). Assume instead that $\nu=\dbl(\overline{\nu})$ for some partition $\overline{\nu}$. Because $E^{\dbl(\overline{\la})}_{(\pm)}\otimes E^{\be_n}_{(\pm)}\cong E^{\dbl(\overline{\nu})}_{(\pm)}$ we then have that $D^{\dbl(\overline{\nu})}\subseteq D^{\dbl(\overline{\la})}\otimes D^{\be_n}$ and then also that $D^{\dbl(\overline{\la})}\subseteq D^{\dbl(\overline{\nu})}\otimes D^\be_n$, since
\[\Hom_{\s_n}(D^{\dbl(\overline{\nu})},D^{\dbl(\overline{\la})}\otimes D^{\be_n})\cong\Hom_{\s_n}(D^{\dbl(\overline{\la})},D^{\dbl(\overline{\nu})}\otimes D^{\be_n}).\]
From Lemma \ref{L280520_3} we then have that $\dbl(\overline{\la})=\dbl(\overline{\nu})$, leading to a contradiction since $E^{\be_n}_{(\pm)}$ is not 1-dimensional by the assumption $n\geq 5$.
\end{proof}

\begin{theor}\label{T080620_2}
Let $p=2$ and $\la\in\Par_2(n)$, where $n$ is odd. If $E^\la_{(\pm)}\otimes E^{\be_n}_\pm\cong E^\nu$ with $\nu\in\Par_2(n)\setminus\Parinv_2(n)$ then $D^\nu$ is a composition factor of some $S((n-a,a))$ with $1\leq a\leq (n-1)/2$. In particular $\nu=(n-b,b)$ or $\dbl(n-b,b)$ and $h(\la)\leq 4$.
\end{theor}

\begin{proof}
Recall from Section \ref{s2} that if $\ga$ is a composition then $p(\ga)$ is the partition obtained by reordering the parts of $\ga$.

By assumption $\psi^{\la(,\pm)}\psi^{\be_n,\pm}=\psi^\nu$. So if $\al\in\Paro(n)$ and $2\nmid\psi^\nu_{\al}=\psi^\nu_{\al(,+)}$ then $2\nmid\psi^{\be_n,\pm}_{\al(,+)}$ and so $\al=(n)$ or $h(\al)=3$ by Lemma \ref{cbs}. So by Lemma \ref{L290520} there exist $d,d_{x,y},d_{x,y,z}\in\N$ such that if
\[\psi=\psi^\nu+d\ze^{(n)}+\sum_{{x,y\text{ odd}:}\atop{x+2y=n}}d_{x,y}\ze^{p(x,2y)}+\sum_{{x>y>z\text{ odd}:}\atop{x+y+z=n}}d_{x,y,z}\ze^{(x,y,z)}\]
then $2\mid\psi_{\al}$ for every $\al\in\Paro(n)$.

To see this note first that by the above remark on $\psi^\nu$ as well as by Lemma \ref{L290520}, if any of $\psi^\nu_\al$, $\ze^{(n)}_\al$, $\ze^{p(x,2y)}_\al$ or $\ze^{(x,y,z)}_\al$ is odd then $h(\al)\leq 3$ and then $h(\al)=1$ or $3$ as $n$ is odd as are the parts of $\al$. Next note that if $\al=(x,y,z)$ with $x>y>z$ all odd then $\ze^{(x,y,z)}_\al$ is odd and if $\ze^{(x,y,z)}_\be$ is odd then $\be\unrhd\al$. Similarly if $\al=(x,y,y)$ with $x\geq y$ both odd or if $\al=(y,y,x)$ with $y\geq x$ both odd then $\ze^{p(x,2y)}_\al$ is odd and if $\ze^{p(x,2y)}_\be$ is odd then $\be\unrhd\al$. Furthermore, $\ze^{(n)}_\al$ is odd if and only if $\al=(n)$ (each of these facts holds by Lemma \ref{L290520}).

It is thus possible to choose coefficients $d,d_{x,y},d_{x,y,z}\in\N$ such that $2\mid\psi_{\al}$ for every $\al\in\Paro(n)$ (this can be done recursively considering partitions $(x,y,z)$, $(x,y,y)$, $(y,y,x)$ and $(n)$ in any order which is a refinement of the reverse dominance order and choosing the corresponding coefficient $d_{x,y,z}$, $d_{x,y}$ or $d$ accordingly).

So $\psi=2\overline{\psi}$ for some Brauer character $\overline{\psi}$, as the 2-Brauer characters of the modules $D^\psi$ are linearly independent modulo 2. In particular $D^\nu$ is a composition factor with odd multiplicity of the reduction modulo 2 of one of the modules $S((n),0)$, $S(p(x,2y),\pm)$ or $S((x,y,z),0)$ (for some $x,y,z$).

If $\mu\in\Par_2(n)$ and $h(\mu)$ is odd, then $S(\mu,0)$ splits when restricted to $\A_n$. Since $\nu\not\in\Parinv_2(n)$, it follows that $[S(\mu,0):D^\nu]$ is even in this case. So $D^\nu$ must be a composition factor of some $S(p(x,2y))$. The theorem then follows from Lemma \ref{L280520_2}, up to the bound on $h(\la)$. To obtain this bound, note that $D^\la\subseteq D^\nu\otimes D^\be_n$, since  $D^\nu\subseteq D^\la\otimes D^{\be_n}$ (arguing similarly to the previous lemma). Since $D^\nu$ is a composition factor of $S(p(x,2y))$ for some $x,y$, we then have that $D^\la$ is a composition factor of $S(p(x,2y))\otimes S((n))$. In particular $h(\la)\leq 2 h(p(x,2y))=4$ by Lemma \ref{bsspin}.
\end{proof}

\begin{theor}\label{T050620_2}
Let $p=2$, $n\equiv 0\Md{4}$ and $\la\in\Par_2(n)$ with $E^\la_{(\pm)}$ not 1-dimensional. If $E^\la_{(\pm)}\otimes E^{\be_n}_\pm\cong E^\nu$ with $\nu\in\Par_2(n)\setminus\Parinv_2(n)$ then one of the following holds:
\begin{itemize}
\item $D^\nu$ is a composition factor of some $S((n-a,a))$ with $1\leq a<n/2$, in particular $\nu=(n-b,b)$ or $\dbl(n-b,b)$ with $1\leq b\leq n/2-2$,

\item $D^\nu$ is a composition factor of some $S(p(n-a-b,a,b))$ with $a\equiv b\equiv \pm 1\Md{4}$, in particular $h(\nu)\leq 4$ or $\nu=\dbl(p(n-c-d,c,d))$ with $c\equiv d\equiv \pm 1\Md 4$.
\end{itemize}
Furthermore, $h(\la)\leq 4$ in the first case and $h(\la)\leq 6$ in the second case.
\end{theor}

\begin{proof}
We will use Lemmas \ref{cbs} and \ref{L290520} without further comment.
 
 By assumption $\psi^{\la(,\pm)}\psi^{\be_n,\pm}=\psi^\nu$. So if $\al\in\Paro(n)$ and $2\nmid\psi^\nu_\al=\psi^\nu_{\al(,+)}$ then $2\nmid\psi^{\be_n,\pm}_{\al(,+)}$ and so $h(\al)=2$ or 4.

If $n/4$ is odd, then $4\mid\ze^{(n-x,x)}_{((n/4)^4)}$ if $x<n/2$ is odd by Lemma \ref{brs}. In addition, if $h(\al)\geq 6$ then $4\mid\ze^{(n-x,x)}_\al$, again by Lemma \ref{brs}. So if
\[\ze=\ze^{(n)}+\sum_{x<n/2\text{ odd}}\ze^{(n-x,x)}\]
then $2\mid\ze_{\al(,+)}$ for every $\al\in\Paro(n)$, and $4\mid\ze_{\al(,+)}$ if $h(\al)\geq 6$ and $4\nmid\ze_{((n/4)^4)}$. Note that then $\ze=2\overline{\ze}$ for some Brauer character $\overline{\ze}$ with $2\nmid\overline{\ze}_{((n/4)^4)}$.

In addition, if $x>y$ are odd then $\ze^{(2x,2y)}=2\overline{\ze}^{(2x,2y)}$ where $\overline{\ze}^{(2x,2y)}$ is a Brauer character with $2\nmid\overline{\ze}^{(2x,2y)}_{(x^2,y^2)}$.

Similarly to the previous theorem one can check that there exist $d,d_{x,y},d_{x,y,z},d_{x,y,z,w}\in\N$ such that if
\begin{align*}
\psi=&\psi^\nu+\de_{8\nmid n}d\overline{\ze}+\sum_{{x>y\text{ odd}:}\atop{x+y=n/2}}d_{x,y}\overline{\ze}^{(2x,2y)}+\sum_{{x>y\text{ odd}:}\atop{x+y=n}}d_{x,y}\ze^{(x,y)}\\
&+\sum_{{x,y,z\text{ odd}:}\atop{x+y+2z=n}}d_{x,y,z}\ze^{p(x,y,2z)}+\sum_{{x>y>z>w\text{ odd}:}\atop{x+y+z+w=n}}d_{x,y,z,w}\ze^{(x,y,z,w)}
\end{align*}
then $\psi=2\overline{\psi}$ for some Brauer character $\overline{\psi}$. We can then conclude similarly to the previous theorem, using the fact that $\ze^{(n)}=\psi^{\be_n}$ by Lemma \ref{bsspin} and $\nu\not=\be_n$ (by dimension).
\end{proof}

\section{Permutation modules}\label{s7}

We start this section with the following results on the structure of some permutation modules when $n$ is odd.

\begin{lemma}\label{L080620_3}
Let $p=2$ and suppose $n\geq 5$ is odd. Then there exist modules $\overline{Y}_2$ and $Y_{1^2}$ such that
\[M_1\cong D_1\oplus D_0,\quad M_2\cong D_1\oplus \overline{Y}_2,\quad M_{1^2}\cong D_1^{\oplus 2}\oplus Y_{1^2}\]
with $Y_{1^2}\sim \overline{Y}_2|\overline{Y}_2$. Furthermore:
\begin{itemize}
\item
if $n\equiv 1\Md{4}$ then $\overline{Y}_2\cong D_0|D_2|D_0$ and $Y_{1^2}\cong D_0|D_2|D_0|D_0|D_2|D_0$,

\item
if $n\equiv 3\Md{4}$ then $\overline{Y}_2\cong D_0\oplus D_2$, $\soc(Y_{1^2})\cong\hd(Y_{1^2})\cong \overline{Y}_2$ and both $Y_{1^2}/D_0$ and $Y_{1^2}/D_2$ have simple socle.
\end{itemize}
\end{lemma}

\begin{proof}
This follows from \cite[Tables 1 and 2]{MO}, except for $Y_{1^2}\sim \overline{Y}_2|\overline{Y}_2$ for $n\equiv 1\Md{4}$, which holds by block decomposition since
\[M_{1^2}\cong \1\ua_{\s_{n-2}}^{\s_{n-2,2}}\ua_{\s_{n-2,2}}^{\s_n}\sim (\1|\1)\ua_{\s_{n-2,2}}^{\s_n}\sim M_2|M_2.\]
\end{proof}

Note that $Y_{1^2}$ is indecomposable in either case. Since $S_{1^2}\subseteq Y_{1^2}$, this module is the unique indecomposable direct summand of $M_{1^2}$ which is not isomorphic to a direct summand of $M_0$, $M_1$ or $M_2$ and is therefore a Young module. The module $\overline{Y}_2$ however is not always indecomposable.

We will now use the above lemma to study the submodule structure of the endomorphism rings of some irreducible modules.

\begin{lemma}\label{L080620_4}
Let $p=2$ and suppose $n\geq 5$ is odd, and let $Y_{1^2}$ be as in Lemma \ref{L080620_3}. Then $D_0\oplus D_1\subseteq\End_F(D^{\be_n})$. In addition, if $n\equiv 1\Md{4}$ then $Y_{1^2}/(D_0|D_2)\subseteq\End_F(D^{\be_n})$, while if $n\equiv 3\Md{4}$ then $Y_{1^2}/D_2\subseteq\End_F(D^{\be_n})$.
\end{lemma}

\begin{proof}
Since $D_0\cong\1$
\[\dim\Hom_{\s_n}(D_0,\End_F(D^{\be_n}))=1\]
and so $D_0\subseteq\End_F(D^{\be_n})$.

Using Lemma \ref{L180121} to study composition factors together with \ref{Lemma39} to study the submodule structure, we have that
\begin{align*}
D^{\be_n}\da_{\s_{n-1}}&\cong D^{\be_{n-1}}|D^{\be_{n-1}},\\
D^{\be_n}\da_{\s_{n-2}}&\cong (D^{\be_{n-2}})^{\oplus 2}.
\end{align*}
Considering the only non-zero block component of $D^{\be_n}\da_{\s_{n-2}}$, by Lemma \ref{L090620} we then also have that
\[D^{\be_n}\da_{\s_{n-2,2}}\cong D^{\be_{n-2}}\boxtimes M^{(1^2)}.\]

From Lemma \ref{L080620_3} we then have that
\begin{align*}
\dim\Hom_{\s_n}(D_1,\End_F(D^{\be_n}))&=1,\\
\dim\Hom_{\s_n}(\overline{Y}_2,\End_F(D^{\be_n}))&=1,\\
\dim\Hom_{\s_n}(Y_{1^2},\End_F(D^{\be_n}))&=2.
\end{align*}
So $D_1\subseteq\End_F(D^{\be_n})$. Furthermore, there exists a quotient $M$ of $Y_{1^2}\sim\overline{Y}_2|\overline{Y}_2$ which is not also a quotient of $\overline{Y}_2$ and which is isomorphic to a submodule of $\End_F(D^{\be_n})$.

If $n\equiv 1\Md{4}$ then $M\in\{Y_{1^2},Y_{1^2}/D_0,Y_{1^2}/(D_0|D_2)\}$. In the first two cases $\overline{Y}_2\subseteq \End_F(D^{\be_n})$ or $D_0\oplus (\overline{Y}_2/D_0)\subseteq \End_F(D^{\be_n})$ (as $D_0\cong \1$). Since $\overline{Y}_2\cong D_0|D_2|D_0$, in either case we get a contradiction to $\dim\Hom_{\s_n}(\overline{Y}_2,\End_F(D^{\be_n}))=1$. So $M=Y_{1^2}/(D_0|D_2)$.

If $n\equiv 3\Md{4}$ then $M\in\{Y_{1^2},Y^{1^2}/D_0,Y_{1^2}/D_2\}$. In the first two cases $D_2\subseteq\End_F(D^{\be_n})$ and then $\overline{Y}_2\cong D_0\oplus D_2\subseteq\End_F(D^{\be_n})$, again contradicting the fact that $\dim\Hom_{\s_n}(\overline{Y}_2,\End_F(D^{\be_n}))=1$. So $M=Y_{1^2}/D_2$.
\end{proof}

\begin{lemma}\label{L080620_5}
Let $p=2$ and suppose $n\geq 5$ is odd, and let $Y_{1^2}$ be as in Lemma \ref{L080620_3}. If $\la\in\Par_2(n)\setminus\Parinv_2(n)$ has exactly two normal nodes, then $D_0\oplus D_1\subseteq\End_F(D^\la)$, and either $Y_{1^2}$ or $Y_{1^2}/D_0$ is contained in $\End_F(D^\la)$.
\end{lemma}

\begin{proof}
We will use Lemma \ref{Lemma39} without further comment.

If $M\cong \1|\1$ then $M\cong \1\ua_{\A_n}^{\s_n}$ (for example from \cite[Lemma 3.24]{m1}). So, since $\la\not\in\Parinv_2(n)$,
\begin{align*}
\dim\Hom_{\s_n}(M,\End_F(D^\la))&=\dim\Hom_{\s_n}(\1\ua_{\A_n}^{\s_n},D^\la\otimes D^\la)\\
&=\dim\Hom_{\s_n}((\1\ua_{\A_n}^{\s_n})\otimes D^\la,D^\la)\\
&=\dim\Hom_{\s_n}(D^\la\da_{\A_n}\ua_{\A_n}^{\s_n},D^\la)\\
&=\dim\Hom_{\s_n}(D^\la\da_{\A_n}\ua_{\A_n}^{\s_n},D^\la)\\
&=\dim\End_{\A_n}(D^\la\da_{\A_n})\\
&=\dim\End_{\A_n}(E^\la)\\
&=1.
\end{align*}
In particular $M\not\subseteq\End_F(D^\la)$.

We have that $\dim\End_{\s_{n-1}}(D^\la\da_{\s_{n-1}})=2$. As $D_0\cong\1$, it follows by Lemma \ref{L080620_3} that
\[\dim\Hom_{\s_n}(D_1,\End_F(D^\la))=1\]
and then $D_0\oplus D_1\subseteq\End_F(D^\la)$. We will prove that
\[\dim\End_{\s_{n-2}}(D^\la\da_{\s_{n-2}})\geq \dim\End_{\s_{n-2,2}}(D^\la\da_{\s_{n-2,2}})+2.\]
By Lemma \ref{L080620_3} it will then follow that there exists a quotient of $Y_{1^2}\sim\overline{Y}_2|\overline{Y}_2$ which is not also a quotient of $\overline{Y}_2$ and which is isomorphic to a submodule of $\End_F(D^\la)$. Since $\1|\1\not\subseteq\End_F(D^\la)$, the lemma then follows (using Lemma \ref{L080620_3} to study quotients of $Y_{1^2}$).

{\bf Case 1:} assume that $\eps_i(\la)=2$ and $\eps_{1-i}(\la)=0$. Then $D^\la\da_{\s_{n-2,2}}\cong A\oplus B$ with $A\da_{\s_{n-2}}\cong (D^{\tilde e_i^2\la})^{\oplus 2}$ and $B\da_{\s_{n-2}}\cong e_{1-i}e_iD^\la$. Furthermore, $A\cong D^{\tilde e_i^2\la}\boxtimes M^{(1^2)}$ from Lemma \ref{L090620}. It follows from $M^{(1^2)}\cong \1|\1$ that
\begin{align*}
&\dim\End_{\s_{n-2}}(D^\la\da_{\s_{n-2}})-\dim\End_{\s_{n-2,2}}(D^\la\da_{\s_{n-2,2}})\\
&\geq\End_{\s_{n-2}}(A\da_{\s_{n-2}})-\dim\End_{\s_{n-2,2}}(A)\\
&=4-2=2.
\end{align*}

{\bf Case 2:} assume that $\eps_0(\la),\eps_1(\la)=1$. Note that in this case
\[D^\la\cong e_{1-i}e_iD^\la+e_ie_{1-i}D^\la\cong e_{1-i}D^{\tilde e_i\la}+e_iD^{\tilde e_{1-i}\la}.\]
Let $A$ be the top normal node of $\la$, $B$ the bottom normal node of $\la$ and $i$ be the residue of $A$. By assumption $B$ then has residue $1-i$. From \cite[Lemma 4.4]{m1} it can be checked that $\eps_{1-i}(\tilde e_i\la)=3$ and $\eps_i(\tilde e_{1-i}\la)=x$ with $x\geq 1$. By \cite[Lemma 4.8]{m1} it then follows that
\begin{align*}
&\dim\End_{\s_{n-2}}(D^\la)\\
&=\dim\End_{\s_{n-2}}(e_{1-i}e_iD^\la)+\dim\Hom_{\s_{n-2}}(e_{1-i}e_iD^\la,e_ie_{1-i}D^\la)\\
&\hspace{11pt}+\dim\Hom_{\s_{n-2}}(e_ie_{1-i}D^\la,e_{1-i}e_iD^\la)+\dim\End_{\s_{n-2}}(e_ie_{1-i}D^\la)\\
&=\dim\End_{\s_{n-2}}(e_{1-i}D^{\tilde e_i\la})+\dim\End_{\s_{n-2}}(e_iD^{\tilde e_{1-i}\la})\\
&\hspace{11pt}+2\dim\Hom_{\s_{n-2}}(e_{1-i}e_iD^\la,e_ie_{1-i}D^\la)\\
&\geq 3+x+2\cdot 1=5+x.
\end{align*}
Let $\la^1:=\tilde e_{1-i}\tilde e_i\la$ and $\la^2:=\tilde e_i\tilde e_{1-i}\la$. Note that $A$ is on the first row, while $B$ is below the third row. By the proof of \cite[Lemma 4.4]{m1}, $\la^1=(\la_A)_B$. Since $A$ is on the first row it is also the top normal node in $\la_B$. Thus
\begin{align*}
[e_{1-i}e_iD^\la:D^{\la^1}]&=[e_{1-i}D^{\tilde e_i\la}:D^{\tilde e_{1-i}\tilde e_i\la}]=\eps_{1-i}(\tilde e_i\la)=3,\\
[e_ie_{1-i}D^\la:D^{\la^1}]&=[e_iD^{\tilde e_{1-i}\la}:D^{(\tilde e_{1-i}\la)_A}]=1.
\end{align*}
Let $W\subseteq e_ie_{i-1}D^\la$ be maximal  with $[W:D^{\la^1}]=0$. Since $\soc(e_{1-i}e_iD^\la)\cong D^{\la^1}$, $W$ is the maximal submodule of $D^\la\da_{\s_{n-2}}$ with no composition factor $D^{\la^1}$. Note that if $N$ is a $F\s_{n-2,2}$-module and $\mu\in\Par_p(n-2)$, then $D^\mu\boxtimes D^{(2)}\subseteq \soc(N)$ if and only if $D^\mu\subseteq \soc(N\da_{\s_{n-2}})$. By maximality, it follows that $W$ is stable under $\s_{n-2,2}$, so from now on we will consider $W$ also as $F\s_{n-2,2}$-module. Let $Y:=D^\la\da_{\s_{n-2,2}}/W$. Since $[e_ie_{1-i}D^\la:D^{\la^1}]=1$, we have that $\soc((e_ie_{1-i}D^\la)/W)\cong D^{\la^1}$. In particular $\soc(Y\da_{\s_{n-2}})\cong (D^{\la^1})^{\oplus 2}$ and then $\soc(Y)\cong(D^{\la^1}\boxtimes D^{(2)})^{\oplus a}$ with $1\leq a\leq 2$.

{\bf Case 2.1:} assume that $\soc(Y)\cong(D^{\la^1}\boxtimes D^{(2)})^{\oplus 2}$. Let $A$ be the $\s_{n-2}$-module with $W\subseteq A\subseteq e_ie_{1-i}D^\la$ and $A/W=\soc((e_ie_{1-i}D^\la)/W)$. Since $\soc(Y)\cong(D^{\la^1}\boxtimes D^{(2)})^{\oplus 2}$, $A$ is stable under $\s_{n-2,2}$. Let $B\subseteq D^\la\da_{\s_{n-2,2}}$ maximal with $A\subseteq B$ and $[B:D^{\la^1}\boxtimes D^{(2)}]=[A:D^{\la^1}\boxtimes D^{(2)}]$.

Note that $B\subseteq e_ie_{1-i}D^\la$ (as $\s_{n-2}$-module), since $\soc(e_{1-i}e_iD^\la)\cong D^{\la^1}$ and $A\subseteq e_ie_{1-i}D^\la$. In addition, since $[A:D^{\la^1}]=[e_ie_{1-i}D^\la:D^{\la^1}]=1$ we have by maximality of $B$ that $B\supseteq e_ie_{1-i}D^\la$.

So $B=e_ie_{1-i}D^\la=e_{1-i}D^\la$ (as vector space) and then $B$ is stable under both $\s_{n-1}$ and $\s_{n-2,2}$, leading to a contradiction since $D^\la$ is irreducible as an $F\s_n$-module.

{\bf Case 2.2:} assume that $\soc(Y)\cong D^{\la^1}\boxtimes D^{(2)}$. Note that $[W:D^{\la^2}\boxtimes D^{(2)}]\leq x-1$, since $\hd((e_ie_{1-i}D^\la)/W)\cong D^{\la^2}$ and $[e_ie_{1-i}D^\la:D^{\la^1}]=1$ (so that $e_ie_{1-i}D^\la\not\subseteq W$). We have that $D^\la\da_{\s_{n-2,2}}\sim W|Y$ with $W$ and $Y$ having simple socles isomorphic to $D^{\la^2}\boxtimes D^{(2)}$ and $D^{\la^1}\boxtimes D^{(2)}$ respectively. By definition of $W$ we then have that
\begin{align*}
&\dim\End_{\s_{n-2,2}}(D^\la\da_{\s_{n-2,2}})\\
&\leq\dim\Hom_{\s_{n-2,2}}(D^\la\da_{\s_{n-2,2}},W)+\dim\Hom_{\s_{n-2,2}}(D^\la\da_{\s_{n-2,2}},Y)\\
&=\dim\Hom_{\s_{n-2,2}}(W^*,W)+\dim\Hom_{\s_{n-2,2}}(D^\la\da_{\s_{n-2,2}},Y)\\
&\leq[W:D^{\la^2}\boxtimes D^{(2)}]+[D^\la\da_{\s_{n-2,2}}:D^{\la^1}\boxtimes D^{(2)}]\\
&\leq(x-1)+4=3+x.
\end{align*}
Thus 
\[\dim\End_{\s_{n-2}}(D^\la\da_{\s_{n-2}})\geq \dim\End_{\s_{n-2,2}}(D^\la\da_{\s_{n-2,2}})+2.\]
\end{proof}

\begin{lemma}\label{L090720_2}
Let $p=2$ and suppose $n$ is odd. Assume that $\la\in\Par_2(n)\setminus\Parinv_2(n)$ has exactly two normal nodes. Then
\[\dim\Hom_{\s_n}(\End_F(D^\la),\End_F(D^{\be_n}))\geq 3.\]
\end{lemma}

\begin{proof}
From Lemmas \ref{L080620_4} and \ref{L080620_5} we have that
\[D_0\oplus D_1\subseteq\End_F(D^\la),\End_F(D^{\be_n}).\]
Further if $Y_{1^2}$ is as in Lemma \ref{L080620_3}, then $Y_{1^2}/M\subseteq\End_F(D^\la)$ with $M=0$ or $D_0$ and $Y_{1^2}/N\subseteq\End_F(D^{\be_n})$ with $N=D_0|D_1$ if $n\equiv 1\Md 4$ or $N=D_0$ if $n\equiv 3\Md 4$. The lemma follows.
\end{proof}

\begin{lemma}\label{L240720}
Let $p=2$, $n$ be odd and $Y_{1^2}$ be as in Lemma \ref{L080620_3}. If $n\equiv 1\Md{4}$ let $c:=2$, while if $n\equiv 3\Md{4}$ let $c:=3$. If $\la\in\Par_2(n)$ with $\eps_i(\la)=3$ and $\eps_{1-i}(\la)=0$ for some residue $i$, then $D_0\oplus D_1^{\oplus 2}\oplus(Y_{1^2}/D_0)^{\oplus c}\subseteq\End_F(D^\la)$. 
\end{lemma}

\begin{proof}
We will use Lemma \ref{Lemma39} and \ref{L080620_3} without further comment. Again from $D_0\cong \1$ we have that
\[D_0\subseteq\End_F(D^\la),\End_F(D^{\be_n}).\]
Note that $n\geq 6$, since $\la$ has three normal nodes. Since $\eps_i(\la)=3$ and $\eps_{1-i}(\la)=0$ we have that $\dim\End_{\s_{n-1}}(D^\la)=3$. So $\dim\Hom_{\s_n}(D_1,\End_F(D^\la))=2$ and it is then enough to prove that $(Y_{1^2}/D_0)^{\oplus c}\subseteq\End_F(D^\la)$, as $\soc(Y_{1^2}/D_0)\cong D_2$.

From Lemma \ref{L090620} we have that $D^\la\da_{\s_{n-2,2}}\cong A\oplus B$ with $A\cong e_i^{(2)}D^\la\boxtimes M^{(1^2)}$ and $B\da_{\s_{n-2}}\cong e_{1-i}e_iD^\la$. In particular
\begin{align*}
&\dim\End_{\s_{n-2}}(D^\la)-\dim\End_{\s_{n-2,2}}(D^\la\da_{\s_{n-2,2}})\\
&\geq\dim\End_{\s_{n-2}}(A\da_{\s_{n-2}})-\dim\End_{\s_{n-2,2}}(A)\\
&=12-6
\end{align*}
and then
\begin{align*}
\dim\Hom_{\s_n}(Y_{1^2},\End_F(D^\la))&\geq\dim\Hom_{\s_n}(\overline{Y}_2,\End_F(D^\la))+4\\
\dim\Hom_{\s_n}(Y_{1^2}/D_0,\End_F(D^\la))&\geq\dim\Hom_{\s_n}(\overline{Y}_2,\End_F(D^\la))+3.
\end{align*}
The lemma then follows from $Y_{1^2}/D_0\sim X|\overline{Y}_2$ with $X\cong D_2|D_0$ if $n\equiv 1\Md{4}$ or $X\cong D_2$ if $n\equiv 3\Md{4}$.
\end{proof}

\begin{lemma}\label{L240720_2}
Let $p=2$ and $n$ be odd. If $n\equiv 1\Md{4}$ let $c:=2$, while if $n\equiv 3\Md{4}$ let $c:=3$. If $\la\in\Par_2(n)$ with $\eps_i(\la)=3$ and $\eps_{1-i}(\la)=0$ for some residue $i$, then
\[\dim\Hom_{\s_n}(\End_F(D^\la),\End_F(D^{\be_n}))\geq 3+c.\]
\end{lemma}

\begin{proof}
This follows from Lemmas \ref{L080620_4} and \ref{L240720} (as in the proof of Lemma \ref{L090720_2}).
\end{proof}

\section{Double split case}\label{s8}

We will now consider tensor products of two splitting modules and show that such products are almost never irreducible.

\begin{theor}\label{T050620}
Suppose $p=2$ and $\la\in\Parinv_2(n)$, with $n\not\equiv 2\Md{4}$. If $E^\la_\pm\otimes E^{\be_n}_\pm$ is irreducible, then
\begin{itemize}
\item
if $\alpha\in\Parod(n)$ and $\psi^{\la,+}_{\al,+}\not=\psi^{\la,-}_{\al,+}$ then $h(\al)\leq 2$,

\item
one of the following holds:
\begin{itemize}
\item
$\la=\be_n$,

\item
$n\equiv 0\Md{4}$ and $\la=\dbl(n-a,a)$ with $a\leq n/2-3$ odd.
\end{itemize}
\end{itemize}
\end{theor}

\begin{proof}
We will use Lemma \ref{cbs} without further comment. From Theorem \ref{T280520} we have that $E^\la_\pm\otimes E^{\be_n}_\pm\cong E^\nu$ with $\nu\not\in\Parinv_2(n)$. In particular $\psi^{\la,\pm}\psi^{\be_n,\pm}=\psi^\nu=\psi^{\la,\mp}\psi^{\be_n,\mp}$. Since $\psi^{\be_n,\pm}_{\al,+}\not=0$ for any $\al\in\Parod(n)$, it follows that if $\psi^{\la,+}_{\al,+}\not=\psi^{\la,-}_{\al,+}$ then $\psi^{\be_n,+}_{\al,+}\not=\psi^{\be_n,-}_{\al,+}$ and so $h(\al)\leq 2$.

If $n$ is odd then $\al=(n)$ and so $\psi^{\la,+}-\psi^{\la,-}=c(\psi^{\be_n,+}-\psi^{\be_n,-})$ for some $c\in\Q$. So $\la=\be_n$.

So we may assume that $n\equiv 0\Md{4}$. Let $\ze^{(n-b,b),\pm}$ be the Brauer characters of the reductions modulo 2 of the two composition factors $T((n-b,b,\pm)$ of $S((n-b,b),0)\da_{\widetilde{\A}_n}$. By \cite[Theorem 7.1]{s} or \cite[p. 235]{s5} (to study $\ze^{(n-b,b),+}-\ze^{(n-b,b),-}$) there exist $c,c_b\in\Q$ with
\begin{align*}
\psi^{\la,+}-\psi^{\la,-}=c(\psi^{\be_n,+}-\psi^{\be_n,-})+\sum_{b\leq n/2-3\text{ odd}}c_b(\ze^{(n-b,b),+}-\ze^{(n-b,b),-})
\end{align*}
(first find $c$ based on $\psi^{\la,+}_{(n/2+1,n/2-1),+}-\psi^{\la,-}_{(n/2+1,n/2-1),+}$ and then find the $c_b$).

If $c_b=0$ for all $b$ then $\la=\be_n$, so we may assume that there exists $b$ with $c_b\not=0$. Let $a$ be maximal with $c_a\not=0$. By Lemma \ref{L280520_2} $D^{\dbl(n-a,a)}$ is a composition factor of $S((n-a,a))$ with multiplicity 1. It follows that $E^{\dbl(n-a,a)}_+$ is a composition factor of the reduction modulo 2 of exactly one $T((n-a,a),\pm)$ and $E^{\dbl(n-a,a)}_-$ of the other one (as $(E^{\dbl(n-a,a)}_\pm)^\si\cong E^{\dbl(n-a,a)}_\mp$ and $T((n-a,a),\pm)^\si\cong T((n-a,a),\mp)$ for any $\si\in\widetilde\s_n\setminus\widetilde\A_n$). Further, again by Lemma \ref{L280520_2}, $D^{\dbl(n-a,a)}$ is not a composition factor of $S((n-b,b))$ with $b<a$. Since $\dbl(n-a,a)\not=\be_n$, it follows that
\[c(\psi^{\be_n,+}-\psi^{\be_n,-})+\sum_{b\leq n/2-3\text{ odd}}c_b(\ze^{(n-b,b),+}-\ze^{(n-b,b),-})=\sum_{\mu\in\Parinv_2(n)}d_\mu(\psi^{\mu,+}-\psi^{\mu,-})\]
with $d_{\dbl(n-a,a)}\not=0$ and then that $\la=\dbl(n-a,a)$.
\end{proof}

\begin{theor}\label{T050620_3}
Let $p=2$, $n\equiv 0\Md{4}$ and $3\leq a\leq n/2-3$ be odd. Then $E^{\dbl(n-a,a)}_\pm\otimes E^{\be_n}_\pm$ is not irreducible.
\end{theor}

\begin{proof}
Assume for a contradiction that $E^{\dbl(n-a,a)}_\pm\otimes E^{\be_n}_\pm$ is irreducible. Let $\la=\dbl(n-a,a)$ and $\nu\in\Par_2(n)\setminus\Parinv_2(n)$ with $E^\la_\pm\otimes E^{\be_n}_\pm\cong E^\nu$ (by Theorem \ref{T280520}). By Lemma \ref{L280520_2} for any $1\leq b<n/2$ with $\dbl(n-b,b)\in\Par_2(n)$ we have that
\[[S((n-b,b))]=\sum_{1\leq x\leq b}e_x[D^{\dbl(n-x,x)}]+\sum_{x<n/2}f_x[D^{(n-x,x)}]\]
with $e_x,f_x\geq 0$ and $e_b>0$. So there exist $c_b,d_b,d'_B\in\Q$ with $c_a>0$ such that
\begin{align*}
[D^\la]&=\sum_{b\leq a}c_b[S((n-b,b))]+\sum_{b<n/2}d'_b[D^{(n-b,b)}]\\
&=\sum_{b\leq a}c_b[S((n-b,b))]+\sum_{b<n/2}d_b[S^{(n-b,b)}].
\end{align*}
Further by Lemma \ref{branchingbs}, $[S((n))]=[D^{\be_n}]$. So 
\[[D^\la\otimes D^{\be_n}]=\sum_{b\leq a}c_b[S((n-b,b))\otimes S((n))]+\sum_{b<n/2}d_b[S^{(n-b,b)}\otimes S((n))].\]
By Lemma \ref{bsspin}
\[[S((n-b,b))\otimes S((n))]=\sum_\mu c_{b,\mu}[D^\mu]\]
for some $c_{b,\mu}$ with $c_{b,\mu}>0$ if $h(\mu)\leq 4$ and $\mu_3+\mu_4=b$ while $c_{b,\mu}=0$ if $\mu_3+\mu_4>b$ or if $h(\mu)\geq 5$. By Lemma \ref{bsspecht}
\[[S^{(n-b,b)}\otimes S((n))]=\sum_{c\leq b}d_{b,c}[S((n-c,c))]\]
for some $d_{b,c}$. So
\[[D^\la\otimes D^{\be_n}]=\sum_\mu c_\mu[D^\mu]+\sum_{b<n/2}d_{(n-b,b)}[S((n-b,b))]\]
for some $c_\mu$ and $d_{(n-b,b)}$ with $c_\mu>0$ if $h(\mu)\leq 4$ and $\mu_3+\mu_4=a$, while $c_\mu=0$ if $h(\mu)>5$ or $\mu_3+\mu_4>a$. Let
\begin{align*}
\nu^1&=((n-a+1)/2,(n-a-1)/2,(a+3)/2,(a-3)/2),\\
\nu^2&=((n-a+3)/2,(n-a-3)/2,(a+1)/2,(a-1)/2).
\end{align*}
By Lemma \ref{L280520_2}, $D^{\nu^1}$ and $D^{\nu^2}$ are not composition factors of $S((n-b,b))$ for any $b<n/2$ and so they are composition factors of $D^\la\otimes D^{\be_n}$.

From Theorem \ref{T050620} we have that if $\al\in\Parod(n)$ and $\psi^{\la,+}_{\al,-}\not=\psi^{\la,-}_{\al,+}$ then $h(\al)=2$. So $\psi^{\la,+}\da_{\A_{n-1}}$ and $\psi^{\la,+}\da_{\A_{n-1}}$ can only differ on the conjugacy classes labeled by $(n-1)$ and then by Lemma \ref{cbs}
\[\psi^{\la,+}\da_{\A_{n-1}}-\psi^{\la,+}\da_{\A_{n-1}}=c(\psi^{\be_{n-1},+}-\psi^{\be_{n-1},-})\]
with $c\in\Q$. By Lemma \ref{L280520_3} any composition factor of $D^{\be_{n-1}}\otimes D^{\be_{n-1}}$ is of the form $D^{(n-x-1,x)}$. In particular any composition factor of $(D^{\be_{n-1}}\otimes D^{\be_{n-1}})\ua^{\s_n}$ is of the form $D^{(n-x,x)}$ or $D^{(n-x-1,x,1)}$ (since they are composition factors of a Specht module of the form $S^{(n-x,x)}$ or $S^{(n-x-1,x,1)}$). Since $D^{\be_n}\da_{\s_{n-1}}\cong D^{\be_{n-1}}$ by Lemma \ref{Lemma39} and so $E^{\be_n}_\pm\da_{\A_{n-1}}\cong E^{\be_{n-1}}_\pm$ (up to exchange of $E^{\be_{n-1}}_\pm$), we have that 
\begin{align*}
&[(E^\la_+\otimes E^{\be_n}_\pm)\da_{\A_{n-1}}\ua^{\A_n}]-[(E^\la_-\otimes E^{\be_n}_\pm)\da_{\A_{n-1}}\ua^{\A_n}]\\
&=c([(E^{\be_{n-1}}_+\otimes E^{\be_{n-1}}_\pm)\ua^{\A_n}]-[(E^{\be_{n-1}}_-\otimes E^{\be_{n-1}}_\pm)\ua^{\A_n}])\\
&=c'[E^{\be_n}_+]+c''[E^{\be_n}_-]+d'[E^{\dbl(n-1,1)}_+]+d''[E^{\dbl(n-1,1)}_-]\\
&+\sum_{x<n/2-1}c_x[E^{(n-x,x)}]+d_x[E^{(n-x-1,x,1)}]
\end{align*}
with $c',c'',d',d'',c_x,d_x\in\Z$.

Since $E^\la_+\otimes E^{\be_n}_\pm\cong (E^\la_-\otimes E^{\be_n}_\mp)^\si$ with $\si\in\s_n\setminus\A_n$, so that $E^{\nu^1}$ and $E^{\nu^2}$ are composition factors of $E^\la_+\oplus E^{\be_n}_\pm$ or $E^\la_-\otimes E^{\be_n}_\pm$, it follows that both $E^{\nu^1}$ and $E^{\nu^2}$ are composition factors of $(E^\la_\pm\otimes E^{\be_n}_\pm)\da_{\A_{n-1}}\ua^{\A_n}\cong E^\nu\da_{\A_{n-1}}\ua^{\A_n}$. It can be checked that one of $\nu^1$ and $\nu^2$ has core $(3,2,1)$, while the other has core $(4,3,2,1)$. So the contents of $\nu^1$ and $\nu^2$ are (in some order):
\begin{align*}
\cont((3,2,1)+(n-6)/2\cdot(1,1)&=(n/2+1,n/2-1),\\
\cont((4,3,2,1)+(n-10)/2\cdot(1,1)&=(n/2-1,n/2+1).
\end{align*}
Let $\cont(\nu)=(c_0,c_1)$. As $E^{\nu^1}$ and $E^{\nu^2}$ are composition factors of $E^\nu\da_{\A_{n-1}}\ua^{\A_n}$ we have that
\[\cont(\nu^1),\cont(\nu^2)\in\{(c_0+1,c_1-1),(c_0,c_1),(c_0-1,c_1+1)\}\]
and then $\cont(\nu)=(n/2,n/2)$, which is the content of modules in the principal block (which has core the empty partition). From Theorem \ref{T050620_2}, it then follows that $D^\nu$ is the composition factor of some $S((n-b,b))$. Since $\nu$ is not the double of a partition by Theorem \ref{T280520}, we have that $\nu=(n-d,d)$ by Lemma \ref{L280520_2}. So any composition factor of $D^\nu\da_{\s_{n-1}}\ua^{\s_n}$ is of the form $D^\pi$ with $\pi_3\leq 1$. In particular $D^{\nu^1}$ and $D^{\nu^2}$ are not composition factors of $D^\nu\da_{\s_{n-1}}\ua^{\s_n}$, leading to a contradiction.
\end{proof}

We will now consider tensor products of basic spin modules with basic or second basic spin modules, and prove that, except possibly for small $n$, such products are not irreducible. The following proof does not require the modules to be splitting modules.

\begin{theor}\label{T080620}
Suppose $p=2$ and $n\geq 6$, and let $V$ be basic spin or second basic spin and $W$ first basic spin. Then $V\otimes W$ is not irreducible.
\end{theor}

\begin{proof}
If $V$ is second basic spin then the proof goes as the proof of \cite[Theorem 6.9]{m3}. So we may assume that $V$ is basic spin. If $n\not\equiv 0\Md{4}$ the result holds by \cite[Theorem 1.2]{m3} (since then basic spin modules in characteristic 2 are the reduction modulo 2 of basic spin modules in characteristic 0 by \cite[Table III]{Wales}). Consider now $n\equiv 0\Md{4}$. By Lemma \ref{L280520_3} any composition factor of $V\otimes W$ is of the form $E^{(n-a,a)}_{(\pm)}$. From Lemma \ref{Lemma39} we have that $D^{(n-a,a)}\da_{\s_{n-1}}\cong D^{(n-a-1,a)}$. Further $(n-a,a)\in\Parinv_2(n)$ if and only if $(n-a-a,a)\in\Parinv_2(n-1)$ (as $n\equiv 0\Md{4}$). So  $E^{(n-a,a)}_{(\pm)}\da_{\A_{n-1}}\cong E^{(n-a-1,a)}_{(\pm)}$ restricts irreducibly to $\A_{n-1}$. Since $(V\otimes W)\da_{\A_{n-1}}$ is not irreducible by the previous part (it is the product of two basic spin modules), it follows that also $V\otimes W$ is not irreducible.
\end{proof}

\section{Split-non-split case}\label{s9}

We will now consider irreducible tensor products of the form $E^\la\otimes E^{\be_n}_\pm$ for $n\not\equiv 2\Md{4}$. We start with the case where $n$ is odd.

\begin{theor}\label{T080620_3}
Let $p=2$, $n$ be odd and $(n)\not=\la\in\Par_2(n)\setminus\Parinv_2(n)$. Then $E^\la\otimes E^{\be_n}_\pm$ is not irreducible.
\end{theor}

\begin{proof}
If $\la$ has at least three normal nodes this holds by \cite[Theorem 13.2]{m4}. If $\la$ has two normal nodes then
\[\dim\Hom_{\s_n}(\End_F(D^\la),\End_F(D^\mu))\geq 3\]
from Lemma \ref{L090720_2} and so $E^\la\otimes E^{\be_n}_\pm$ is not irreducible by \cite[Lemma 10.1]{m4}.

So we may assume that $\la$ is JS. From Theorem \ref{T080620_2} we may assume that $h(\la)\leq 4$ and that $E^\la\otimes E^{\be_n}_\pm\cong E^\nu$ with $\nu=(n-a,a)$ or $\dbl(n-a,a)$.

Since $n$ is odd and $\la$ is JS, we then have that $h(\la)=3$ and the parts of $\la$ are odd (as by assumption $\la\not=(n)$). If $\la_1\equiv\la_2\equiv\la_3\Md{4}$, then
\[\overline{\dbl}(\la)_1\not\equiv\overline{\dbl}(\la)_2\not\equiv\overline{\dbl}(\la)_3\not\equiv\overline{\dbl}(\la)_4\not\equiv\overline{\dbl}(\la)_5\not\equiv\overline{\dbl}(\la)_6\Md{2}\]
and $\overline{\dbl}(\la)_7=0$. So the 2-core of $\overline{\dbl}(\la)$ has five or six rows. By Lemma \ref{L080620},  $S(\la)$ and $D^{\overline{\dbl}(\la)}$ are in the same block. It follows that $S(\la)$ and $D^\nu$ are in different blocks, leading to a contradiction thanks to Lemma \ref{L080620_2}.

If for some $\{j,k,\ell\}=\{1,2,3\}$ we have $\la_j\equiv\la_k\not\equiv\la_\ell\Md{4}$ then $\la_j\equiv n\Md{4}$, so $\overline{\dbl}(\la)$ and $\overline{\dbl}(n)$ have the same 2-core. If $\nu=\dbl(n-a,a)$, then either $n-a$ or $a$ is congruent to $2\Md{4}$ (since $D^\nu$ does not split). So the other part is not congruent to $n\Md{4}$, and then $\nu$ and $\overline{\dbl}(n)$ have different 2-cores and we can then conclude as in the previous case.

Thus we may now assume that $\nu=(n-a,a)$ for some $a$. Since $E^\la\otimes E^{\be_n}_\pm\cong E^\nu$ we have that $D^\nu\subseteq D^\la\otimes D^{\be_n}$ and so $D^\la\subseteq D^\nu\otimes D^{\be_n}$. From Lemmas \ref{L280520_2} and \ref{bsspecht}, we have that $h(\la)\leq 2$ or $\la=\dbl(n-b,b)$. In either case this leads to a contradiction to $\la$ having only odd parts and $(n)\not=\la\not\in\Parinv_2(n)$.
\end{proof}

The case where $n\equiv 0\Md{4}$ will turn out to be more complicated. The first of the two cases appearing in Theorem \ref{T050620_2} will be handled using results for $n$ odd, so we will only consider the second of the two cases. Note that we may assume that $\la$ has at most three normal nodes from \cite[Theorem 13.2]{m4}, and that $h(\la)\leq 6$ by Theorem \ref{T050620_2}. JS-partitions will be ruled out considering block decomposition, so we will only consider partitions $\la$ with two or three normal nodes.

\begin{lemma}\label{L090720}
Suppose $p=2$, $n\equiv 0\Md{4}$ and $\la,\nu\in\Par_2(n)\setminus\Parinv_2(n)$ with $h(\la)\leq 6$, $2\leq \eps_0(\la)+\eps_1(\la)\leq 3$, and that $D^\nu$ is a composition factor of some $S(p(n-a-b,a,b))$ with $a\equiv b\equiv\pm 1\Md{4}$. If $E^\la\otimes E^{\be_n}_\pm\cong E^\nu$ then one of the following holds for some residue $i$:
\begin{enumerate}
\item $\eps_i(\la)>0$ and $E^{\tilde e_i\la}_{(\pm)}\otimes E^{\be_{n-1}}_\pm$ is irreducible,

\item $\eps_i(\la)=2$, $\eps_{1-i}(\la)=0$, $\tilde e_i\la\not\in\Parinv_2(n)$, $E^{\tilde e_i\la}\otimes E^{\be_{n-1}}_\pm$ has simple head and socle isomorphic to $E^{\tilde e_j\nu}$ and $[E^{\tilde e_i\la}\otimes E^{\be_{n-1}}_\pm:E^{\tilde e_j\nu}]=2$ with $\nu$ and $j$ as in one of the following:
\begin{itemize}
\item $h(\nu)=4$ and all removable nodes of $\nu$ have residue $j$,

\item $\nu=\dbl(c,d,n-c-d)$ with $c\equiv d\equiv 1\Md 4$ and $j=0$,

\item $\nu=\dbl(c,n-c-d,d)$ with $c\equiv d\equiv \pm 1\Md 4$, $d\geq 3$ and $j=\res(1,(c+1)/2)$,

\item $\nu=\dbl(n-c-d,c,d)$ with $c\equiv d\equiv \pm 1\Md 4$, $d\geq 3$ and $j=\res(3,(c+1)/2)$,
\end{itemize}

\item $\eps_i(\la)=2$, $\eps_{1-i}(\la)=0$, $\nu=\dbl(c,d,n-c-d)$ with $c\equiv d\equiv 3\Md 4$, the head and socle of $(D^{\tilde e_i\la}\da_{\A_{n-1}})\otimes E^{\be_{n-1}}_\pm$ are contained in $E^{\tilde e_1\nu}_+\oplus E^{\tilde e_1\nu}_-$ and $[(D^{\tilde e_i\la}\da_{\A_{n-1}})\otimes E^{\be_{n-1}}_\pm:E^{\tilde e_1\nu}_\pm]\leq 2$,

\item $\eps_0(\la),\eps_1(\la)=1$, $\tilde e_0\la,\tilde e_1\la\not\in\Parinv_2(n)$, $\nu=\dbl(n-c-d,c,d)$ with $c\equiv d\equiv 1\Md 4$, $E^{\tilde e_i\la}\otimes E^{\be_{n-1}}_\pm\cong E^{\tilde e_1\nu}_+\oplus E^{\tilde e_1\nu}_-$ and $E^{\tilde e_i\la}\otimes E^{\be_{n-1}}_\pm\cong e_0 E^\nu$,

\item $\eps_i(\la)=3$, $\eps_{1-i}(\la)=0$, $\tilde e_i\la\not\in\Parinv_2(n)$, $\nu=\dbl(c,d,n-c-d)$ with $c\equiv d\equiv 3\Md 4$ and one of the following holds:
\begin{itemize}
\item $E^{\tilde e_i\la}\otimes E^{\be_{n-1}}_\pm$ has socle $E^{\tilde e_1\nu}_\pm$, head $E^{\tilde e_1\nu}_\mp$ and
\[[E^{\tilde e_i\la}\otimes E^{\be_{n-1}}_\pm:E^{\tilde e_1\nu}_+]=[E^{\tilde e_i\la}\otimes E^{\be_{n-1}}_\pm:E^{\tilde e_1\nu}_-]=1,\]

\item $E^{\tilde e_i\la}\otimes E^{\be_{n-1}}_\pm\cong E^{\tilde e_1\nu}_+\oplus E^{\tilde e_1\nu}_-$,
\end{itemize}

\item $\eps_i(\la)=1$, $\eps_{1-i}(\la)=2$, $\tilde e_0\la,\tilde e_1\la\not\in\Parinv_2(n)$, $\nu=\dbl(n-c-d,c,d)$ with $c\equiv d\equiv 1\Md 4$ and $d\geq 5$, $E^{\tilde e_i\la}\otimes E^{\be_{n-1}}_\pm\cong E^{\tilde e_1\nu}_+\oplus E^{\tilde e_1\nu}_-$ and $E^{\tilde e_{1-i}\la}\otimes E^{\be_{n-1}}_\pm$ has simple head and socle isomorphic to $E^{\tilde e_0\nu}$ and $[E^{\tilde e_{1-i}\la}\otimes E^{\be_{n-1}}_\pm:E^{\tilde e_0\nu}]=2$.
\end{enumerate}
\end{lemma}

\begin{proof}
In view of case (i) we will assume that $E^{\tilde e_i\la}_{(\pm)}\otimes E^{\be_{n-1}}_\pm$ is reducible for all $i$ such that $\eps_i(\la)>0$. We will use Lemma \ref{Lemma39} without further comment. Since $n\equiv 0\Md 4$, Lemma~\ref{branchingbs} gives $E^{\be_n}_\pm\da_{\A_{n-1}}\cong E^{\be_{n-1}}_\pm$.

Note that since $D^\nu$ is in the block of $S(p(n-a-b,a,b))$ the 2-core of $\nu$ is $(4^x,3,2,1)$ with $0\leq x\leq 1$. In particular, by Lemma \ref{L280520_2}, either $3\leq h(\nu)\leq 4$ and $\nu_1\not\equiv\nu_2\not\equiv\nu_3\not\equiv\nu_4\Md{2}$ or $\nu=\dbl(p(n-c-d,c,d))$ with $c\equiv d\equiv \pm 1\Md{4}$. From Lemmas \ref{L060820} and \ref{L060820_2} it can be checked that $\soc(e_jE^\nu)\cong E^{\tilde e_j\nu}$ or $E^{\tilde e_j\nu}_+\oplus E^{\tilde e_j\nu}_-$ for all $j$ such that $\eps_j(\nu)>0$.

{\bf Case 1:} assume that $\nu\not=\dbl(n-c-d,c,d)$ with $c\equiv d\equiv 1\Md{4}$ and $\nu\not=\dbl(c,d,n-c-d)$ with $c\equiv d\equiv 3\Md{4}$.

If $h(\nu)\leq 4$ then all the removable nodes of $\nu$ have the same residue $j$, since $\nu_1\not\equiv\nu_2\not\equiv\nu_3\not\equiv\nu_4\Md{2}$. Furthermore, $\nu_1-\nu_2\geq 3$ or $\nu_3-\nu_4\geq 3$ (since $\nu\not\in\Parinv_2(n)$). If $h(\nu)=4$ then $\tilde e_j\nu=(\nu_1,\nu_2,\nu_3,\nu_4-1)\not\in\Parinv_2(n-1)$. If $h(\nu)=3$ then $\tilde e_j(\nu)=(\nu_1,\nu_2,\nu_3-1)\not\in\Parinv_2(n-1)$ (since $\nu_1-\nu_2\geq 3$ or $\nu_3-1\geq 2$).

If $\nu=\dbl(n-c-d,c,d)$ with $c\equiv d\equiv 3\Md{4}$, $\nu=\dbl(c,n-c-d,d)$ with $c\equiv d\equiv \pm 1\Md{4}$ or $\nu=\dbl(c,d,n-c-d)$ with $c\equiv d\equiv 1\Md{4}$, then $\eps_j(\nu)=3+\de_{d\geq 3}$, $\eps_{1-j}(\nu)=0$ and $\tilde e_j(\nu)\not\in\Parinv_2(n-1)$, where $j$ is the residue of $(1,(c+1)/2)$ and $(3,(c+1)/2)$.

Since $E^\la\otimes E^{\be_n}_\pm\cong E^\nu$, it follows that $E^\la\da_{\A_{n-1}}\otimes E^{\be_{n-1}}_\pm\cong e_jE^\nu$. In particular $\soc(E^\la\da_{\A_{n-1}})\otimes E^{\be_{n-1}}_\pm$ is both a submodule and a quotient of $e_jE^\la$. Since $\soc(e_jE^\nu)\cong E^{\tilde e_j\nu}$ is simple, it follows that $\soc(E^\la\da_{\A_{n-1}})$ is simple, so $\eps_{1-i}(\la)=0$ for some residue $i$ and $\tilde e_i\la\not\in\Parinv_2(n)$. Furthermore, the head and socle of $E^{\tilde e_i\la}\otimes E^{\be_{n-1}}_\pm$ are isomorphic to $E^{\tilde e_j\nu}$. Since by assumption $E^{\tilde e_i\la}\otimes E^{\be_{n-1}}_\pm$ is not simple, we have
\[2\leq[E^{\tilde e_i\la}\otimes E^{\be_{n-1}}_\pm:E^{\tilde e_j\nu}]\leq[(e_iE^\la)\otimes E^{\be_{n-1}}_\pm:E^{\tilde e_j\nu}]/\eps_i(\la)=\eps_j(\nu)/\eps_i(\la).\]
Because $\eps_i(\la)\geq 2$ and $\eps_j(\nu)\leq 4$ it follows that $\eps_i(\la)=2$ and $\eps_j(\nu)=4$ (and so $h(\nu)=4$ if $h(\nu)\leq 4$ while $d\geq 3$ if $\nu=\dbl(n-c-d,c,d)$ or $\dbl(c,n-c-d,d)$).

{\bf Case 2:} assume that $\nu=\dbl(n-c-d,c,d)$ with $c\equiv d\equiv 1\Md{4}$. Then $\eps_1(\nu)=1$, $\eps_0(\nu)=3+\de_{d\geq 5}$, $\tilde e_1\nu\in\Parinv_2(n-1)$ and $\tilde e_0\nu\not\in\Parinv_2(n-1)$.

{\bf Case 2.1:} assume that $\eps_0(\la),\eps_1(\la)>0$. Then
\[((e_0 E^\la)\otimes E^{\be_{n-1}}_\pm)\oplus ((e_1 E^\la)\otimes E^{\be_{n-1}}_\pm)\cong e_0E^\nu\oplus E^{\tilde e_1\nu}_+\oplus E^{\tilde e_1\nu}_-.\]
Note that $\soc(e_0E^\nu)$ is simple and that $e_1E^\nu\cong E^{\tilde e_1\nu}_+\oplus E^{\tilde e_1\nu}_-$. Since we are excluding the case where some $E^{\tilde e_i\la}_\pm\otimes E^{\be_{n-1}}_\pm$ is simple, we may then assume that, for some residue $i$,
\[(e_iE^\la)\otimes E^{\be_{n-1}}_\pm\cong E^{\tilde e_1\nu}_+\oplus E^{\tilde e_1\nu}_-\hspace{11pt}\text{and}\hspace{11pt}(e_{1-i}E^\la)\otimes E^{\be_{n-1}}_\pm\cong e_0E^\nu.\]
Then $\tilde e_i\la\not\in\Parinv_2(n-1)$ since otherwise $E^{\tilde e_i\la}_\pm\otimes E^{\be_{n-1}}_\pm$ would be simple. Furthermore, $\tilde e_{1-i}\la\not\in\Parinv_2(n-1)$ since $\soc(e_0E^\la)$ is simple. 

We can then conclude similarly to case 1, using the fact that
\[[E^{\tilde e_x\la}\otimes E^{\be_{n-1}}_\pm:E^{\tilde e_y\nu}_{(\pm)}]\leq\eps_y(\nu)/\eps_x(\la)\]
for $(x,y)\in\{(i,1),(1-i,0)\}$.

{\bf Case 2.2:} assume that $\eps_{1-i}(\la)=0$ for some $i$. In this case
\[\eps_i(\la)\geq 2>1=\eps_1(\nu),\]
so $(D^{\tilde e_i\la}\da_{\A_{n-1}})\otimes E^{\be_{n-1}}_\pm\subseteq e_0E^\nu$ (and then $(D^{\tilde e_i\la}\da_{\A_{n-1}})\otimes E^{\be_{n-1}}_\pm$ is also a quotient of $e_0E^\nu$). We can then conclude similarly to the previous cases.

{\bf Case 3:} assume that $\nu=\dbl(c,d,n-c-d)$ with $c\equiv d\equiv 3\Md{4}$. Then $\eps_0(\nu)=0$, $\eps_1(\nu)=5$ and $\tilde e_1\nu\in\Parinv_2(n-1)$.

{\bf Case 3.1:} assume that $\eps_0(\la),\eps_1(\la)>0$. Then
\[((e_0 E^\la)\otimes E^{\be_{n-1}}_\pm)\oplus ((e_1 E^\la)\otimes E^{\be_{n-1}}_\pm)\cong e_1E^\nu.\]

Since $\soc(e_1E^\nu)\cong E^{\tilde e_1\nu}_+\oplus E^{\tilde e_1\nu}_-$, we have $e_1E^\nu\cong A\oplus B$ with $\soc(A)\cong E^{\tilde e_1\nu}_+$ and $\soc(B)\cong E^{\tilde e_1\nu}_-$. Note that if $\si\in\s_{n-1}\setminus\A_{n-1}$, then $A^\si\subseteq (e_1E^\nu)^\si\cong e_1E^\nu$. Since $\soc(A^\si)=E^{\tilde e_1\nu}_-$ it follows from Lemma \ref{simplesoc} that $A^\si$ is isomorphic to a submodule of $B$. Similarly $B^\si$ is isomorphic to a submodule of $A$. By dimensions we then have that $A^\si\cong B$. 

Since
\[((e_0 E^\la)\otimes E^{\be_{n-1}}_\pm)\oplus ((e_1 E^\la)\otimes E^{\be_{n-1}}_\pm)\cong A\oplus B,\]
$(e_i E^\la)\otimes E^{\be_{n-1}}_\pm\not=0$ for any residue $i$ and $A$ and $B$ have simple socle, it follows that $(e_i E^\la)\otimes E^{\be_{n-1}}_\pm$ also has simple socle. Thus, up to exchange of $E^{\tilde e_1\nu}_\pm$ (and of $A$ and $B$), we may then assume that
\[\soc((e_0 E^\la)\otimes E^{\be_{n-1}}_\pm)\cong E^{\tilde e_1\nu}_+\hspace{11pt}\text{and}\hspace{11pt}\soc((e_1 E^\la)\otimes E^{\be_{n-1}}_\pm)\cong E^{\tilde e_1\nu}_-.\]
From Lemma \ref{simplesoc} it follows that (up to isomorphism), 
\[(e_0 E^\la)\otimes E^{\be_{n-1}}_\pm\subseteq A\hspace{11pt}\text{and}\hspace{11pt}(e_1 E^\la)\otimes E^{\be_{n-1}}_\pm\subseteq B\]
and then comparing dimensions that
\[(e_0 E^\la)\otimes E^{\be_{n-1}}_\pm\cong A\hspace{11pt}\text{and}\hspace{11pt}(e_1 E^\la)\otimes E^{\be_{n-1}}_\pm\cong B.\]
In particular $((e_0 E^\la)\otimes E^{\be_{n-1}}_\pm)^\si\cong (e_1 E^\la)\otimes E^{\be_{n-1}}_\pm$ for $\si\in\s_{n-1}\setminus\A_{n-1}$ and then $(e_0 D^\la)\otimes D^{\be_{n-1}}$ and $(e_1 D^\la)\otimes D^{\be_{n-1}}$ have the same Brauer character. In view of Lemma \ref{cbs} the same holds also for $e_0 D^\la$ and $e_1D^\la$, leading to a contradiction since these modules are non-zero and in different blocks.

{\bf Case 3.2:} assume that $\eps_{1-i}(\la)=0$ for some $i$. In this case the head and socle of $(D^{\tilde e_i\la}\da_{\A_{n-1}})\otimes E^{\be_{n-1}}_\pm$ are contained in the head and socle of $e_1E^\nu$ which are isomorphic to $E^{\tilde e_1\nu}_+\oplus E^{\tilde e_1\nu}_-$. The lemma then follows from the fact that
\[[(D^{\tilde e_i\la}\da_{\A_{n-1}})\otimes E^{\be_{n-1}}_\pm:E^{\tilde e_1\nu}_\pm]\leq \eps_1(\nu)/\eps_i(\la),\]
using arguments similar to those in cases 1 and 2.
\end{proof}

Case (i) of Lemma \ref{L090720} can be handled using induction. In order to show that $E^\la\otimes E^{\be_n}_\pm$ is never irreducible, we will thus now show that cases (ii) to (vi) of Lemma \ref{L090720} never arise.

\begin{theor}\label{T100720}
Suppose $p=2$, $n\equiv 0\Md{4}$, $\la,\nu\in\Par_2(n)\setminus\Parinv_2(n)$ with $h(\la)\leq 6$, $\eps_0(\la)+\eps_1(\la)=3$ and $D^\nu$ is a composition factor of some $S(p(n-a-b,a,b))$ with $a\equiv b\equiv\pm 1\Md{4}$. If $E^\la\otimes E^{\be_n}_\pm\cong E^\nu$ then, for some residue $i$, $\eps_i(\la)>0$ and $E^{\tilde e_i\la}_{(\pm)}\otimes E^{\be_{n-1}}_\pm$ is irreducible.
\end{theor}

\begin{proof}
We will use Lemma \ref{Lemma39} without further comment. Since $\la$ has three normal nodes, $h(\la)\geq 3$, so that $n\geq 8$ (as $n\equiv 0\Md{4}$). By Lemma \ref{branchingbs}, $E^{\be_n}_\pm\da_{\A_{n-1}}\cong E^{\be_{n-1}}_\pm$, $D^{\be_{n-1}}\da_{\s_{n-2}}\cong D^{\be_{n-2}}|D^{\be_{n-2}}$ and $D^{\be_{n-2}}\da_{\s_{n-3}}\cong D^{\be_{n-3}}$.

We may assume that $E^{\tilde e_i\la}_{(\pm)}\otimes E^{\be_{n-1}}_\pm$ is not irreducible whenever $\eps_i(\la)>0$. Then, by Lemma \ref{L090720}, whenever $\eps_i(\la)>0$, $\tilde e_i\la\not\in\Parinv_2(n-1)$ and for every $\de\in\{\pm\}$
\begin{align*}
&\dim\Hom_{\A_{n-1}}(\End_F(E^{\tilde e_i\la}),\Hom_F(E^{\be_{n-1}}_\pm,E^{\be_{n-1}}_\de))\\
&=\dim\Hom_{\A_{n-1}}((E^{\tilde e_i\la})^*\otimes E^{\tilde e_i\la},(E^{\be_{n-1}}_\pm)^*\otimes E^{\be_{n-1}}_\de))\\
&=\dim\Hom_{\A_{n-1}}(E^{\tilde e_i\la}\otimes E^{\be_{n-1}}_\pm,E^{\tilde e_i\la}\otimes E^{\be_{n-1}}_\de)\\
&\leq 2
\end{align*}
(note that for any $\mu\in\Par_2(m)\setminus\Parinv_2(m)$, $\pi\in\Parinv_2(m)$ and $\si\in\s_m\setminus\A_m$, $E^\mu\otimes E^\pi_\mp\cong (E^\mu\otimes E^\pi_\pm)^\si$). It also follows that
\begin{align*}
&\dim\Hom_{\s_{n-1}}(\End_F(D^{\tilde e_i\la}),\End_F(D^{\be_{n-1}}))\\
&=\dim\Hom_{\s_{n-1}}(D^{\tilde e_i\la}\otimes D^{\tilde e_i\la},D^{\be_{n-1}}\otimes D^{\be_{n-1}})\\
&=\dim\Hom_{\s_{n-1}}(D^{\tilde e_i\la}\otimes D^{\tilde e_i\la},((E^{\be_{n-1}}_\pm)^*\ua^{\s_{n-1}})\otimes D^{\be_{n-1}})\\
&=\dim\Hom_{\s_{n-1}}(D^{\tilde e_i\la}\otimes D^{\tilde e_i\la},((E^{\be_{n-1}}_\pm)^*\otimes (D^{\be_{n-1}})\da_{\A_{n-1}})\ua^{\s_{n-1}}))\\
&=\dim\Hom_{\s_{n-1}}(D^{\tilde e_i\la}\otimes D^{\tilde e_i\la},((E^{\be_{n-1}}_\pm)^*\otimes (E^{\be_{n-1}}_+\oplus E^{\be_{n-1}}_-)\ua^{\s_{n-1}}))\\
&=\dim\Hom_{\A_{n-1}}((D^{\tilde e_i\la}\otimes D^{\tilde e_i\la})\da_{\A_{n-1}},((E^{\be_{n-1}}_\pm)^*\otimes (E^{\be_{n-1}}_+\oplus E^{\be_{n-1}}_-)\ua^{\s_{n-1}}))\\
&=\dim\Hom_{\A_{n-1}}(E^{\tilde e_i\la}\otimes E^{\tilde e_i\la},((E^{\be_{n-1}}_\pm)^*\otimes (E^{\be_{n-1}}_+\oplus E^{\be_{n-1}}_-)\ua^{\s_{n-1}}))\\
&=\dim\Hom_{\A_{n-1}}(\End_F(E^{\tilde e_i\la}),\Hom_F(E^{\be_{n-1}}_\pm,E^{\be_{n-1}}_+\oplus E^{\be_{n-1}}_-)))\\
&\leq 4.
\end{align*}

Assume first that $\eps_0(\tilde e_i\la)+\eps_1(\tilde e_i\la)\geq 4$ for some $i$ with $\eps_i(\la)>0$. Then $\dim\End_{\s_{n-2}}(D^{\tilde e_i\la})\geq 4$ and so $D_1^{\oplus 3}\subseteq \End_F(D^{\tilde e_i\la})$ by Lemma \ref{L080620_3}. Furthermore, $D_1\subseteq\End_F(D^{\be_{n-1}})$ by Lemma  \ref{L080620_4}. So $E_1^{\oplus 3}\subseteq \End_F(E^{\tilde e_i\la})$ and $E_1\subseteq\Hom_F(E^{\be_{n-1}}_\pm,E^{\be_{n-1}}_\de)$ for some $\de$, in particular
\[\dim\Hom_{\A_{n-1}}(\End_F(E^{\tilde e_i\la}),\Hom_F(E^{\be_{n-1}}_\pm,E^{\be_{n-1}}_\de))\geq 3\]
leading to a contradiction.

If $\eps_0(\la)>0$ and $\eps_1(\la)>0$, then $\eps_0(\tilde e_i\la)+\eps_1(\tilde e_i\la)\geq 4$ for some $i$ by \cite[Lemma 4.4]{m1}. So we may now assume that $\eps_i(\la)=3$, $\eps_{1-i}(\la)=0$ and $\eps_{1-i}(\tilde e_i\la)\leq 1$ for some $i$. Hence we are in case (v) of Lemma \ref{L090720}. In particular, $\nu=\dbl(c,d,n-c-d)$ with $c,d\equiv 3\Md{4}$ and there are two cases to be considered.

{\bf Case 1:} assume that $E^{\tilde e_i\la}\otimes E^{\be_{n-1}}_\pm$ has socle $E^{\tilde e_1\nu}_\pm$, head $E^{\tilde e_1\nu}_\mp$ and
\[[E^{\tilde e_i\la}\otimes E^{\be_{n-1}}_\pm:E^{\tilde e_1\nu}_+]=[E^{\tilde e_i\la}\otimes E^{\be_{n-1}}_\pm:E^{\tilde e_1\nu}_-]=1.\]
Then 
\begin{align*}
&\dim\Hom_{\s_{n-1}}(\End_F(D^{\tilde e_i\la}),\End_F(D^{\be_{n-1}}))\\
&=\dim\Hom_{\A_{n-1}}(\End_F(E^{\tilde e_i\la}),\Hom_F(E^{\be_{n-1}}_\pm,E^{\be_{n-1}}_+\oplus E^{\be_{n-1}}_-)))\\
&=\dim\Hom_{\A_{n-1}}(E^{\tilde e_i\la}\otimes E^{\be_{n-1}}_\pm,E^{\tilde e_i\la}\otimes (E^{\be_{n-1}}_+\oplus E^{\be_{n-1}}_-))\\
&=2
\end{align*}
(with the first two equalities holding as above). If $\eps_{1-i}(\tilde e_i\la)=0$ this leads to a contradiction, thanks to Lemma \ref{L090720_2}. If instead $\eps_{1-i}(\tilde e_i\la)=1$, then $D_0\oplus D_1^{\oplus 2}\subseteq \End_F(D^{\tilde e_i\la})$ by Lemma \ref{L080620_3}, leading to a contradiction thanks to Lemma \ref{L080620_4}.

{\bf Case 2:} assume that $E^{\tilde e_i\la}\otimes E^{\be_{n-1}}_\pm\cong E^{\tilde e_1\nu}_+\oplus E^{\tilde e_1\nu}_-$. Since $\nu=\dbl(c,d,n-c-d)$ with $c,d\equiv 3\Md{4}$, we have that $\tilde e_1\nu=\dbl(c,d,n-c-d-1)$, $\eps_1(\tilde e_1\nu)=4$, $\eps_0(\tilde e_0\nu)=0$ and $\tilde e_1^2\nu\not\in\Parinv_2(n-2)$.

Furthermore, from $E^{\tilde e_i\la}\otimes E^{\be_{n-1}}_\pm\cong E^{\tilde e_1\nu}_+\oplus E^{\tilde e_1\nu}_-$ we have that
\[D^{\tilde e_i\la}\otimes D^{\be_{n-1}}\cong (E^{\tilde e_i\la}\otimes E^{\be_{n-1}}_\pm)\ua^{\s_{n-1}}\cong (D^{\tilde e_1\nu})^{\oplus 2}\]
and so
\[(D^{\tilde e_i\la}\otimes D^{\be_{n-1}})\da_{\s_{n-2}}\cong (e_1D^{\tilde e_1\nu})^{\oplus 2}.\]

{\bf Case 2.1:} assume that $\eps_{1-i}(\tilde e_i\la)=1$. In this case
\[(D^{\tilde e_i\la}\otimes D^{\be_{n-1}})\da_{\s_{n-2}}\cong (\overbrace{(e_0D^{\tilde e_i\la})\otimes (D^{\be_{n-1}}\da_{\s_{n-2}})}^A)\oplus(\overbrace{(e_1D^{\tilde e_i\la})\otimes (D^{\be_{n-1}}\da_{\s_{n-2}})}^B)\]
with $A,B\not=0$. Since $A\oplus B\cong(e_1D^{\tilde e_1\nu})^{\oplus 2}$ and so $\soc(A\oplus B)\cong (D^{\tilde e_1^2\nu})^{\oplus 2}$, we have that $A$ and $B$ both have simple socle. Thus, $A,B$ are both isomorphic to submodules of $e_1D^{\tilde e_1\nu}$ by Lemma \ref{simplesoc}, so $A,B\cong e_1D^{\tilde e_1\nu}$ by dimension. In particular $A\cong B$. From Lemma \ref{cbs} it follows that $e_0D^{\tilde e_i\la}$ and $e_1D^{\tilde e_i\la}$ have the same Brauer character. 
This leads to a contradiction, as $e_0D^{\tilde e_i\la}$ and $e_1D^{\tilde e_i\la}$ are non-zero and lie in distinct blocks.

{\bf Case 2.2:} assume that $\eps_{1-i}(\tilde e_i\la)=0$. Since $(D^{\tilde e_i\la}\otimes D^{\be_{n-1}})\da_{\s_{n-2}}\cong (e_1D^{\tilde e_1\nu})^{\oplus 2}$, $\eps_i(\tilde e_i\la)=\eps_i(\la)-1=2$, $\eps_1(\nu)=4$ and $D^{\be_{n-1}}\da_{\s_{n-2}}\cong D^{\be_{n-2}}|D^{\be_{n-2}}$ we have that
\[[D^{\tilde e_i^2\la}\otimes D^{\be_{n-2}}:D^{\tilde e_1^2\nu}]\leq 2\eps_1(\tilde e_1\nu)/(2\eps_i(\tilde e_i\la))=2.\]
Furthermore, since $\tilde e_1^2\nu=\dbl(c,d-1,n-c-d-1)$ and $n-c-d-1\geq 1$, we have by \cite[Theorem 1.1]{m1} that $D^{\tilde e_i^2\la}\otimes D^{\be_{n-2}}\not\cong D^{\tilde e_1^2\nu}$. So from
\[D^{\tilde e_i^2\la}\otimes D^{\be_{n-2}}\subseteq (e_1D^{\tilde e_1\nu})^{\oplus 2}\]
we either have that either $[D^{\tilde e_i^2\la}\otimes D^{\be_{n-2}}:D^{\tilde e_1^2\nu}]=2$ and the head and socle of $D^{\tilde e_i^2\la}\otimes D^{\be_{n-2}}$ are isomorphic to $D^{\tilde e_1^2\nu}$, or that $D^{\tilde e_i^2\la}\otimes D^{\be_{n-2}}\cong(D^{\tilde e_1^2\nu})^{\oplus 2}$.

Assume that we are in the first case. Then $D^{\tilde e_i^2\la}\otimes D^{\be_{n-2}}\subseteq e_1D^{\tilde e_1\nu}$ by Lemma \ref{simplesoc}. By Lemma \ref{l8} we then have that $D^{\tilde e_i^2\la}\otimes D^{\be_{n-2}}\cong V_2$ where $V_2$ is the unique submodule of $e_1D^{\tilde e_1\nu}$ with $[V_2:D^{\tilde e_1^2\nu}]=2$ and $\hd(V_2)\cong D^{\tilde e_1^2\nu}$. Since $\tilde e_1^2\nu\not\in\Parinv_2(n-2)$ we have that $E^{\tilde e_1^2\nu}\ua^{\s_{n-2}}\cong D^{\tilde e_1^2\nu}|D^{\tilde e_1^2\nu}$. Since $\tilde e_1\nu\in\Parinv_2(n-1)$ we have from Lemma \ref{L060820_2} that $E^{\tilde e_1^2\nu}\ua^{\s_{n-2}}\subseteq e_1D^{\tilde e_1\nu}$ and so $V_2\cong E^{\tilde e_1^2\nu}\ua^{\s_{n-2}}$.

It follows that in either case
\begin{align*}
[D^\la\otimes D^{\be_n}]&=2[D^\nu],\\
[D^{\tilde e_i\la}\otimes D^{\be_{n-1}}]&=2[D^{\tilde e_1\nu}],\\
[D^{\tilde e_i^2\la}\otimes D^{\be_{n-2}}]&=2[D^{\tilde e_1^2\nu}]
\end{align*}
(the first equality holding by $E^\la\otimes E^{\be_n}_\pm\cong E^\nu$).

By Lemma \ref{L080620_2}, $S^{\overline{\dbl}(\la)}$ is in the same block as $D^\nu$, $S^{\overline{\dbl}(\tilde e_i\la)}$ in the same block as $D^{\tilde e_1\nu}$ and $S^{\overline{\dbl}(\tilde e_i^2\la)}$ in the same block as $D^{\tilde e_1^2\nu}$. Since $\nu=\dbl(c,d,n-c-d)$, $\tilde e_1^2\nu=\dbl(c,d,n-c-d-1)$ and $\tilde e_1\nu=\dbl(c,d-1,n-c-d-1)$ with $c\equiv d\equiv 3\Md{4}$ and $n-c-d\equiv 2\Md{4}$,
\[|\{k:\la_k\equiv 3\Md{4}\}|=|\{k:\la_k\equiv 1\Md{4}\}|+2\]
and if $(k,\la_k)$ is either of the bottom two $i$-normal nodes of $\la$, then $\la_k$ is congruent to 2 or 3 modulo 4.

Further note that $\eps_i(\tilde e_i^2\la)=1$, $\eps_1(\tilde e_1^2\nu)=3$ and $\eps_0(\tilde e_1^2\nu)=0$.

{\bf Case 2.2.1:} assume that $\eps_{1-i}(\tilde e_i^2\la)=0$. Then $\tilde e_i^2\la$ is JS. Since the two top $i$-conormal nodes of $\tilde e_i^2\la$ are the bottom two $i$-normal nodes of $\la$, it follows that $(h(\la),1)$ is the $i$-good node of $\la$, leading to a contradiction. 

{\bf Case 2.2.2:} assume that $\eps_{1-i}(\tilde e_i^2\la)>0$. Since $D^{\tilde e_i^2\la}\otimes D^{\be_{n-2}}\sim D^{\tilde e_1^2\nu}|D^{\tilde e_1^2\nu}$,
\[[D^{\tilde e_i^3\la}\otimes D^{\be_{n-3}}:D^{\tilde e_1^3\nu}]+\eps_{1-i}(\tilde e_i^2\la)[D^{\tilde e_{1-i}\tilde e_i^2\la}\otimes D^{\be_{n-3}}:D^{\tilde e_1^3\nu}]\leq 2\eps_1(\tilde e_1^2\nu)=6.\]
From the fact that
\[\soc(D^{\tilde e_i^2\la}\da_{\s_{n-3}}\otimes D^{\be_{n-3}})\subseteq \soc(D^{\tilde e_1^\nu}\da_{\s_{n-3}})^{\oplus 2}\cong (D^{\tilde e_1^3\nu})^{\oplus 2}\]
(and similarly for the head) we have that, for every $j$, $D^{\tilde e_j\tilde e_i^2\la}\otimes D^{\be_{n-3}}$ has simple head and socle isomorphic to $D^{\tilde e_1^3\nu}$. In particular
\begin{align*}
\dim\Hom_{\s_{n-3}}(\End_F(D^{\tilde e_j\tilde e_i^2\la}),\End_F(D^{\be_{n-3}}))&=\dim\End_{\s_{n-3}}(D^{\tilde e_j\tilde e_i^2\la}\otimes D^{\be_{n-3}})\\
&\leq [D^{\tilde e_j\tilde e_i^2\la}\otimes D^{\be_{n-3}}:D^{\tilde e_1^3\nu}].
\end{align*}

{\bf Case 2.2.2.1:} assume that $D^{\tilde e_j\tilde e_i^2\la}\otimes D^{\be_{n-3}}$ is simple for some $j$. Then by \cite[Main Theorem]{bk}, $\tilde e_j\tilde e_i^2\la=(n-3)$. Since $\eps_i(\la)=3$ it follows that $\la=(n-3,2,1)$ and $(3,1)$ is the $i$-good node of $\la$, again giving a contradiction.

{\bf Case 2.2.2.2:} assume that $\eps_{1-i}(\tilde e_i^2\la)=2$ and $[D^{\tilde e_j\tilde e_i^2\la}\otimes D^{\be_{n-3}}:D^{\tilde e_1^3\nu}]=2$ for every $j$. 

By \cite[Lemma 4.4]{m1} there exists $j$ such that $\tilde e_j\tilde e_i^2\la$ has four normal nodes. From Lemmas \ref{L080620_3} and \ref{L080620_4} we have that $D_0\oplus D_1^{\oplus 3}\subseteq \End_F(D^{\tilde e_j\tilde e_i^2\la})$ and $D_0\oplus D_1\subseteq\End_F(D^{\be_{n-3}})$, so that
\[\dim\Hom_{\s_{n-3}}(\End_F(D^{\tilde e_j\tilde e_i^2\la}),\End_F(D^{\be_{n-3}}))\geq 4,\]
leading to a contradiction.

{\bf Case 2.2.2.3:} assume that $\eps_{1-i}(\tilde e_i^2\la)=1$ and $[D^{\tilde e_j\tilde e_i^2\la}\otimes D^{\be_{n-3}}:D^{\tilde e_1^3\nu}]\leq 4$ for every $j$. By \cite[Lemma 4.4]{m1} there exists $j$ such that $\tilde e_{1-j}\tilde e_i^2\la$ has three normal nodes, all of residue $j$, so by Lemma \ref{L240720_2} we have that
\[\dim\Hom_{\s_{n-3}}(\End_F(D^{\tilde e_{1-j}\tilde e_i^2\la}),\End_F(D^{\be_{n-3}}))\geq 5,\]
again giving a contradiction.
\end{proof}

\begin{theor}\label{T240720}
Suppose $n\equiv 0\Md{4}$, $\la,\nu\in\Par_2(n)\setminus\Parinv_2(n)$ with $h(\la)\leq 6$, $\eps_0(\la),\eps_1(\la)=1$ and $\nu=\dbl(n-c-d,c,d)$ with $c\equiv d\equiv 1\Md{4}$. If $E^\la\otimes E^{\be_n}_\pm\cong E^\nu$ then $E^{\tilde e_i\la}_{(\pm)}\otimes E^{\be_{n-1}}_\pm$ is irreducible for some residue $i$.
\end{theor}

\begin{proof}
We will use Lemma \ref{Lemma39} without further comment. By Lemma \ref{branchingbs} $E^{\be_n}_\pm\da_{\A_{n-1}}\cong E^{\be_{n-1}}_\pm$. Furthermore, since
\[[E^{\be_n}_\pm\da_{\A_{n-2}}]=[D^{\be_n}\da_{\A_{n-2}}]/2=[D^{\be_n}\da_{\s_{n-2}}\da_{\A_{n-2}}]/2=[D^{\be_{n-2}}\da_{\A_{n-2}}]=[E^{\be_{n-2}}],\]
we also have that $E^{\be_{n-1}}_\pm\da_{\A_{n-2}}\cong E^{\be_{n-2}}$.

In view of Lemma \ref{L090720} we may assume that $\tilde e_0\la\not\in\Parinv_2(n-1)$, that $\tilde e_1\la\not\in\Parinv_2(n-1)$ and that $E^{\tilde e_i\la}\otimes E^{\be_{n-1}}_\pm\cong E^{\tilde e_1\nu}_+\oplus E^{\tilde e_1\nu}_-$ and $E^{\tilde e_{1-i}\la}\otimes E^{\be_{n-1}}_\pm\cong e_0E^\nu$  for some $i$.

In particular
\begin{align*}
&\dim\Hom_{\s_{n-1}}(\End_F(D^{\tilde e_i\la}),\End_F(D^{\be_{n-1}}))\\
&=\dim\Hom_{\s_{n-1}}(D^{\tilde e_i\la}\otimes D^{\tilde e_i\la},D^{\be_{n-1}}\otimes D^{\be_{n-1}})\\
&=\dim\End_{\s_{n-1}}(D^{\tilde e_i\la}\otimes D^{\be_{n-1}})\\
&=\dim\Hom_{\s_{n-1}}(D^{\tilde e_i\la}\otimes (E^{\be_{n-1}}_\pm)\ua^{\s_{n-1}},D^{\tilde e_i\la}\otimes D^{\be_{n-1}})\\
&=\dim\Hom_{\s_{n-1}}((E^{\tilde e_i\la}\otimes E^{\be_{n-1}}_\pm)\ua^{\s_{n-1}},D^{\tilde e_i\la}\otimes D^{\be_{n-1}})\\
&=\dim\Hom_{\A_{n-1}}(E^{\tilde e_i\la}\otimes E^{\be_{n-1}}_\pm,(D^{\tilde e_i\la}\otimes D^{\be_{n-1}})\da_{\A_{n-1}})\\
&=\dim\Hom_{\A_{n-1}}(E^{\tilde e_i\la}\otimes E^{\be_{n-1}}_\pm,E^{\tilde e_i\la}\otimes (E^{\be_{n-1}}_+\oplus E^{\be_{n-1}}_-))\\
&=4
\end{align*}
and so $D^{\tilde e_i\la}\otimes D^{\be_{n-1}}\cong (D^{\tilde e_1\nu})^{\oplus 2}$ since
\[[D^{\tilde e_i\la}\otimes D^{\be_{n-1}}]=[(E^{\tilde e_i\la}\otimes E^{\be_{n-1}}_\pm)\ua^{\s_{n-1}}]=2[E^{\tilde e_1\nu}\ua^{\s_{n-1}}]=2[D^{\tilde e_1\nu}].\]
Further $\eps_{1-i}(\tilde e_i\la)\leq 2$ by Lemma \ref{L240720_2} and then by \cite[Lemma 4.4]{m1} the $(1-i)$-normal node of $\la$ is above the $i$-normal node of $\la$.

Since $D^{\tilde e_1\nu}$ is a composition factor of $S(n-c-d-1,c,d)$ by Lemma \ref{L280520_2} and $e_1e_0S(n-c-d-1,c,d)=0$, we also have that $D^{\tilde e_1\nu}\da_{\s_{n-3}}\cong e_0^2D^{\tilde e_1\nu}$ (as $e_1D^{\tilde e_1\nu}=0$). In view of Lemma \ref{L090620} it follows that
\[(D^{\tilde e_i\la}\otimes D^{\be_{n-1}})\da_{\s_{n-3,2}}\cong(e_0^{(2)}D^{\tilde e_1\nu}\boxtimes M^{(1^2)})^{\oplus 2}.\]

Since $\eps_i(\la)=1$, $D^{\tilde e_i\la}\da_{\s_{n-3}}$ has at most two block components (as then $e_iD^{\tilde e_i\la}=0$), and so the same holds also for $D^{\tilde e_i\la}\da_{\s_{n-3,2}}$.

{\bf Case 1:} assume that $D^{\tilde e_i\la}\da_{\s_{n-3,2}}$ has two block components, $A$ and $B$. Let $C:=D^{\be_{n-1}}\da_{\s_{n-3,2}}$. Note that
\[(A\otimes C)\oplus (B\otimes C)\cong (D^{\tilde e_i\la}\otimes D^{\be_{n-1}})\da_{\s_{n-3,2}}\cong(e_0^{(2)}D^{\tilde e_1\nu}\boxtimes M^{(1^2)})^{\oplus 2}.\]
So
\[\soc(A\otimes C)\oplus\soc(B\otimes C)\cong \soc(e_0^{(2)}D^{\tilde e_1\nu}\boxtimes M^{(1^2)})^{\oplus 2}\cong (D^{\tilde e_0^2\tilde e_1\nu}\boxtimes D^{(2)})^{\oplus 2}.\]
In particular $A\otimes C$ and $B\otimes C$ have simple socles. So, from Lemma \ref{simplesoc}, $A\otimes C,B\otimes C\subseteq e_0^{(2)}D^{\tilde e_1\nu}\boxtimes M^{(1^2)}$ and then $A\otimes C\cong B\otimes C$ by dimension. Let $\psi^A$, $\psi^B$ and $\psi^C$ be the Brauer characters of $A$, $B$ and $C$. Since $\psi^A\psi^C=\psi^B\psi^C$ and $\psi^C$ is non-zero on any 2-regular element by Lemma \ref{cbs}, we have that $\psi^A=\psi^B$, contradicting $A$ and $B$ being in different blocks.

{\bf Case 2:} assume that $D^{\tilde e_i\la}\da_{\s_{n-3,2}}$ has only one block component. Since $\eps_0(\la)=\eps_1(\la)=1$ we have that $\la_1-\la_2$ is even and so $\la_1-\la_2\geq 2$. As the $i$-normal node of $\la$ is not on the first row, $(1,\la_1)$ has residue $i-1$ and $D^{(\tilde e_i\la)_{(1,\la_1)}}$ is a composition factor of $e_{1-i}D^{\tilde e_i\la}$ and $\eps_i(\tilde e_i\la_{(1,\la_1)})>0$. In particular $e_ie_{1-i}D^{\tilde e_i\la}\not=0$ and so $e_{1-i}^2D^{\tilde e_i\la}=0$, that is $\eps_{1-i}(\tilde e_i\la)=1$. As further $\eps_i(\tilde e_i\la)=0$, $\tilde e_i\la$ is JS. So by \cite[Lemma 6.1]{m1}, since $n$ is even, $\la=(\mu,n-m)$ with $\mu\in\Par_2(m)$ having an even number of odd parts and no even part ($\mu$ has an even number of parts since $\eps_0(\la)=\eps_1(\la)=1$) and $i=1$. In particular $\tilde e_0\tilde e_1\la=\tilde e_1\tilde e_0\la$ and $\eps_1(\tilde e_0\la)=3$.

Because $\eps_1(\tilde e_1\la)=1$ and $\eps_0(\tilde e_1\la)=1$ we have that $D^{\tilde e_1\la}\da_{\s_{n-2}}\cong D^{\tilde e_0\tilde e_1\la}$ is simple. So
\begin{align*}
D^{\tilde e_0\tilde e_1\la}\da_{\A_{n-2}}\otimes E^{\be_{n-2}}&\cong (E^{\tilde e_1\la}\otimes E^{\be_{n-1}}_\pm)\da_{\A_{n-2}}\\
&\cong (E^{\tilde e_1\nu}_+\oplus E^{\tilde e_1\nu}_-)\da_{\A_{n-2}}\\
&\cong e_0(E^{\tilde e_1\nu}_+\oplus E^{\tilde e_1\nu}_-).
\end{align*}
Furthermore,
\[(e_1E^{\tilde e_0\la})\otimes E^{\be_{n-2}}\cong (E^{\tilde e_1\la}\otimes E^{\be_{n-1}}_\pm)\da_{\A_{n-2}}\cong (e_0E^\nu)\da_{\A_{n-2}}.\]
So
\[[D^{\tilde e_0\tilde e_1\la}\otimes D^{\be_{n-2}}]=[e_0D^{\tilde e_1\nu}]\hspace{11pt}\text{and}\hspace{11pt}[(e_1D^{\tilde e_0\la})\otimes D^{\be_{n-2}}]=[e_1e_0D^{\nu}]+[e_0^2D^\nu].\]
In particular
\begin{align*}
[e_1e_0D^\nu:D^{\tilde e_0\tilde e_1\nu}]&=[(e_1D^{\tilde e_0\la})\otimes D^{\be_{n-2}}:D^{\tilde e_0\tilde e_1\nu}]\\
&\geq[e_1D^{\tilde e_0\la}:D^{\tilde e_0\tilde e_1\la}]\cdot[D^{\tilde e_0\tilde e_1\la}\otimes D^{\be_{n-2}}:D^{\tilde e_0\tilde e_1\nu}]\\
&=[e_1D^{\tilde e_0\la}:D^{\tilde e_1\tilde e_0\la}]\cdot[e_0D^{\tilde e_1\nu}:D^{\tilde e_0\tilde e_1\nu}]\\
&=\eps_1(\tilde e_0\la)\eps_0(\tilde e_1\nu)\\
&=3(6-\de_{d,1}).
\end{align*}

Since $D^\nu$ is a composition factor of $S(n-c-d,c,d)$, any composition factor of $e_0D^\nu$ is a composition factor of $S(n-c-d,c-1,d)$ or $S(n-c-d,c,d-1)$. Because
\[e_1^2S(n-c-d,c-1,d)=0\hspace{11pt}\text{and}\hspace{11pt}e_1^2S(n-c-d,c,d-1)=0,\]
it follows that if $D^\mu$ is a composition factor of $e_0D^\mu$ then $\eps_1(\mu)\leq 1$. In particular if $[e_1D^\mu:D^{\tilde e_0\tilde e_1\nu}]>0$, then $\mu=\tilde f_1\tilde e_0\tilde e_1\nu=\tilde e_0\nu$. So
\[[e_1e_0D^\nu:D^{\tilde e_0\tilde e_1\nu}]=[e_0D^\nu:D^{\tilde e_0\nu}]\cdot[e_1D^{\tilde e_0\nu}:D^{\tilde e_1\tilde e_0\nu}]=\eps_0(\nu)\eps_1(\tilde e_0\nu)=4-\de_{d,1},\]
contradicting $[e_1e_0D^\nu:D^{\tilde e_0\tilde e_1\nu}]\geq 3(6-\de_{d,1})$.
\end{proof}

\begin{lemma}\label{L121120}
Suppose $p=2$, $n\equiv 0\Md{4}$, $\la,\nu\in\Par_2(n)\setminus\Parinv_2(n)$ and $E^\la\otimes E^{\be_n}_\pm\cong E^\nu$. Assume that for some residues $i$ and $j$:
\begin{itemize}
\item $\eps_i(\la)=2$, $\eps_{1-i}(\la)=0$,

\item $\tilde e_i\la\not\in\Parinv_2(n-1)$,

\item $\eps_j(\nu)\geq 4$, $\eps_{1-j}(\nu)\leq 1$,

\item $\tilde e_j\nu\not\in\Parinv_2(n-1)$,

\item $\tilde e_j^2\nu\not\in\Parinv_2(n-2)$,

\item $\tilde e_j^{\eps_j(\nu)}\nu\not\in\Parinv_2(n-\eps_j(\nu))$ or $\tilde f_j^{\phi_j(\nu)}\nu\not\in\Parinv_2(n+\phi_j(\nu))$,

\item $E^{\tilde e_i\la}\otimes E^{\be_{n-1}}_\pm$ is not irreducible.
%
\end{itemize}
Let $M\subseteq e_jD^\nu$ with $[M:D^{\tilde e_j\nu}]=2$ and $\hd(M),\soc(M)\cong D^{\tilde e_j\nu}$. Then the following hold:
\begin{itemize}
\item $E^{\tilde e_i\la}\otimes E^{\be_{n-1}}_\pm\cong M\da_{\A_{n-1}}$,

\item $(D^{\tilde e^2_i\la}\da_{\A_{n-2}})\otimes E^{\be_{n-2}}\subseteq (e_{1-j}M)\da_{\A_{n-2}}$,

\item there exists $\mu\in\Par_2(n-2)$ such that $E^\mu_{(\pm)}\subseteq (e_{1-j}M)\da_{\A_{n-2}}$ and $(E^\mu_{(\pm)})^{\oplus 2}\subseteq E^\nu\da_{\A_{n-2}}$.
%
\end{itemize}
\end{lemma}

\begin{proof}
We will use Lemma \ref{Lemma39} without further reference. From Lemma \ref{branchingbs} we have that $E^{\be_n}_\pm\da_{\A_{n-1}}\cong E^{\be_{n-1}}_\pm$ and $E^{\be_{n-1}}_\pm\da_{\A_{n-2}}\cong E^{\be_{n-2}}$ (the last isomorphism holding as in the previous lemma).

Let $D^{\tilde e_j\nu}\cong V_1\subseteq V_2\subseteq \ldots\subseteq V_{\eps_j(\nu)}=e_jD^\nu$ be as in Lemma \ref{l8}. Then $M=V_2$ by Lemma \ref{l8}. By Lemma \ref{L060820_2} we have that $\soc((e_jD^\nu)\da_{\A_{n-1}})=V_1\da_{\A_{n-1}}$ is simple. So $\soc(V_k\da_{\A_{n-1}})=V_1\da_{\A_{n-1}}$ is simple for each $k$. In the following we will simply write $V_k$ also for the corresponding submodules of $e_jE^\nu\cong (e_jD^\nu)\da_{\A_{n-1}}$. As $e_jE^\nu\cong (e_jD^\nu)\da_{\A_{n-1}}$ has simple socle isomorphic to $E^{\tilde e_j\nu}$,
\begin{align*}
[e_jE^\nu:E^{\tilde e_j\nu}]&=[e_jD^\nu:D^{\tilde e_j\nu}]=\dim\End_{\s_n}(e_jD^\nu)\leq\dim\End_{\A_n}(e_jE^\nu)\\
&\leq[e_jE^\nu:E^{\tilde e_j\nu}]
\end{align*}
and so $\dim\End_{\A_n}(e_jE^\nu)=[e_jE^\nu:E^{\tilde e_j\nu}]$. By \cite[Lemma 2.2]{gkm} we then have that if $V\subseteq e_jE^\nu$ has head isomorphic to $E^{\tilde e_j\nu}$, then $V=V_k$ where $k=[V:E^{\tilde e_j\nu}]$.

Note that $E^{\tilde e_i\la}\otimes E^{\be_{n-1}}_\pm\subseteq E^\nu\da_{\A_{n-1}}$ and so
\[\soc(E^{\tilde e_i\la}\otimes E^{\be_{n-1}}_\pm)\subseteq \soc(E^\nu\da_{\A_{n-1}})\cong \soc(e_jE^\nu)\oplus\soc(e_{1-j}E^\nu)\cong E^{\tilde e_j\nu}\oplus e_{1-j}E^\nu\]
(use Lemma \ref{L060820_2} for the socle of $e_jE^\nu$ and that $e_{1-j}D^\nu$ is simple if it is non-zero). From
\[(e_iE^\la)\otimes E^{\be_{n-1}}_\pm\cong(E^\la\otimes E^{\be_n}_\pm)\da_{\A_{n-1}}\cong (E^\nu)\da_{\A_{n-1}}\]
we have that
\[[E^{\tilde e_i\la}\otimes E^{\be_{n-1}}_\pm:E^{\tilde e_\ell\nu}_{(\pm)}]\leq\eps_\ell(\nu)/\eps_i(\la)=\eps_\ell(\nu)/2\]
and then that $[E^{\tilde e_i\la}\otimes E^{\be_{n-1}}_\pm:E^{\tilde e_j\nu}]\leq 2$ and $[E^{\tilde e_i\la}\otimes E^{\be_{n-1}}_\pm:E^{\tilde e_{1-j}\nu}_{(\pm)}]=0$ (if $\eps_{1-j}(\nu)=1$), which proves the second assertion.

So the socle, and then by self-duality also the head, of $E^{\tilde e_i\la}\otimes E^{\be_{n-1}}_\pm$ is isomorphic to $E^{\tilde e_j\nu}$. Since by assumption $E^{\tilde e_i\la}\otimes E^{\be_{n-1}}_\pm$ is not simple we then have that $E^{\tilde e_i\la}\otimes E^{\be_{n-1}}_\pm\cong V_2=M$, proving the first assertion.

Further 
\[\dim\Hom_{\s_{n-1}}(f_jD^{\tilde e_j^2\nu},e_jD^\nu)=\dim\Hom_{\s_{n-2}}(D^{\tilde e_j^2\nu},e_j^2D^\nu)=2.\]
As both $f_jD^{\tilde e_j^2\nu}$  and $e_jD^\nu$ have simple head and socle isomorphic to $D^{\tilde e_j\nu}$ and so
\[\dim\Hom_{\s_{n-1}}(f_jD^{\tilde e_j^2\nu},e_jD^\nu)>1=\dim\Hom_{\s_{n-1}}(f_jD^{\tilde e_j^2\nu},\soc(e_jD^\nu)),\]
it follows that there exists a submodule $N\subseteq e_jD^\nu$ such that $N$ has simple head and socle isomorphic to $D^{\tilde e_j\nu}$, $[N:D^{\tilde e_j\nu}]\geq 2$ (as otherwise $N=\soc(e_jD^\nu)$) and such that $N$ is isomorphic to a quotient of $f_jD^{\tilde e_j^2\nu}$. Considering all but the last property of $N$ we have that $N=V_k$ by Lemma \ref{l8}. So some $V_k$ with $k\geq 2$ is isomorphic to a quotient of $f_jD^{\tilde e_j^2\nu}$. As $V_2\subseteq V_k$ and both modules are self-dual, it also follows that $V_2$ is isomorphic to a quotient of $f_jD^{\tilde e_j^2\nu}$. Upon restriction to $A_n$ we then have that $V_2$ is isomorphic also to a quotient of $f_jE^{\tilde e_j^2\nu}$

As $V_2$ is isomorphic to $E^{\tilde e_i\la}\otimes E^{\be_{n-1}}_\pm$ as well as to a quotient of $f_jE^{\tilde e_j^2\nu}$ we then have that
\begin{align*}
&\dim\Hom_{\A_{n-2}}(E^{\tilde e_j^2\nu},E^{\tilde e_i\la}\da_{\A_{n-2}}\otimes E^{\be_{n-2}})\\
&=\dim\Hom_{\A_{n-2}}(E^{\tilde e_j^2\nu},(E^{\tilde e_i\la}\otimes E^{\be_{n-1}}_\pm)\da_{\A_{n-2}})\\
&=\dim\Hom_{\A_{n-2}}(E^{\tilde e_j^2\nu},e_j(E^{\tilde e_i\la}\otimes E^{\be_{n-1}}_\pm))\\
&=\dim\Hom_{\A_{n-1}}(f_jE^{\tilde e_j^2\nu},E^{\tilde e_i\la}\otimes E^{\be_{n-1}}_\pm)\\
&\geq \dim\End_{\A_{n-1}}(V_2)\\
&=2.
\end{align*}

Since $\soc(e_jM)\subseteq\soc(e_j^2E^\nu)\cong (E^{\tilde e_j^2\nu})^{\oplus 2}$ by Lemma \ref{L060820_2} and $\eps_i(\tilde e_i\la)=1$, so that $e_iE^{\tilde e_i\la}\cong D^{\tilde e_i^2\la}\da_{\A_{n-2}}$, we have that
\begin{align*}
\soc((D^{\tilde e_i^2\la}\da_{\A_{n-2}})\otimes E^{\be_{n-2}})&\cong (E^{\tilde e_j^2\nu})^{\oplus x}\oplus X,\\
\soc((e_{1-i}E^{\tilde e_i^2\la})\otimes E^{\be_{n-2}})&\cong (E^{\tilde e_j^2\nu})^{\oplus y}\oplus Y
\end{align*}
with $x+y\geq 2$ and $X,Y\subseteq e_{1-j}M\subseteq e_{1-j}e_jE^\nu$. From
\[(D^{\tilde e_i^2\la}\da_{\A_{n-2}})^{\oplus 2}\oplus (e_{1-i}E^{\tilde e_i^2\la})\subseteq E^\la\da_{\A_{n-2}}\]
it follows that
\[(E^{\tilde e_j^2\nu})^{\oplus 2x+y}\oplus X^{\oplus 2}\oplus Y\subseteq (E^\la\otimes E^{\be_n}_\pm)\da_{\A_{n-2}}\cong E^\nu\da_{\A_{n-2}}.\]
Since $E^\la\otimes E^{\be_n}_\pm\cong E^\nu$, comparing blocks it follows that $(E^{\tilde e_j^2\nu})^{\oplus 2x+y}\subseteq e_j^2E^\nu$ and then that $2x+y\leq 2$. So $x=0$ and then $X\not=0$, from which the last assertion follows.
\end{proof}

\begin{theor}\label{T121120_2}
Suppose $p=2$, $n\equiv 0\Md{4}$ and $\la,\nu\in\Par_2(n)\setminus\Parinv_2(n)$, and that $i,j$ are residues such that $\eps_i(\la)=2$, $\eps_{1-i}(\la)=0$, $\tilde e_i\la\not\in\Parinv_2(n-1)$, $h(\nu)=4$ and all removable nodes of $\nu$ have residue $j$. If $E^\la\otimes E^{\be_n}_\pm\cong E^\nu$ then $E^{\tilde e_i\la}\otimes E^{\be_{n-1}}_{\pm}$ is irreducible.
\end{theor}

\begin{proof}
Since $h(\nu)=4$ and all removable nodes of $\nu$ have residue $j$, we have that $\eps_j(\nu)=4$ and $\eps_{1-j}(\nu)=0$. So $e_{1-j}E^\nu=0$ by Lemma \ref{Lemma39}. By Lemma \ref{L121120} it is then enough to check that $\soc(e_{1-j}e_jE^\nu)\cong\hd(e_{1-j}e_jE^\nu)$ is multiplicity free. For $1\leq k\leq 4$ let $\nu^{(k)}:=(\nu_1,\ldots,\nu_{k-1},\nu_k-2,\nu_{k+1},\ldots,\nu_4)$. If $\nu^{(k)}\in\Par(n-2)$ let $X_k:=S^{\nu^{(k)}}$, while if $\nu^{(k)}\not\in\Par(n-2)$ let $X_k:=0$. Since $D^\nu\cong\hd(S^\nu)$, by \cite[Theorem 9.3]{JamesBook} there exist quotients $W_k$ of $X_k$ such that
\[e_{j-1}e_jD^\la\sim W_1|W_2|W_3|W_4.\]
So it is enough to check that
\[\hd((X_1\oplus X_2\oplus X_3\oplus X_4)\da_{\A_{n-2}})\]
is multiplicity free. Note that $\nu_1-\nu_2$, $\nu_2-\nu_3$ and $\nu_3-\nu_4$ are odd, since all removable nodes of $\nu$ have residue $j$. Since $h(\nu)=4$ we then have that, for $1\leq k\leq 4$, if $\nu^{(k)}\in\Par(n-2)$ then $\nu^{(k)}\in\Par_2(n-2)$.

Assume that $X_k\not=0$. Then $\hd(X_k)\cong D^{\nu^{(k)}}$ and $[X_k:D^{\nu^{(k)}}]=1$. So, since
\[\dim\Hom_{\A_{n-2}}(X_k\da_{\A_{n-2}},E^\mu_{(\pm)})=\dim\Hom_{\s_{n-2}}(X_k,E^\mu_{(\pm)}\ua^{\s_{n-2}})=0\]
for $\mu\in\Par_2(n-2)$ with $\mu\not=\nu^{(k)}$, we have that then $\hd(X_k\da_{\A_{n-2}})\cong D^{\nu^{(k)}}\da_{\A_{n-2}}$.

In particular $\hd((X_1\oplus X_2\oplus X_3\oplus X_4)\da_{\A_{n-2}})$ is multiplicity free.
\end{proof}

\begin{theor}\label{T071220}
Suppose $p=2$, $n\equiv 0\Md{4}$ and $\la,\nu\in\Par_2(n)\setminus\Parinv_2(n)$, Assume that $\nu=\dbl(c,d,n-c-d)\in\Par_2(n)$ with $c\equiv d\equiv 1\Md{4}$ and that $i$ is a residue such that $\eps_i(\la)=2$, $\eps_{1-i}(\la)=0$ and $\tilde e_i\la\not\in\Parinv_2(n-1)$. If $E^\la\otimes E^{\be_n}_\pm\cong E^\nu$ then $E^{\tilde e_i\la}\otimes E^{\be_{n-1}}_{\pm}$ is irreducible.
\end{theor}

\begin{proof}
We will use Lemmas \ref{Lemma39} and \ref{Lemma40} without further comment.

Note that $\eps_0(\nu)=4$, $\eps_1(\nu)=0$, $\tilde e_0\nu\not\in\Parinv_2(n-1)$, $\tilde e_0^2\nu\not\in\Parinv_2(n-2)$ and $\tilde e_0^4\nu\not\in\Parinv_2(n-4)$. Suppose $M\subseteq e_0D^\nu$ has head and socle isomorphic to $D^{\tilde e_0\nu}$ and $[M:D^{\tilde e_0\nu}]=2$. By Lemma \ref{L121120} it is enough to check that if $E^\mu_{(\pm)}\subseteq (e_1M)\da_{\A_{n-2}}$ then $(E^\mu_{(\pm)})^{\oplus 2}\not\subseteq E^\nu\da_{\A_{n-2}}$. Note that for any $\s_m$-module $V$ and $\pi\in\Par_2(m)$, $D^\pi\subseteq V$ if and only if $E^\pi_{(\pm)}\subseteq V\da_{\A_m}$.

In view of Lemma \ref{L280520_2} $D^\nu$ is a composition factor of $S((c,d,n-c-d))$. In particular any composition factor $D^\psi$ of $D^\nu\da_{\s_{n-2}}$ is a composition factor of some $S(\xi)$ with $h(\xi)\leq 3$ and then again by Lemma \ref{L280520_2} either $h(\psi)\leq 4$ or $\psi$ is the double of a partition with at most three parts. By Lemma \ref{L110820_2}
\begin{align*}
\tilde f_1\mu&\unrhd (\dbl(c),\overline{\dbl}(d-1),\dbl(n-c-d)),\\
\mu^B&\unlhd \dbl(c,d-1,n-c-d)
\end{align*}
where $B$ is the bottom $1$-conormal node of $\mu$. In particular
\[(\mu^B)_5+(\mu^B)_6\geq n-c-d\geq 2,\]
so $\mu_5>0$ and then $\mu=\dbl(\overline{\mu})$ for some partition $\overline{\mu}\in\Par(n-2)$ with
\[(c-1,d-1,n-c-d)\unlhd\overline{\mu}\unlhd(c,d-1,n-c-d-1).\]
Considering the four possibilities for $\overline{\mu}$ individually (two of them can be excluded by comparing blocks), we obtain that if $E^\mu_{(\pm)}\subseteq (e_1M)\da_{\A_{n-2}}$ then $\mu=\dbl(c,d-1,n-c-d-1)$.

It can be checked that $\mu\in\Parinv_2(n-2)$ and $\tilde e_0\nu=\mu^C$ with $C$ the second highest $1$-conormal node of $\mu$. So
\[[f_1E^\mu_+:E^{\tilde e_0\nu}]+[f_1E^\mu_-:E^{\tilde e_0\nu}]=[f_1D^\mu:D^{\tilde e_0\nu}]=2\]
and then $[f_1E^\mu_\pm:E^{\tilde e_0\nu}]=1$. Since $\eps_1(\nu)=0$ (so that $e_1E^\nu=0$) and $\soc(e_0E^\nu)\cong E^{\tilde e_0\nu}$ by Lemma \ref{L060820_2}, it follows that
\begin{align*}
\dim\Hom_{\A_{n-2}}(E^\mu_\pm,E^\nu\da_{\A_{n-2}})&=\dim\Hom_{\A_{n-2}}(E^\mu_\pm,e_1e_0E^\nu)\\
&=\dim\Hom_{\A_{n-1}}(f_1E^\mu_\pm,e_0E^\nu)\\
&\leq[f_1E^\mu_\pm:E^{\tilde e_0\nu}]\\
&=1.
\end{align*}
\end{proof}

\begin{theor}\label{T071220_2}
Suppose $p=2$, $n\equiv 0\Md{4}$ and $\la,\nu\in\Par_2(n)\setminus\Parinv_2(n)$, Assume that $\nu=\dbl(c,n-c-d,d)\in\Par_2(n)$ with $d\geq 5$ and $c,d\equiv 1\Md{4}$ and $i$ is a residue such that $\eps_i(\la)=2$, $\eps_{1-i}(\la)=0$ and $\tilde e_i\la\not\in\Parinv_2(n-1)$. If $E^\la\otimes E^{\be_n}_\pm\cong E^\nu$ then $E^{\tilde e_i\la}\otimes E^{\be_{n-1}}_{\pm}$ is irreducible.
\end{theor}

\begin{proof}
Similar to Theorem \ref{T071220}. In this case $\mu=\dbl(c,n-c-d-1,d-1)$.
\end{proof}

\begin{theor}\label{T071220_5}
Suppose $p=2$, $n\equiv 0\Md{4}$ and $\la,\nu\in\Par_2(n)\setminus\Parinv_2(n)$, Assume that $\nu=\dbl(n-c-d,c,d)\in\Par_2(n)$ with $c\equiv d\equiv 3\Md{4}$ and $i$ is a residue such that $\eps_i(\la)=2$, $\eps_{1-i}(\la)=0$ and $\tilde e_i\la\not\in\Parinv_2(n-1)$. If $E^\la\otimes E^{\be_n}_\pm\cong E^\nu$ then $E^{\tilde e_i\la}\otimes E^{\be_{n-1}}_{\pm}$ is irreducible.
\end{theor}

\begin{proof}
Arguing similarly to the first part of Theorem \ref{T071220} it is enough to show that there is no partition $\mu$ with $\tilde f_0\mu\unrhd\nu_A=(\dbl(c,n-c-d),\overline{\dbl}(d-1))$ and $\mu^B\unlhd\tilde e_1\nu=\dbl(n-c-d,c,d-1)$ where $A$ is the second lowest $1$-normal node of $\nu$ and $B$ is the bottom $0$-conormal node of $\mu$. As $(\mu^B)_5+(\mu^B)_6\geq d-1\geq 2$, we have that $h(\mu)\geq 5$ and so we can again additionally assume that $\mu=\dbl(\overline{\mu})$ for some partition $\overline{\mu}\in\Par(n-2)$ with three parts. Using block decomposition it can be checked that no such partition $\mu$ exists.
\end{proof}

\begin{theor}\label{T071220_4}
Suppose $p=2$, $n\equiv 0\Md{4}$ and $\la,\nu\in\Par_2(n)\setminus\Parinv_2(n)$, Assume that $\nu=\dbl(c,n-c-d,d)\in\Par_2(n)$ with $c\equiv d\equiv 3\Md{4}$ and that $i$ is a residue such that $\eps_i(\la)=2$, $\eps_{1-i}(\la)=0$ and $\tilde e_i\la\not\in\Parinv_2(n-1)$. If $E^\la\otimes E^{\be_n}_\pm\cong E^\nu$ then $E^{\tilde e_i\la}\otimes E^{\be_{n-1}}_{\pm}$ is irreducible.
\end{theor}

\begin{proof}
Let $M$ be as in Lemma \ref{L121120}. Similar to the first part of Theorem \ref{T071220} we have that if $D^\mu\subseteq e_0M$ then $\mu=\dbl(c,n-c-d-2,d)$. Let $M\subseteq e_0D^\nu$ with $\hd(M)\cong D^{\tilde e_0\nu}$ and $[M:D^{\tilde e_0\nu}]=2$. By Lemma \ref{l8} we have that $M$ is the unique submodule with these properties and that $M$ is self-dual. By self-duality of $M$ and $e_0D^\nu$, we also have that $M$ is isomorphic to the unique quotient $N$ of $e_0D^\la$ with $\soc(N)\cong D^{\tilde e_0\nu}$ and $[N:D^{\tilde e_0\nu}]=2$. It follows that $e_0D^\nu\sim M|B|M$ for some module $B$, since $[e_1D^\nu/M:D^{\tilde e_1\nu}]=2$.

By Lemma \ref{L280520_2}, $D^\nu$ is a composition factor of $S((c,n-c-d,d))$, and then 
\begin{align*}
[e_0M:D^\mu]&\leq 1/2[e_0e_1D^\nu:D^\mu]\\
&\leq1/2[e_0e_1S((c,n-c-d,d)):D^\mu]\\
&=[S((c,n-c-d-2,d)):D^\mu]\\
&=1.
\end{align*}
Since all composition factors of the head and socle of $e_0M$ are isomorphic to $D^\mu$ and $[e_0M:D^\mu]\leq 1$, we have that $e_0M\cong D^\mu$ if $e_0M$ is non-zero. So $(D^{\dbl(c,n-c-d-1,d-1)}\otimes D^{\be_{n-2}})\da_{\A_{n-2}}\subseteq D^\mu\da_{\A_{n-2}}$ by Lemma \ref{L121120} and then $D^{\dbl(c,n-c-d-1,d-1)}\otimes D^{\be_{n-2}}\cong D^\mu$. This however contradicts \cite[Theorem 1.1]{m1}, since $h(\dbl(c,n-c-d-1,d-1))\geq 5$.
\end{proof}

\begin{theor}\label{T071220_3}
Suppose $p=2$, $n\equiv 0\Md{4}$ and $\la,\nu\in\Par_2(n)\setminus\Parinv_2(n)$. Assume that $\nu=\dbl(n-c-d,c,d)\in\Par_2(n)$ with $d\geq 5$ and $c\equiv d\equiv 1\Md{4}$ and $i$ is a residue such that $\eps_i(\la)=2$, $\eps_{1-i}(\la)=0$ and $\tilde e_i\la\not\in\Parinv_2(n-1)$. If $E^\la\otimes E^{\be_n}_\pm\cong E^\nu$ then $E^{\tilde e_i\la}\otimes E^{\be_{n-1}}_{\pm}$ is irreducible.
\end{theor}

\begin{proof}
We will use Lemma \ref{Lemma39} without further comment.

From Lemma \ref{L080620} if $D^\mu\subseteq e_1e_0D^\nu$ then $\mu$ is the double of a partition provided $h(\mu)\geq 5$. So by Lemma \ref{L110820_2} if $D^\mu\subseteq e_1e_0D^\nu$ then
\[\mu=\dbl(n-c-d-1,c,d-1)=\tilde e_1\tilde e_0\nu=\tilde e_0\tilde e_1\nu\]
(the partition $\psi=\dbl(n-c-d-1,c-1,d)$ is excluded from $\eps_1(\nu)=\eps_1(\psi)+1$ and $\tilde e_1\tilde e_0\nu\not=\psi$).

Let $\pi=\dbl(n-c-d,c,d-1)=\tilde e_0\nu=\tilde f_1\mu$. Since $\pi\not\in\Parinv_2(n-1)$, $\tilde e_0^{(\eps_0(\nu))}\nu\not\in\Parinv_2(n-\eps_0(\nu))$ and $\mu\in\Parinv_2(n-2)$ we have that $E^\pi\ua^{\s_{n-1}}\not\subseteq e_0D^\nu$ and $E^\pi\ua^{\s_{n-1}}\subseteq f_1D^\mu$ by Lemmas \ref{L060820_2} and \ref{L060820a_2}. 

For $1\leq k\leq\eps_0(\nu)$ (resp. $1\leq k\leq \phi_1(\mu)$) let $V_k$ (resp. $W_k$) be the unique submodule of $e_0D^\nu$ (resp. the unique quotient of $f_1D^\mu$) with head and socle isomorphic to $D^\pi$ and $[V_k:D^\pi]=k$ (resp. $[W_k:D^\pi]=k)$, see Lemma \ref{l8} and the dual version of Lemma \ref{l9}. From above we have that $V_2\not\cong W_2$ (by duality of $f_1D^\mu$).

If
\[\dim\Hom_{\s_{n-2}}(D^\mu,e_1e_0D^\nu)=\dim\Hom_{\s_{n-1}}(f_1D^\mu,e_0D^\nu)\geq 2\]
then there exists $k\geq 2$ such that $V_k\cong W_k\not=0$. Because $V_\ell\subseteq V_k$ and $W_\ell\subseteq W_k$ for $\ell\leq k$ (by Lemmas \ref{l8} and \ref{l9} and duality) it follows by uniqueness of $V_2$ and $W_2$ that $V_2\cong W_2$, leading to a contradiction. So $\dim\Hom_{\s_{n-2}}(D^\mu,e_1e_0D^\nu)= 1$, as $\mu=\tilde e_1\tilde e_0\nu$.

Since
\[\dim\Hom_{\A_{n-2}}(E^\mu_\pm,e_1e_0E^\nu)=\dim\Hom_{\s_{n-2}}(D^\mu,e_1e_0D^\nu)\]
we then have that $\soc(e_1e_0E^\nu)\cong E^\mu_+\oplus E^\mu_-$. Further by Lemma \ref{L060820a}
\[\soc(e_0e_1E^\nu)\cong \soc(e_0(E^{\tilde e_1\nu}_+\oplus E^{\tilde e_1\nu}_-))\cong E^\mu_+\oplus E^\mu_-.\]

In view of Lemma \ref{L121120} we have that $A:=(D^{\tilde e_i^2\la})\da_{\A_{n-2}}\otimes E^{\be_{n-2}}\subseteq e_1V_2$. Since $A^{\oplus 2}$ is a direct summand of $(E^\la\otimes E^{\be_n}_\pm)\da_{\A_{n-2}}\cong E^\nu\da_{\A_{n-2}}$, it follows by comparing blocks that $A^{\oplus 2}$ is isomorphic to a direct summand of $B:=e_1e_0E^\nu\oplus e_0e_1E^\nu$. Because $\soc(B)\subseteq (D^\mu\da_{\A_{n-2}})^{\oplus 2}$ and $A$ is fixed up to conjugation with $\s_{n-2}$, it follows that $\soc(A^{\oplus 2})\cong (D^\mu\da_{\A_{n-2}})^{\oplus 2}\cong \soc(B)$. Since $A^{\oplus 2}$ is isomorphic to a direct summand of $B$ and $A^{\oplus 2}$ and $B$ have isomorphic socles, we have that $A^{\oplus 2}\cong B$. 

Because $[e_0D^\nu/V_2:D^\pi]=2$, we have by the dual version of Lemma \ref{l8} that $e_0D^\nu\sim V_2|C|V_2$ for a certain module $C$. This leads to a contradiction comparing dimensions, since $A\subseteq e_1V_2$ and $A^{\oplus 2}\cong B=e_1e_0E^\nu\oplus e_0e_1E^\nu$ with $e_0e_1E^\nu\not=0$.
\end{proof}

\begin{theor}\label{T071220_6}
Suppose $p=2$, $n\equiv 0\Md{4}$ and $\la,\nu\in\Par_2(n)\setminus\Parinv_2(n)$, Assume that $\nu=\dbl(c,d,n-c-d)\in\Par_2(n)$ with $c\equiv d\equiv 3\Md{4}$ and that $i$ is a residue such that $\eps_i(\la)=2$, $\eps_{1-i}(\la)=0$ and $\tilde e_i\la\not\in\Parinv_2(n-1)$. Then $E^\la\otimes E^{\be_n}_\pm\not\cong E^\nu$.
\end{theor}

\begin{proof}
We will use Lemma \ref{Lemma39} without further comment. Assume for a contradiction that $E^\la\otimes E^{\be_n}_\pm\cong E^\nu$.

Let $\mu\in\Par_2(n-2)$ with $D^\mu\subseteq e_0e_1D^\nu$. Then the 2-core of $\mu$ is $(4,3,2,1)$.

By Lemma \ref{L110820_2}
\[\mu^A\unlhd\dbl(c,d,n-c-d-1)\hspace{22pt}\text{and}\hspace{22pt}\mu^B\unrhd(\overline{\dbl}(c-1),\dbl(d,n-c-d))\]
where $A$ and $B$ are the bottom and top $0$-conormal nodes of $\mu$ respectively. Furthermore, $D^\mu$ is a composition factor of $S((c,d,n-c-d-2))$ since $D^\nu$ is a composition factor of $S((c,d,n-c-d))$. In particular $h(\mu)\leq 6$ and if $h(\mu)\geq 5$ then $\mu$ is the double of a partition by Lemma \ref{L080620}. If $h(\mu)\leq 4$ then $h(\mu)=4$ and $\mu=(\mu_1,\mu_2,\mu_3,\mu_4)$ with $\mu_1$ and $\mu_3$ even and $\mu_2$ and $\mu_4$ odd (considering the core). In this case
\begin{align*}
(\mu_1,\mu_2,\mu_3,\mu_4,1)\unlhd\mu^A&\unlhd\dbl(c,d,n-c-d-1)
\end{align*}
and then $n-c-d=2$. As $(c,d,n-c-d-2)$ has only 2 parts in this case, $\mu$ is also in this case the double of a partition (again by Lemma \ref{L080620}).

Comparing cores we then have that $\mu=\dbl(c,d,n-c-d-2)$. 

Further since $D^\nu$ is a composition factor of $S((c,d,n-c-d))$
\begin{align*}
[e_0e_1E^\nu:E^\mu_\pm]&=[e_0e_1D^\nu:D^\mu]\\
&\leq[e_0e_1S((c,d,n-c-d)):D^\mu]\\
&=(1+\de_{n-c-d>2})[S((c,d,n-c-d-2)):D^\mu]\\
&=1+\de_{n-c-d>2}.
\end{align*}

In particular $E^\nu\da_{\A_{n-2}}\cong e_1^2E^\nu\oplus e_0e_1E^\nu$ and, by Lemma \ref{L060820_2} and the above,
\[\soc(e_1^2E^\nu)\cong (E^{\tilde e_1^2\nu})^{\oplus 2}\hspace{22pt}\text{and}\hspace{22pt}\soc(e_0e_1E^\nu)\cong (E^\mu_+\oplus E^\mu_-)^{\oplus x}\]
with $x\leq 1+\de_{n-c-d>2}$.

By Lemma \ref{branchingbs} $E^{\be_n}_\pm\da_{\A_{n-2}}\cong E^{\be_{n-2}}$ and so $E^\la\da_{\A_{n-2}}\otimes E^{\be_{n-2}}\cong E^\nu\da_{\A_{n-2}}$.

Assume that $\tilde e_i\la$ is JS. Then, since $n$ is even, all parts of $\tilde e_i\la$ are odd and $h(\tilde e_i\la)$ is also odd. So the unique normal node of $\tilde e_i\la$ has residue 0 while the two conormal nodes of $\tilde e_i\la$ have residue 1. This however contradicts the assumption that $\eps_i(\la)=2$, so that $\eps_i(\tilde e_i\la)=1$ and $\phi_i(\tilde e_i\la)\geq 1$. Thus $\tilde e_i\la$ is not JS, that is $\eps_{1-i}(\tilde e_i\la)\geq 1$ (as again $\eps_i(\tilde e_i\la)=1$).

For any residue $k$ let $W_k:=e_ke_iE^\la\otimes E^{\be_{n-2}}$. Since $\eps_k(\tilde e_i\la)\geq 1$ we have that $e_ke_iE^\la\not=0$ and so $W_k\not=0$. Furthermore, $W_k$ is fixed under conjugating by elements of $\s_{n-2}\setminus\A_{n-2}$ since both $E^\la$ and $E^{\be_{n-2}}$ are. The same then holds also for its socle. Further for $k=i$ we have that
\[W_i=e_i^2E^\la\otimes E^{\be_{n-2}}\cong (e_i^2D^\la)\da_{\A_{n-2}}\otimes E^{\be_{n-2}}\cong ((e_i^{(2)}D^\la)\da_{\A_{n-2}}\otimes E^{\be_{n-2}})^{\oplus 2},\]
so $W_i\cong \overline{W}_i^{\oplus 2}$ with $\overline{W}_i:=(e_i^{(2)}D^\la)\da_{\A_{n-2}}\otimes E^{\be_{n-2}}$. Also the socle of $\overline{W}_i$ is fixed by conjugation with $\s_{n-2}$.

From $E^\nu\da_{\A_{n-2}}\cong \overline{W}_i^{\oplus 2}\oplus W_{1-i}$ and
\[\soc(E^\nu\da_{\A_{n-2}})\cong (E^{\tilde e_1^2\nu})^{\oplus 2}\oplus (E^\mu_+\oplus E^\mu_-)^{\oplus x}\]
with $x\leq 2$ we then have that, for some $k$,
\begin{align*}
\soc(W_k)&\cong (E^{\tilde e_1^2\nu})^{\oplus 2}\cong\soc(e_1^2E^\nu),\\
\soc(W_{1-k})&\cong(E^\mu_+\oplus E^\mu_-)^{\oplus x}\cong\soc(e_0e_1E^\nu).
\end{align*}

From $\eps_{1-k}(\tilde e_i\la)\geq 1$ we have that $D^{\tilde e_{1-k}\tilde e_i\la}\da_{\A_{n-2}}\otimes E^{\be_{n-2}}\subseteq W_{1-k}$. In particular the head and socle of $D^{\tilde e_{1-k}\tilde e_i\la}\da_{\A_{n-2}}\otimes E^{\be_{n-2}}$ are both isomorphic to $E^\mu_+\oplus E^\mu_-$. Furthermore,
\begin{align*}
[D^{\tilde e_{1-k}\tilde e_i\la}\da_{\A_{n-2}}\otimes E^{\be_{n-2}}:E^\mu_\pm]&\leq [E^\nu\da_{\A_{n-2}}:E^\mu_\pm]/[D^\la\da_{\s_{n-2}}:D^{\tilde e_{1-k}\tilde e_i\la}]\\
&\leq [E^\nu\da_{\A_{n-2}}:E^\mu_\pm]/(\eps_i(\la)\eps_{1-k}(\tilde e_i\la))\\
&\leq (1+\de_{n-c-d>2})/2.
\end{align*}
So $n-c-d>2$ and $D^{\tilde e_{1-k}\tilde e_i\la}\da_{\A_{n-2}}\otimes E^{\be_{n-2}}\cong E^\mu_+\oplus E^\mu_-$. It follows that $D^{\tilde e_{1-k}\tilde e_i\la}\otimes D^{\be_{n-2}}\cong D^\mu$. This leads to a contradiction by \cite[Theorem 1.1]{m1} since $h(\mu)>4$.
\end{proof}

\section{Proof of Theorem \ref{MT}}\label{s10}

Assume that $V\otimes W$ is irreducible and let $\nu\in\Par_2(n)$ be such that $V\otimes W\cong E^\nu_{(\pm)}$, that is $V\otimes W$ is isomorphic to either $E^\nu$ or $E^\nu_\pm$. By Theorem \ref{T280520} we have that $\nu\not\in\Parinv_2(n)$, so $V\otimes W\cong E^\nu$ and then $V\otimes W\cong V^\si\otimes W^\si$ for $\si\in\s_n\setminus\A_n$. For $n\leq 4$ the module $W$ is 1-dimensional. For $n=5$ the theorem can be checked using decomposition matrices and character tables \cite{GAP} (in this case $W$ has dimension 2). So we will assume $n\geq 6$ from now on.

{\bf Case 1:} assume that $n\equiv 2\Md{4}$ and $V\cong E^\la$. Then 
\[(D^\la\otimes D^\be_n)\da_{\A_n}\cong E^\la\otimes E^{\be_n}\cong V\otimes W\cong E^\nu\cong D^\nu\da_{\A_n}.\]
In particular $D^\la\otimes D^{\be_n}$ is irreducible and it does not split when restricted to $\A_n$. This leads to a contradiction due to \cite[Theorem 1.1]{m1}, since by assumption neither $V$ nor $W$ (and thus also neither $D^\la$ nor $D^{\be}$) is 1-dimensional.

{\bf Case 2:} assume that $n\equiv 2\Md{4}$ and $V\cong E^\la_\pm$. From the above remarks $E^\la_\pm\otimes E^\be_n\cong E^\nu\cong E^\la_\mp\otimes E^{\be_n}$. Since $\psi^{\be_n}_\al\not=0$ for any $\al\in\Paro(n)$ by Lemma \ref{cbs}, it follows that $\psi^{\la,+}=\psi^\nu/\psi^{\be_n}=\psi^{\la,-}$, which also leads to a contradiction.

{\bf Case 3:} assume that $n\not\equiv 2\Md{4}$ and $V\cong E^\la_\pm$. By Theorem \ref{T080620} we may assume that $V$ is neither basic spin nor second basic spin. Further by Theorem \ref{T050620_3} we may assume that $\la\not=\dbl(n-a,a)$ with $3\leq a\leq n/2-3$ odd if $n\equiv 0\Md{4}$. This gives a contradiction thanks to Theorem \ref{T050620}.

{\bf Case 4:} assume that $n$ is odd and $V\cong E^\la$. Then Theorem \ref{T080620_3} leads to a contradiction.

{\bf Case 5:} assume that $n\equiv 0\Md 4$ and $V\cong E^\la$. From Theorem \ref{T050620_2}, $h(\la)\leq 6$ and $D^\nu$ is a composition factor of some module of the form $S((n-a,a))$ with $a<n/2$ or $S(p(n-a-b,a,b))$ with $a\equiv b\equiv \pm 1\Md{4}$.

{\bf Case 5.1:} assume that $D^\nu$ is a composition factor of $S((n-a,a))$ with $a<n/2$. Since $\nu\not\in\Parinv_2(n)$, by Lemma \ref{L280520_2} we then have that $\nu=(n-b,b)$ with $n-2b\geq 4$ or $\nu=\dbl(n-b,b)$ with $n-b,b\equiv 2\Md{4}$. In either case $\nu$ is JS and if the normal node of $\nu$ has residue $i$ then $\tilde e_i\nu\not\in\Parinv_2(n-1)$. So by Lemma \ref{Lemma39}
\[E^\la\da_{\A_{n-1}}\otimes E^{\be_{n-1}}_\pm\cong (E^\la\otimes E^{\be_n}_\pm)\da_{\A_{n-1}}\cong E^\nu\da_{\A_{n-1}}\cong E^{\tilde e_i\nu}\]
is irreducible, leading to a contradiction due to the above $n$ odd case (note that neither $E^{\be_{n-1}}_\pm\cong E^{\be_n}_\pm\da_{\A_{n-1}}$ nor $E^\la\da_{\A_{n-1}}$ is 1-dimensional).

{\bf Case 5.2:} assume that $D^\nu$ is a composition factor of $S(p(n-a-b,a,b))$ with $a\equiv b\equiv \pm 1\Md{4}$. From Lemma \ref{L080620_2}, $S(\la)$ is in the same block of $D^\nu$ (since from $E^\la\otimes E^{\be_n}_\pm$ it follows that $[D^\la\otimes D^{\be_n}]=2[D^\nu]$). In view of Lemma \ref{L080620}, it can then be checked that, for some $x\geq 0$, $\la$ has $x+2$ parts $\la_j\equiv a\Md{4}$ and $x$ parts $\la_j\equiv -a\Md{4}$. It follows (since $n\equiv 0\Md{4}$) that $\la$ has an odd number of parts $\la_j\equiv 2\Md{4}$. Since $\la$ has both even and odd parts, it is not JS. From \cite[Theorem 13.2]{m4} we may assume that $\la$ has two or three normal nodes and from Theorem \ref{T050620_2} that $h(\la)\leq 6$.

Since $h(\la)\geq 3$ we have that $\tilde e_i\la\not=(n-1)$ if $\eps_i(\la)>0$. So, by the $n$ odd case, $E^{\tilde e_i\la}_{(\pm)}\otimes E^{\be_{n-1}}_\pm$ is not irreducible. If $\la$ has 3 normal nodes then $E^\la\otimes E^{\be_n}_\pm$ is not irreducible by Theorem \ref{T100720}. If $\la$ has two normal nodes of different residues then we are in case (iv) of Lemma \ref{L090720} and so $E^\la\otimes E^{\be_n}_\pm$ is not irreducible by Theorem \ref{T240720}. If $\la$ has two normal nodes of the same residue then we are in cases (ii) or (iii) of Lemma \ref{L090720} and $E^\la\otimes E^{\be_n}_\pm$ is not irreducible by Theorems \ref{T121120_2} to \ref{T071220_6}.

\section{Proof of Theorem \ref{T150620}}\label{s11}

The characteristic 0 case is covered by \cite[Theorem 5.6]{bk3} and implicitly by \cite{z1}.

If $p$ is not 2, 3 or 5 this is \cite[Main Theorem]{bk2}, up to using \cite[Lemma 8.1]{m2} to be able to remove the assumption $h(\la)\not\equiv 0\Md{p}$ and \cite[l. 17-23 on p. 28]{m4} to compare the two different partitions labeling $V\otimes W$ (the two different labelings in characteristic $p\neq 2$ are due to the non-trivial isomorphism $D^\nu\da_{\A_n}\cong D^\nu\otimes \sgn\da_{\A_n}$ for $\nu\in\Par_p(n)\setminus\Parinv_p(n)$). If $p=5$ the result is \cite[Theorem 1.1]{m2} (again using \cite[l. 17-23 on p. 28]{m4}), while if $p=3$ or if $p=2$ and neither of $V$ and $W$ is basic spin then the result is \cite[Theorem 1.1]{m4}. If $p=2$ and $V$ or $W$ is basic spin the result holds by Theorem \ref{MT}.

\section*{Acknowledgments}

The author thanks Matt Fayers and the referee for comments on the paper.

\end{document}